\documentclass[titlepage,a4paper, 11pt, DIV=12]{scrartcl}

\usepackage[top=3cm, bottom=3cm, left=3cm, right=3cm]{geometry}

\usepackage[utf8]{inputenc}
\usepackage[T1]{fontenc}
\usepackage[british]{babel}
\usepackage{graphicx, xcolor}
\usepackage{booktabs}
\usepackage[shortlabels]{enumitem}
\setlist[enumerate,1]{label=(\roman*),font=\itshape}
\usepackage{amsthm}
\usepackage[normalem]{ulem}

\usepackage{relsize}
\usepackage{bbm}
\usepackage{dsfont}

\usepackage{newtxtext,newtxmath}
\usepackage{silence}
\usepackage[colorlinks=true, linkcolor=blue, filecolor=magenta, urlcolor=cyan, citecolor=gray, unicode]{hyperref}
\usepackage[capitalize]{cleveref}
\crefname{equation}{}{}
\crefname{enumi}{}{}

\usepackage{mathtools}
\DeclarePairedDelimiter\ceil{\lceil}{\rceil}
\DeclarePairedDelimiter\floor{\lfloor}{\rfloor}
\numberwithin{equation}{section}

\theoremstyle{definition}

\crefname{question}{Question}{Questions}
\newtheorem{definition}{Definition}[section]
\crefname{definition}{Definition}{Definitions}

\crefname{ex}{Example}{Examples}

\theoremstyle{plain}
\newtheorem{thm}[definition]{Theorem}
\crefname{thm}{Theorem}{Theorems}
\newtheorem{conj}{Conjecture}[section]
\crefname{conj}{Conjecture}{Conjectures}
\newtheorem{prob}[conj]{Problem}
\crefname{prob}{Problem}{Problems}
\newtheorem{lemma}[definition]{Lemma}
\crefname{lemma}{Lemma}{Lemmas}
\newtheorem{cor}[definition]{Corollary}
\crefname{cor}{Corollary}{Corollaries}
\newtheorem{prop}[definition]{Proposition}
\crefname{prop}{Proposition}{Propositions}
\newtheorem{obs}[definition]{Observation}
\crefname{obs}{Observation}{Observations}
\newtheorem{claim}[definition]{Claim}
\crefname{claim}{Claim}{Claims}
\newtheorem{step}{Step}
\crefname{step}{Step}{Steps}

\crefname{property}{Property}{Property}
\newtheorem*{clm*}{Claim}

\theoremstyle{remark}

\newcommand{\oldqed}{}
\newcommand{\qedClaim}{\hfill$\blacksquare$}
\newenvironment{claimproof}[1][Proof of Claim]{
        \renewcommand{\oldqed}{\qedsymbol}
	\renewcommand{\qedsymbol}{\qedClaim}
	\begin{proof}[\underline{#1}]
}{
	\end{proof}
	\renewcommand{\qedsymbol}{\oldqed}
} 

\newcommand{\Pro}[1]{\mathbb{P} \left[#1\right]}
\newcommand{\Exp}[1]{\mathbb{E} \left[ #1 \right]}

\def\EE{\mathbb E}
\def\PP{\mathbb P}

\newcommand{\str}[1]{\ifmmode\text{\sout{\ensuremath{#1}}}\else\sout{#1}\fi}

\newcommand{\R}{\mathbb{R}}

\newcommand{\Z}{\mathbb{Z}}
\newcommand{\N}{\mathbb{N}}
\newcommand{\1}{\mathbbm{1}}

\def\NN{\mathbb N}

\newcommand{\e}{\varepsilon}
\newcommand{\eps}{\varepsilon}

\newcommand{\cF}{\mathcal{F}}

\newcommand{\cM}{\mathcal{M}}

\newcommand{\cT}{\mathcal{T}}
\newcommand{\cU}{\mathcal{U}}

\newcommand{\cV}{\mathcal{V}}
\newcommand{\pzc}[1]{\mathcal{#1}}

\newcommand{\uu}{\underline{u}}
\newcommand{\uw}{\underline{w}}
\newcommand{\uv}{\underline{v}}
\newcommand{\logb}[1]{\log\left(#1\right)}
\newcommand{\noe}{\overline{e}}

\newcommand{\card}[1]{\left| #1 \right|}

\DeclareMathOperator{\tr}{Tr}
\DeclareMathOperator{\rg}{rg}

\usepackage{scalerel}

\newsavebox{\abstractbox}
\renewenvironment{abstract}
{\begin{lrbox}{0}\begin{minipage}{\textwidth}
 \begin{center} \vspace{6mm}\normalfont\sectfont\abstractname\end{center}\quotation}
 {\endquotation\end{minipage}\end{lrbox}%
 \global\setbox\abstractbox=\box0 }

\usepackage{etoolbox}
\makeatletter
\expandafter\patchcmd\csname\string\maketitle\endcsname
  {\vskip\z@\@plus3fill}
  {\vskip\z@\@plus2fill\box\abstractbox\vskip\z@\@plus8fill}
  {}{}
\makeatother

\author{Peter Allen 
\and Julia B\"{o}ttcher  
\and Jan Corsten 
\and Ewan Davies 
\and Matthew Jenssen\footnote{Research supported by a UK Research and Innovation Future Leaders Fellowship MR/W007320/1. } 
\and Patrick Morris\footnote{Research supported by a Deutsche Forschungsgemeinschaft (DFG, German Research
Foundation) Walter Benjamin Fellowship (Project Number: 504502205).} \and Barnaby Roberts
\and Jozef Skokan}

\publishers{ \small PA, JB, JS: \\ 
 Department of Mathematics, London School of Economics and Political Science (LSE), Houghton Street, London, WC2A 2AE, United Kingdom. \\
\texttt{\{p.d.allen|j.boettcher|j.skokan\}@lse.ac.uk}\\

JC, BR: \\
 Department of Mathematics, London School of Economics and Political Science (LSE), Houghton Street, London, WC2A 2AE, United Kingdom. \\
\texttt{\{jan.corsten92|roberts.barnaby\}@gmail.com}\\

 ED: \\
 Department of Computer Science, Colorado State University, Fort Collins, Colorado 80523, USA. \\
\texttt{ewan.davies@colostate.edu}\\

MJ:\\
Department of Mathematics, King’s College London, Strand, London, WC2R 2LS, United Kingdom. \\
\texttt{matthew.jenssen@kcl.ac.uk}\\

 PM: \\ 
 Department of Mathematics, Universitat Polit\`ecnica de Catalunya (UPC), Barcelona, 08034,  Spain. \\
\texttt{pmorrismaths@gmail.com}

}

\date{\today}
\title{\LARGE A robust Corr\'adi--Hajnal Theorem}

\begin{document}

\begin{abstract}
\vspace{-4mm}
    For a graph $G$ and $p\in[0,1]$,  we denote by   $G_p$  the random sparsification of $G$ obtained by keeping each edge of $G$ independently, with probability $p$.  We show that there exists a $C>0$ such that if $p\geq C(\log n)^{1/3}n^{-2/3}$ and $G$ is an $n$-vertex graph with $n\in 3\N$ and $\delta(G)\geq \tfrac{2n}{3}$, then with high probability $G_p$ contains a triangle factor. Both the minimum degree condition and the probability condition, up to the choice of $C$, are tight. Our result can be viewed as a common strengthening of the seminal theorems of Corr\'adi and Hajnal, which deals with the extremal minimum degree condition for containing triangle factors (corresponding to $p=1$ in our result), and Johansson, Kahn and Vu, which deals with the threshold for the appearance of a triangle factor in $G(n,p)$ (corresponding to $G=K_n$ in our result). It also implies a lower bound on the number of triangle factors in graphs with minimum degree at least $\tfrac{2n}{3}$ which gets close to the truth.
    \vspace{6mm}
\end{abstract}

\maketitle

\section{Introduction} \label{sec:intro}

As a natural generalisation of perfect matchings in graphs, triangle factors
are a fundamental object in graph theory with a wealth of results studying their
appearance.  Here, a \emph{triangle factor} in a graph~$G$ is a collection of
vertex-disjoint triangles which completely cover the vertex set of $G$.  Note
that for a graph $G$ to contain a triangle factor, the number of vertices of $G$
must be divisible by 3. In extremal graph theory, a fundamental result is the
well-known theorem of Corr\'adi and Hajnal~\cite{Corradi1963}, which determines
the smallest \emph{minimum degree}~$\delta(G)$ guaranteeing the existence of a
triangle factor.

\begin{thm}[Corr\'{a}di, Hajnal~\cite{Corradi1963}]\label{thm:CorradiHajnal}
  Any $n$-vertex graph~$G$ with $n\in 3 \N$ and $\delta(G) \geq \tfrac{2n}{3}$
  contains a triangle factor.
\end{thm}

A breakthrough by Johansson, Kahn and Vu~\cite{Johansson2008} in probabilistic
graph theory, on the other hand, established the threshold for the
\emph{binomial random graph} $G(n,p)$ to contain a triangle-factor. Here,
$G(n,p)$ is obtained by including each possible edge among~$n$ vertices
independently at random with probability $p=p(n)$, and $p^*(n)$ is a
\emph{threshold} for a graph property~$P$ if the probability that $G(n,p)$
has~$P$ tends to~$0$ as~$n$ tends to infinity whenever $p(n)/p^*(n)\to 0$ and
to~$1$ whenever $p^*(n)/p(n)\to 0$. Johansson, Kahn and Vu~\cite{Johansson2008}
showed that the threshold for the appearance of a triangle-factor is $(\log
n)^{1/3}n^{-2/3}$.

In this paper we are interested in a combination of these two results, giving a
so-called \emph{robustness version} of the Corr\'adi--Hajnal Theorem.  More
precisely, we consider graphs~$G$ satisfying a minimum degree condition and ask
for which~$p$ their \emph{random sparsification}~$G_p$, which is obtained by
keeping every edge of~$G$ independently with probability $p$, contains a
triangle-factor. Such a robustness result follows already from the sparse blow-up
lemma~\cite[Theorem~1.11]{AllenBoettcherHanKohayakawaPerson-blowup-sparse}: For
every $\gamma>0$ and $p\ge C(\frac{\log n}{n})^{1/2}$ any $n$-vertex graph~$G$
with minimum degree $\delta(G) \geq(\tfrac{2}{3}+\gamma)n$ satisfies that~$G_p$
has a triangle factor whp. Here, we say a property holds \emph{with high
probability}, abbreviated \emph{whp}, if the probability it holds tends to~$1$
as~$n$ tends to infinity.

Turning this into an exact result (in terms of the minimum degree condition)
requires more work, and moving to smaller probabilities~$p$ is substantially
harder. Here we achieve both, showing that graphs~$G$ satisfying the properties
of the Corr\'adi--Hajnal Theorem are strongly robust for triangle factors: $G_p$
retains a triangle factor all the way down to the threshold probability~$p$
for triangle factors. Hence, our result is a common strengthening of two
cornerstone theorems in extremal and probabilistic graph theory, implying that
both the minimum degree condition and the condition on the probability are tight.

\begin{thm}[main result]\label{thm:main}
  There is $C >0$ such that for all $n \in 3\N$ and $p \geq C (\log n)^{1/3}n^{-2/3}$ the following holds. If  $G$ is an $n$-vertex graph with $\delta(G) \geq \tfrac{2n}{3}$ then whp $G_p$ has a triangle factor. 
\end{thm}

Our proof of \cref{thm:main} builds on an alternative proof of the threshold for
triangle factors in $G(n,p)$ due to Kohayakawa and a subset of the
authors~\cite{Allen2020+}. This proof in turn shares some of the key ideas with
that of Johansson, Kahn and Vu~\cite{Johansson2008} (as well as
\cite{Kahn2019,kahn2020hitting}), in particular the use of entropy, but follows
a different scheme of `building' our triangle factor one triangle at a
time. This scheme provides the opportunity for us to strengthen the proof to
deal with incomplete graphs~$G$. We defer a detailed discussion of our proof to
\cref{sec:proofoverview}.

As a corollary to \cref{thm:main}, we can provide a lower bound on the number of
triangle factors in every graph $G$ with $\delta(G)\geq \tfrac{2n}{3}$.

\begin{cor} \label{cor:triangle factor count}
There is $c>0$ such that any graph~$G$ on~$n\in 3\N$ vertices with $\delta(G)\geq \tfrac{2n}{3}$ contains at least 
\[\bigg(\frac{c n}{(\log n)^{1/2}}\bigg)^{2n/3}\]
triangle factors. 
\end{cor}

\cref{cor:triangle factor count} follows easily from \cref{thm:main} by
considering the expected number of triangle factors in $G_p$ and the fact that
each triangle factor survives in $G_p$ with probability $p^n$. Indeed, for a
graph~$F$ let $T(F)$ denote the number of triangle factors
in~$F$. \cref{thm:main} implies $\Pro{T(G_p)\ge 1}\ge\frac12$ for $p \geq C
(\log n)^{1/3}n^{-2/3}$, for~$G$ as in \cref{cor:triangle factor count}, and
for~$n$ sufficiently large. Since further $\Exp{T(G_p)}=T(G)\cdot p^n$ we get
\[
\frac12\le\Pro{T(G_p)\ge 1}\le\Exp{T(G_p)}=T(G)\cdot  \bigg(C\frac{(\log n)^{1/3}}{n^{2/3}}\bigg)^n\,,
\]
implying \cref{cor:triangle factor count} for~$c$ sufficiently small.

To our knowledge, \cref{cor:triangle factor count} is the first of its kind and
it gets close to the truth. Indeed, letting $n\in 3 \N$ and $H=G(n,q)$ be the
binomial random graph with $q=\tfrac{2}{3}+o(1)$, we have that whp $H$ has
minimum degree at least $\tfrac{2n}{3}$ and the expected number of triangle
factors in $H$ is
\[\frac{q^nn!}{\left(n/3\right)!6^{n/3}}=\left((1+o(1))\frac{2}{e(\sqrt{3})^3}n\right)^{2n/3}.\]
It is believable that every graph as in \cref{cor:triangle factor count} has at
least this many triangle factors. As a first step, removing the $(\log n)^{1/2}$
from the expression in \cref{cor:triangle factor count} poses an interesting
open problem.

\paragraph{\textbf{Related work: Hamiltonicity.}}

To put our work into context, let us briefly discuss robustness results with
respect to another graph property, where these types of questions have been
explored extensively. A \emph{Hamilton cycle} in a graph $G$ is a cycle covering
all the vertices of $G$ and a graph that contains a Hamilton cycle is said to
\emph{Hamiltonian}. The classical extremal theorem of Dirac~\cite{Dirac1952}
states that any $n$-vertex graph $G$ with $\delta(G)\geq \tfrac n2$ is
Hamiltonian. The idea that graphs satisfying Dirac's condition are robustly
Hamiltonian in some sense, has been around for some time, with various measures
of robustness being proposed. For example, S\'ark\"ozy, Selkow and Szemer\'edi
\cite{sarkozy2003number} showed that there is $c>0$ such that any
$n$-vertex graph $G$ with $\delta(G)\geq \tfrac n2$ contains at least $c^nn!\geq
(c^2n)^n$ Hamilton cycles. This is tight up to the value of $c$ and the authors
of~\cite{sarkozy2003number} conjectured that one can in fact take
$c=\tfrac12-o(1)$, which was settled by Cuckler and
Kahn~\cite{cuckler2009hamiltonian}. This value of~$c$ is best possible, as can
be seen by considering a $G(n,p)$ with $p=\tfrac12+o(1)$.

Having a large number of Hamilton cycles is compelling evidence for such graphs
being robustly Hamiltonian but this property alone does not preclude the
possibility that these Hamilton cycles are somehow concentrated on a small part
of the graph, for example that many of them share a small subset of
edges. Further research has gone into proving stronger notions of robustness,
for example showing the existence of many edge-disjoint Hamilton cycles or the
existence of a Hamilton cycle when an adversary forbids the use of certain
combinations of edges (see the nice survey of
Sudakov~\cite{sudakov2017robustness} and the references therein).

An essentially optimal robustness result concerning random sparsifications of
graphs satisfying Dirac's condition was obtained by Krivelevich, Lee and
Sudakov~\cite{Krivelevich2014}, who proved that for any $n$-vertex graph~$G$
with $\delta(G) \geq \tfrac n2$, whp $G_p$ is Hamiltonian when $p \geq C
(\log n)/n$ for sufficiently large~$C$. For comparison, as proved by
Kor{\v{s}}unov~\cite{Korshunov1976} and P\'osa~\cite{Posa1976} the threshold for
$G(n,p)$ to be Hamiltonian is also $(\log n)/n$.

The robustness given by the theorem of Krivelevich, Lee and
Sudakov~\cite{Krivelevich2014} is relatively strong in that is can easily be
used to infer other notions of robustness. For example, as every Hamilton cycle
in a graph $G$ survives in $G_p$ with probability $p^n$, by considering the
expected number of Hamilton cycles in $G_p$ analogously to our derivation of \cref{cor:triangle factor count}, we can conclude that any graph $G$
with $\delta(G)\geq \tfrac n2$ has at least $(\tfrac{cn}{\log n})^n$ Hamilton
cycles for some $c>0$, which is only slightly weaker than the aforementioned
results counting Hamilton cycles. One can also obtain many edge-disjoint
Hamilton cycles by considering a random partition of the edges of $G$.

Several further results have built on the idea of using random sparsifications
to give robustness, such as those of Johansson~\cite{johansson2020hamilton} and
Alon and Krivelevich~\cite{alon2019hitting} concerning `hitting times'. Other
graph properties, such as the existence of long paths and cycles or perfect
matchings, have also been investigated in the random sparsification setting; we
refer again to the survey~\cite{sudakov2017robustness} for details.

\paragraph{\textbf{Additional note.}} Since this paper was first submitted and a preprint posted online,  Pham, Sah, Sawhney and Simkin \cite{pham2022toolkit} have provided a general method for proving robust threshold results. Their approach uses spread measures and the pioneering result of Frankston, Kahn, Narayanan and Park \cite{frankston2019thresholds} which allows one to upper bound thresholds in terms of how \emph{spread} the graph property is. In the context of clique factors, they could use their methods to prove an analogue to Theorem~\ref{thm:main} for $K_k$-factors for all $k\geq 3$ and also to answer Problem \ref{prob:no of factors} from our concluding remarks in the affirmative, establishing a lower bound on the number of clique factors in graphs above the extremal threshold.  In particular, they provide an alternative proof of Theorem \ref{thm:main} and  remove the $\log$ factor in Corollary \ref{cor:triangle factor count}. Their proof follows the same general scheme as ours in first reducing to a partite super-regular setting which we give here as our main technical theorem, Theorem~\ref{thm:main-super-reg}. The reduction is very similar to ours given here and indeed they use some of the tools we develop here including a stability version for the fractional Hajnal-Szemer\'edi theorem (Theorem~\ref{thm:HSzFrac} of this paper).  It is in the proof of  Theorem \ref{thm:main-super-reg},  that our approaches  diverge completely. As previously mentioned, they use the recent breakthrough result \cite{frankston2019thresholds} on thresholds which reduces the problem to finding an appropriate spread measure. In order to get the correct $\log$ factor in the robust threshold, they also need to transition to finding perfect matchings in random hypergraphs, by using coupling results of Riordan \cite{Riordan2018}. On the other hand, our approach to Theorem \ref{thm:main-super-reg}  is based on entropy, builds on the original proof of Johansson, Kahn and Vu \cite{Johansson2008} for the threshold of clique factors and is self-contained. Whilst the proof of Pham, Sah, Sawhney and Simkin is more succinct and generalises immediately to other settings, we believe that both proof methods develop exciting new ideas and have great potential to be used in further work. 

\paragraph{\textbf{Organisation.}}

The remainder of the paper is organised as follows.
In \cref{sec:prelim} we collect some basic definitions and a variety of tools that we shall need for our proof. In particular, we discuss large matchings of cliques in \cref{sec:det-partial}, mention concentration inequalities we use in \cref{sec:probtools}, introduce what we need from the regularity method in \cref{sec:regularity}, and list useful facts about entropy in \cref{sec:entropybasics}.

In \cref{sec:proofoverview} we explain that the main instrument for proving \cref{thm:main} is a result on triangle factors in random sparsifications of super-regular tripartite graphs, \cref{thm:main-super-reg}. We then give an overview of the proof of this main technical theorem, state the main propositions and lemmas needed for this and show how these imply \cref{thm:main-super-reg}. More precisely, we shall formulate one proposition, \cref{prop:almost-factor}, allowing us to count certain partial triangle factors, one proposition, \cref{prop:full-step}, allowing us to extend a partial triangle factor by one triangle, and a key lemma, which we call the Local Distribution Lemma  (\cref{lem:LDL}). After this, we provide some results on triangle counts in \cref{sec:smallsubgraphs}, which will be useful in the proofs of our propositions. In \cref{sec:otherlemmas}, we prove \cref{prop:almost-factor} and \cref{prop:full-step}, using \cref{lem:LDL} as a black box.
In \cref{sec:LDL}, we show \cref{lem:LDL}. An important ingredient of this proof is a lemma which we call the  Entropy Lemma (\cref{lem:entropy lemma}).

This  will complete the proof of the main technical theorem, Theorem~\ref{thm:main-super-reg}, and it will remain to deduce \cref{thm:main} from Theorem~\ref{thm:main-super-reg}. Before embarking on this, we need to build some more theory. We begin in \cref{sec:HSzFrac} by  providing a stability statement of a fractional version of the Hajnal--Szemer\'edi theorem, which may be of independent interest. Next, in \cref{sec:triangle matchings}, we  derive a sequence of probabilistic lemmas which imply the existence of~$K_3$-matchings in various random sparsification settings. 
In \cref{sec:reduction}, finally, we show  how
\cref{thm:main-super-reg} implies \cref{thm:main}. 
The basic approach we use is a combination of the regularity method with an analysis of the extremal cases, as is common in the area.

Finally in \cref{sec:conclude} we provide some concluding remarks.

\subsection*{Acknowledgements}

The authors would like to thank Yoshiharu Kohayakawa for insightful discussions leading to 
this project. Additionally, the sixth author would like to thank Michael
Anastos, Shagnik Das, and Jie Han for enlightening conversations about this paper
and related results. Finally, we thank the anonymous reviewer for their suggestions. 

\section{Preliminaries}
\label{sec:prelim}

Here we collect the notation we will use and provide some of the necessary definitions and tools.

\subsection{Notation} \label{sec:notation}

\paragraph{\textbf{Basics:}}
We use $[n]_0$ to denote $[n]\cup \{0\}$. For $0\leq t \leq n\in\N$, we define $n!_t$ to be the number of ways to select a list of $t$ distinct numbers from $[n]$. That is, $n!_0\coloneq 
1$ and for $1\leq t \leq n$, we have \[n!_t\coloneq\frac{n!}{(n-t)!}=n\cdot (n-1)\cdots (n-t+1).\]
We use the notation $x=y \pm z$ to denote that $x\leq y+z$ and $x \geq y-z$. Throughout we use~$\log$ to denote the natural (base~$e$) logarithm function. Finally, we drop ceilings and floors unless necessary, so as not to clutter the arguments. 

\paragraph{\textbf{Constants:}}
 At times we will define  constant hierarchies within proofs, writing  statements such as the following: Choose constants \begin{equation} \label{eq:constant-hierarchy-example}
0 < c_1 \ll c_2 \ll \ldots \ll c_\ell  \ll d.
\end{equation} 
This should be taken to mean that given some constant $d$ (given by the statement we aim to prove), one can choose all the remaining constants (the $c_i$) from right to left so that all the subsequent constraints are satisfied. That is, there exist increasing functions $f_i$ for $i\in [\ell+1]$ such that whenever $c_{i}\leq f_{i+1}(c_{i+1})$ for all $i\in[\ell-1]$ and $c_{\ell}\leq f_{\ell+1}(d)$, all constraints on these constants that are in the proof, are satisfied.

\paragraph*{\textbf{Neighbourhoods and degrees:}}
Given a graph~$G$, a vertex~$v \in V(G)$ and a set~$U \subseteq V(G)$, we define the \emph{neighbourhood} of~$v$ in~$U$ as~$N_G(v;U) \coloneq \{u \in U: uv \in E(G)\}$. If~$U = V(G)$, we simply write~$N_G(v)$ and if~$G$ is clear from context we drop the subscript.
If two vertices~$u_1,u_2 \in V(G)$ are given, then~$N_G(u_1,u_2) \coloneq N_G(u_1) \cap N_G(u_2)$ denotes the \emph{common neighbourhood} of~$u_1$ and~$u_2$. We will also use this notation for an edge~$e=u_1u_2$, taking that~$N_G(e)=N_G(u_1,u_2)$. 
Similarly, if~$S\subset V(G)$ is some subset of vertices,~$N_G(S)\coloneq\cap_{u\in S}N_G(u)$ denotes the common neighbourhood of the vertices in~$S$ and if~$\uu=(u_1,\ldots,u_\ell)$ is a tuple of vertices (an ordered set),~$N_G(\uu)\coloneq\cap_{j\in[\ell]}N_G(u_j)$ denotes the common neighbourhood of the set of vertices in~$\uu$. 
The parameters~$N_G(u_1,u_2;U)$,~$N_G(S;U)$ and~$N_G(\uu;U)$ 
are all defined analogously as the sets of common neighbours that lie in~$U$.
We follow the convention that~$N_G(\emptyset) = V(G)$. 
We also define \emph{degrees}~$\deg_G(u)=|N_G(u)|$ with~$\deg_G(u;U), \deg_G(S),$ $ \deg_G(S;U),$ $\deg_G(\uu)$ and~$\deg_G(\uu;U)$ 
defined analogously. Again, if the graph~$G$ is clear from the context then we drop the subscripts. Finally, we let $\delta(G):=\min_{u\in V(G)}\deg_G(u)$ denote the minimum degree of the graph $G$ and $\Delta(G):=\max_{u\in V(G)}\deg_G(u)$ the maximum degree.

\paragraph*{\textbf{Edge subsets as subgraphs:}}
Sometimes, given a graph~$G$ and a subset of edges~$E'\subseteq E(G)$, we will think of~$E'$ as the subgraph~$H_{E'}\coloneq(V(E'), E')$ of~$G$, where~$V(E')$ is the set of vertices that lie in edges in~$E'$. We then  use notation like 
$\delta(E') \coloneq \delta(H_{E'})$ and~$\deg_{E'}(v) \coloneq \deg_{H_{E'}}(v)$. Furthermore, for a vertex set~$A\subset V(G)$,~$E'[A]$ denotes the edges induced by~$H_{E'}$ on~$A$. That is,~$E'[A]:=\{e\in E':e\subset A\}$.

\paragraph*{\textbf{Triangles and cliques:}}
 For a graph~$G$ and~$r\in \NN$,~$r\geq 2$, we define~$K_r(G)$ to be the set of copies of~$K_r$ in~$G$.  For example,~$K_2(G)=E(G)$.  Given a set of~$r$-cliques~$\Sigma\subseteq K_r(G)$, we use the notation~$V(\Sigma)$ to denote all vertices that feature in cliques in~$\Sigma$, i.e.~$V(\Sigma):=\cup_{S\in\Sigma}S$. 
 For~$u\in V(G)$ we let~$K_r(G,u)\subseteq K_r(G)$ denote the subset of cliques containing~$u$.
 
 Now for a vertex~$v\in V(G)$, we let~$\tr_v(G)$ denote the \emph{triangle neighbourhood of~$v$}: the set of edges in~$E(G)$ that form a triangle with~$v$ in~$G$. That is,~$\tr_v(G)=\{e\in E(G):v\in N_G(e)\}$.  Note that~$K_3(G,u)=\{f\cup\{u\}:f\in \tr_u(G)\}$.

\paragraph*{\textbf{Matchings and factors:}}
For $r
\geq 2$, a~\emph{$K_r$-matching} in~$G$ is a collection of vertex-disjoint copies of~$K_r$ in~$G$. The \emph{size} of a~$K_r$-matching is the number of  vertex-disjoint copies of~$K_r$ in the collection.  Note that when~$r=2$ is a single edge, a~$K_r$-matching is simply a matching and when~$r=3$, we will also refer to a~$K_3$-matching as a \emph{triangle matching}. If a~$K_r$-matching covers the vertex set of~$G$ (implying that~$n\in r\mathbb{N}$), then we refer to the~$K_r$-matching as a~\emph{$K_r$-factor} in~$G$. Thus, when~$r=2$, a~$K_2$-factor is a perfect matching and when~$r=3$, we also refer to a~$K_3$-factor as a \emph{triangle factor}. At times, we will refer to a~$K_r$-matching as a \emph{partial~$K_r$-factor}. Although these two terms refer to the same objects, we reserve the use of partial factors for when there is an aim for the partial~$K_r$-factor/$K_r$-matching to  contribute to a full~$K_r$-factor.

\paragraph*{\textbf{Vertex sets and tuples in tripartite graphs:}}
For a large part of our proof, we will be concerned with the host graph being a balanced tripartite graph. In such a setting, we will take as convention that the disjoint vertex sets that form the tripartition are labelled~$V^1, V^2$ and~$V^3$ and are each of size~$n$. 
It will be useful for us to considered ordered tuples of vertices from these vertex sets. We therefore fix~$\pzc{V}\coloneq\{\emptyset\} \cup V^1 \cup (V^1 \times V^2) \cup (V^1\times V^2\times V^3)$.
That is, an element~$\uu\in \pzc{V}$ is a vector of some \emph{length}~$0\leq \ell(\uu)\leq  3$ such that for each~$i\leq \ell(\uu)$, we have that~$\uu$ contains exactly one vertex from~$V^i$.

\paragraph*{\textbf{Vertex sets with elements removed:}}
Given a graph~$G$, a collection of vertices~$u_1,\ldots,u_\ell \in V(G)$ and a subset of vertices~$W\subseteq V(G)$, we use the notation~$W_{\hat u_1,\ldots,\hat u_\ell}$ to denote the subset~$W$ with the~$u_i$ removed. That is, \begin{equation*} \label{eq:adjusted set}
W_{\hat u_1,\ldots,\hat u_\ell}\coloneq W\setminus (W\cap \{u_1,\ldots,u_\ell\}).
\end{equation*}
Note that we do not impose that  the~$u_i$ need lie in~$W$. We remark that we add a hat on the removed vertices~$u_i$ in this notation to distinguish it from similar notation (see below)  where vertices appear in subscripts without hats, signalling that these vertices are used for certain purposes.

To ease notation, we will sometimes group together some of the collection of vertices we wish to omit, as an ordered tuple. For example, if
$\uu=(u_1,\ldots,u_\ell) \in \pzc{V}$ for some~$\ell\in[3]_0$ as above, we define~$W_{\hat \uu}\coloneq W_{\hat u_1,\ldots,\hat u_\ell}$.

\paragraph*{\textbf{Partial triangle factors in tripartite graphs:}}
We will be concerned with embedding partial triangle factors in a given host tripartite graph. For~$t \in [n]_0$, we therefore define~$D_t$ to be the graph on vertex set~$[t]\times [3]$, whose edge set consists of the edges~$\{\{(s,i),(s,j)\}:  s\in [t], i\neq j \in [3]\}$. Thus~$D_t$ simply consists of~$t$ labelled vertex-disjoint triangles.

Given a graph~$G$ on a fixed vertex partition~$V^1 \cup V^2 \cup V^3$ as above, we define~$\Psi^t(G)$ to be the collection of labelled embeddings of~$D_t$ into~$G$, that map~$[t]\times\{i\}$ to a subset of~$V^i$ for~$i\in [3]$.
We will be interested in embeddings that fix certain vertices to be isolated.
Given~$\uu=(u_1,\ldots,u_\ell) \in \pzc{V}$ of length~$\ell \leq 3$ as above and~$t\in [n-1]$, we define~$\Psi_{\hat{\uu}}^t(G) \subseteq \Psi^t(G)$ to be those~$\psi \in \Psi^t(G)$ for which~$\psi((s,i)) \neq u_i$ for all~$i \in [\ell]$ and~$s\in [t]$. That is, we fix the~$\ell$  vertices in~$\uu$ to be isolated in the embedding of~$D_t$.

We remark that if~$\uu=\emptyset$, then~$\Psi_{\hat{\uu}}^t(G) = \Psi^t(G)$ and also note that for an arbitrary~$\uu\in \pzc{V}$ one has that~$\Psi_{\hat{\uu}}^t(G)=\Psi^t(G_{\hat \uu})$ where~$G_{\hat \uu}$ is considered as a tripartite graph on partition~$V^1_{\hat \uu}\cup V^2_{\hat \uu}\cup V^3_{\hat \uu}$. 

Finally, given a vertex~$v \in V^1$, we denote by~$\Psi_{v}^t(G) \subseteq \Psi^t(G)$ the set of embeddings~$\psi \in \Psi^t(G)$ for which~$\psi((1,1)) = v$.

\paragraph*{\textbf{Induced subgraphs:}}
For a graph~$G = (V,E)$ and some~$U \subseteq V$, we define~$G[U]$ to be the subgraph of~$G$ \emph{induced} by~$U$, that is~$V(G[U]) = U$ and~$E(G[U]) = \{e \in E: e \subset U\}$.
Similarly, given disjoint subsets~$U_1, \ldots, U_k \subset V$, we define~$G[U_1, \ldots, U_k]$ to be the~$k$-partite subgraph of~$G$ induced by~$U_1, \ldots, U_k$, that is~$V(G[U]) = U_1 \cup \ldots \cup U_k$ and
\[E(G[U_1, \ldots, U_k]) = \{e \in E: e \subset U_1 \cup \ldots \cup U_k \text{ and } |e\cap U_i| \leq 1 \text{ for all } i \in[k]\}.\]
Given a graph~$G$ and a collection of vertices~$u_1,\ldots, u_\ell$, we consider the graph induced after removing the~$u_i$, by defining the shorthand~$G_{\hat u_1,\ldots,\hat u_k}\coloneq G[V_{\hat u_1,\ldots,\hat u_k}]$, where~$V=V(G)$. For a tuple of vertices~$\uu$, the graph~$G_{\hat \uu}$ is defined analogously.

\subsection{$K_k$-matchings in dense graphs}
\label{sec:det-partial}
The Hajnal--Szemer\'{e}di Theorem \cite{Hajnal1970} states that any graph with maximum degree $\Delta$ has an \emph{equitable colouring} with $\Delta+1$ colours, that is, a colouring where the colour classes differ in size by at most one. Applying this to the complement of $G$, which has maximum degree $n-1-\delta(G)$, we find a collection of $n-\delta(G)$ vertex-disjoint cliques in $G$ whose sizes differ by at most one and that cover $V(G)$. We will make use of the following corollary, which we obtain from the fact that when $\delta(G)=(\tfrac{k-1}{k} - x) n$ for some $0\le x<1$, then the Hajnal--Szemer\'{e}di Theorem provides us with  $(\frac1k+x)n$ vertex-disjoint cliques. If $0<x<\tfrac{1}{k(k-1)}$, some of these cliques, say $\alpha$, are of size~$k$, and the others are of size~$k-1$, hence we have $n=\alpha k +((\frac1k+x)n-\alpha)(k-1)=\alpha+\frac{n}{k}(1+kx)(k-1)$. Solving this for~$\alpha$ gives the following result.

\begin{thm}[Hajnal, Szemer\'edi~\cite{Hajnal1970}]\label{thm:HajSze-factor}
Let~$n,k \geq 2$ be integers and let~$0 \leq x <1$. Suppose that~$G$ is an~$n$-vertex graph with~$\delta(G) \geq \big(\tfrac{k-1}{k} - x\big) n$.
Then~$G$ contains a~$K_k$-matching of size at least~$(1-(k-1)kx) \lfloor \tfrac{n}{k} \rfloor$.
\end{thm}

This statement is often used in extremal graph theory, and in particular the case $x=0$, which gives the best possible minimum degree condition for containing a  $K_k$-factor.

\subsection{Concentration Inequalities}\label{sec:probtools}
We will frequently use the following concentration inequalities for random variables. The first such inequality, Chernoff's inequality \cite{Chernoff1952} (see also \cite[Corollary 2.3]{Janson2011}), deals with the case of binomial random variables.  

\begin{thm}[Chernoff's concentration inequality] \label{thm:chernoff}
Let~$X$ be the sum of a set of mutually  independent Bernoulli random variables and let~$\lambda=\EE[X]$. Then for any~$0<\delta<\tfrac{3}{2}$, we have that 
\[\PP[X\geq (1+\delta)\lambda]\leq  e^{-\delta^2\lambda/3 } \hspace{2mm} \mbox{ and } \hspace{2mm} \PP[X\leq (1-\delta)\lambda] \leq  e^{-\delta^2\lambda/2 }.\]
\end{thm}

Recall that given a graph $G$ and some $p \in [0,1]$, we denote by $G_p$ the random subgraph of $G$ with $V(G_p) = V(G)$ in which every edge of $G$ is present independently with probability $p$. Given a subgraph $F \subset E(G)$ of $G$ (given by its edge set), we denote by $I_{F}$ the indicator random variable which is $1$ if $F$ is present in $G_p$ and $0$ otherwise.
Chernoff's inequality is particularly useful to give sharp bounds on random variables of the form $X = \sum_{F \in \cF} I_{F}$, where $\cF \subset 2^{E(G)}$ is a collection of edge-disjoint subgraphs of $G$.

However, when $\cF$ consists of not-necessarily edge disjoint subgraphs of $G$, the situation becomes more complicated.
Janson's inequality \cite{Janson1990} (see also \cite[Theorem 2.14]{Janson2011}) provides a bound for the lower tail in this case.

\begin{lemma}[Janson's concentration inequality]\label{lem:janson}
Let $G$ be a graph and $\cF \subset 2^{E(G)}$ be a collection of subgraphs of $G$ and let $p \in [0,1]$. Let $X = \sum_{F \in \cF} I_{F}$, let $\lambda = \Exp X$ and let
\[\bar \Delta = \sum_{(F,F') \in \cF^2: \ F \cap F' \not = \emptyset} \Exp{I_F I_{F'}}.\]
Then, for every $ \e \in (0,1)$, we have
\[\Pro{X \leq (1-\e)\lambda} \leq \exp \left(- \frac{\e^2\lambda^2}{2 \bar \Delta} \right).\]
\end{lemma}

If we additionally require a bound for the upper tail, we will use the Kim--Vu inequality \cite{Kim2000} (see also \cite[Theorem 7.8.1]{Alon2015}).
 Let $X = \sum_{F \in \cF} I_{F}$ as above. Given an edge $e \in E(G)$, we write~$t_e$ for $I_{\{e\}}$. With this we can write $X$ as a polynomial with variables $t_e$:
 \[X = \sum_{F \in \cF} \prod_{e \in F} t_e.\]
 Given some $A \subset E(G)$, we obtain $X_A$ from $X$ by deleting all summands corresponding to $F \in \cF$ which do not contain $A$ and replacing every $t_e$ with $e \in A$ by $1$.
 That is,
 \[X_A = \sum_{F \in \cF: \ A \subseteq F} \prod_{e \in F \setminus A} t_e.\]
 In other words, $X_A$ is the number of $F \in \cF$ that contain $A$ and are present in $G_p\cup A$.
 \begin{lemma}[Kim--Vu polynomial concentration]\label{lem:Kim--Vu}
 For every $k \in \N$, there is a constant $c = c(k) >0$ such that the following is true. Let $G$ be a graph and $\cF \subset 2^{E(G)}$ be a collection of subgraphs of $G$, each with at most $k$ edges. Let $X=\sum_{F \in \cF} I_F$ as above and $\lambda\coloneq \Exp{X}$. For $i \in [k]$, define $E_i \coloneq \max \{\Exp{X_A}: \ A \subset E(G), \ |A| = i\}$. Further define $E' \coloneq \max_{i \in [k]} E_i$ and $E = \max \{\lambda,E'\}$.
 Then, for every $\mu >1$, we have
   \[ \Pro{|X - \lambda| > c \cdot (EE')^{1/2} \mu^k} \leq c\cdot e(G)^{k-1} e^{-\mu}.\]
 \end{lemma}

 Finally we will need a basic concentration result for the \emph{hypergeometric distribution}: A random variable~$X$ is \emph{hypergeometrically distributed} with parameters~$N \in \N$ and ~$K,t \in [N]_0$  if for all~$k\in [K]_0$,~$\Pro{X=k}$ is the probability that when drawing~$t$ balls from a set of~$N$ balls ($K$ of which are blue and~$N-K$ red) without replacement, exactly~$k$ are blue. That is,
\[
\Pro{X=k} = \frac{\binom{K}{k} \binom{N-K}{t-k}}{\binom{N}{t}}.
\]
We will use the following concentration inequality, which Chv\'{a}tal~\cite{Chvatal1979} deduced from  Hoeffding's inequality~\cite{Hoeffding1963}, see also~\cite{skala2013hypergeometric}.

\begin{lemma}\label{lem:conc-hypergeometric}
	Let~$X$ be hypergeometrically distributed with parameters~$N \in \N$,~$K \in [N]_0$ and~$t \in[N]_0$ and let~$\lambda \coloneq \Exp{X} = \tfrac{tK}{N}$. Then, for all~$\e > 0$, we have
	\[
	\Pro{|X-\lambda| > \e \lambda} \leq 2 e^{-2\e^2 (K/N)  \lambda}.
	\]
\end{lemma}

\subsection{Regularity} \label{sec:regularity}
We will use the famous regularity lemma due to Szemer\'edi~\cite{szemeredi} which is an extremely powerful tool in modern extremal combinatorics. The lemma and its consequences appeared in the form we give here, in a survey of Koml\'os and Simonovits~\cite{Komlos1996}, which we also recommend for further details on the subject. First we introduce some necessary terminology. 
Let~$G$ be a graph and let~$A,B \subset V(G)$  be disjoint subsets of the vertices of~$G$. For non-empty sets~$X\subseteq A$,~$Y\subseteq B$, we define the \emph{density of~$G[X,Y]$} to be~$d_G(X, Y)\coloneq \tfrac{e_G(X,Y)}{|X||Y|}$.
Given~$\e > 0$, we say that a pair~$(A,B)$ is~\emph{$\e$-regular} in~$G$ if for all sets~$X \subseteq A$ and~$Y \subseteq B$ with~$|X|\geq \e |A|$ and~$|Y| \geq \e |B|$ we have~$|d_G(A,B) - d_G(X,Y)| < \e$.
We say that~$(A,B)$ is~$(\e,d)$-regular if~$(A,B)$ is~$\e$-regular and~$d_G(A,B)=d$.

Furthermore, we say~$(A,B)$ is~$(\e,d,\delta)$-\emph{super-regular} if~$(A,B)$ is~$(\e,d)$-regular and satisfies $\deg_G(v;A) \geq \delta|A|$ for all~$v\in B$ and likewise~$\deg(v;B)\geq \delta |B|$ for all~$v\in A$. We say that~$(A,B)$ is~$(\e,d)$-super-regular if it is~$(\e,d,d-\e)$-super-regular.
We say that a~$k$-tuple~$(A_1,\ldots, A_k)$ of (pairwise disjoint) subsets of~$V(G)$ is~$(\e,d)$-(super-)regular if each of the pairs~$(A_i,A_j)$ with~$i\neq j\in[k]$  is~$(\e,d)$-(super-)regular.
We call a~$k$-partite graph~$G$ with parts~$A_1,\ldots,A_k$,~\emph{$(\e,d)$-(super-)regular} if~$(A_1, \ldots, A_k)$ is an~$(\e,d)$-(super-)regular tuple in~$G$.
In the interest of brevity, we use the term   \emph{(super-)regular tuple} interchangeably to refer to the tuple of vertex sets~$(A_1,\ldots,A_k)$ and also to refer to the   (super-)regular~$k$-partite graph~$G[A_1, \ldots, A_k]$ that~$G$ induces on~$A_1 \cup \ldots \cup A_k$. Finally we say that~$(A,B)$  is~\emph{$(\eps,d^+)$-regular} if it is~$(\eps,d')$-regular for some~$d'\geq d$. Similarly, we say~$(A,B)$ is~\emph{$(\eps,d^+,\delta)$-super-regular}  if it is~$(\eps,d',\delta)$-super-regular for some~$d'\geq d$ and we say~$(A,B)$ is~$(\eps,d^+)$-super-regular  if it is~$(\eps,d',d-\eps)$-super-regular for some~$d'\geq d$.   The corresponding definitions are made analogously for regular tuples where we require the densities between all pairs involved to be at least~$d$ (and do not require these densities to be equal).  

\smallskip

We say that a partition~$V(G)=V_0\cup V_1\cup\dots\cup V_t$ is an~\emph{$\eps$-regular partition} if~$|V_0|\le\eps |V(G)|$,~$|V_1|=\cdots=|V_t|$, and for all but at most~$\eps t^2$ pairs~$(i,j)\in [t]\times [t]$, the pair~$(V_i,V_j)$ is~$\eps$-regular. We refer to the sets~$V_i$ for~$i\in [t]$  as \emph{clusters} and also use this term to refer to subsets~$V_i'\subset V_i$ for~$i\in [t]$.   We refer to~$V_0$ as the \emph{exceptional set} and the vertices in~$V_0$ are \emph{exceptional vertices}. Given an~$\eps$-regular partition and~$d\in[0,1]$, we say~$R$ is the~$(\eps,d)$-reduced graph of~$G$ (with respect to the partition) if~$V(R)=[t]$ and~$ij\in E(R)$ if and only if~$(V_i,V_j)$ is~$(\eps,d^+)$-regular. 
We will use  Szemer\'{e}di's Regularity Lemma~\cite{szemeredi} in the following form which follows easily from e.g.~\cite[Theorem 1.10]{Komlos1996}.
\begin{lemma}[Regularity Lemma]
\label{lem:reglem}
   For all~$0<\eps\le 1$ and~$m_0\in \N$
 there exists~$M_0\in \N$ such that for every~$0<d<\gamma<1$, every graph~$G$ on~$n>M_0$ vertices with minimum
 degree~$\delta(G)\ge\gamma n$ has an~$\eps$-regular partition~$V_0\cup V_1\cup\dots\cup V_m$ with~$(\eps,d)$-reduced graph~$R$ on~$m$
 vertices such that~$m_0\le m\le M_0$ and~$\delta(R)\ge(\gamma-d-2\eps)m$.
\end{lemma}
We will further make use of the following well-known results about (super-)regular tuples. See, for example,~\cite[Facts 1.3 and 1.5]{Komlos1996}. 

\begin{lemma}[Slicing Lemma]\label{lem:reg-slicing}
Let~$0<\e <\beta,d \le 1$ and let~$(V_1,V_2)$ be an~$(\e,d)$-regular pair. Then any pair~$(U_1,U_2)$ with~$|U_i| \geq \beta |V_i|$ and~$U_i\subseteq V_i$,~$i=1,2$, is~$(\e',d')$-regular with~$\e' = \max\{ \tfrac{\e}{\beta}, 2\e \}$ and some~$d'>0$ such that~$|d' -d|\leq \e$.
\end{lemma}

\begin{lemma}\label{lem:reg-degree}
  Let~$0<\e<d\le 1$ and~$(V_1, V_2)$ be an~$(\e,d)$-regular pair  and let~$X_2 \subseteq V_2$ with~$|X_2| \geq \e |V_2|$. Then all but at most~$\e |V_1|$ vertices~$v \in V_1$ satisfy~$\deg(v;X_2) \geq (d-\e)|X_2|$. Likewise, all  but at most~$\e |V_1|$ vertices~$v \in V_1$ satisfy~$\deg(v;X_2) \leq (d+\e)|X_2|$
\end{lemma}

The following lemma can be proven by combining the two previous lemmas. 

\begin{lemma}\label{lem:super-reg}
Let~$k\in \NN$  and~$0<\e<d\le 1$ with~$\e \leq \tfrac{1}{2k}$. If~$Z = (V_1, \ldots, V_k)$ is an~$(\eps,d^+)$-regular tuple of disjoint vertex sets 
  of size~$n$, 
  then there are subsets~$\tilde V_1 \subseteq V_1, \ldots, \tilde V_k \subseteq V_k$ with
$|\tilde{V}_i| = \lceil (1-k\e) n \rceil$ for all~$i \in [k]$ so that the~$k$-tuple~$\tilde Z = (\tilde V_1, \ldots, \tilde V_k)$ is~$(2\e, (d-\e)^+, d - k\e)$-super-regular.
\end{lemma}

Our next lemma shows that any sufficiently dense pair is automatically regular. It follows directly from the definition of regularity.

\begin{lemma}\label{lem:very-dense-implies-regular}
    Let~$0<\e<1$ and ~$(V_1,V_2)$ be a pair of vertex sets such that~$\deg(v_i;V_{3-i})\geq \left(1-\eps^2\right)|V_{3-i}|$ for all~$i\in[2]$ and~$v_i\in V_i$. Then~$(V_1,V_2)$ form an~$\left(\eps,\left(1-\eps^2\right)^+\right)$-super-regular pair. 
\end{lemma}

We will also need the following lemma which is closely related to the well-known counting lemma  and can be derived easily from the definition of~$\e$-regularity, we omit the proof here. 

\begin{lemma}\label{lem:super-reg-triangle-count}
Let~$0<\e< d_{1,2},d_{1,3},d_{2,3}\le 1$ and let~$\Gamma$ be a  
tripartite graph with parts~$V^1, V^2, V^3$ of size~$n$  such that~$(V^i,V^j)$ is~$(\e, d_{i,j})$-regular for all~$1 \leq i < j \leq 3$. 
Let~$X_i \subseteq V^i$ with~$|X_i| \geq \e n$ for all~$i \in [3]$.
Then,
\[ \card{K_3(\Gamma[X_1 \cup X_2 \cup X_3])} = d_{1,2}d_{1,3}d_{2,3} |X_1||X_2||X_3| \pm 10 \e n^3.\]
\end{lemma}

Finally, the following lemma further allows us to control the exact density of a super-regular pair by deleting edges if necessary. We recall here that we say a pair $(A,B)$ of disjoint vertex sets in $(\e,d^+)$-super-regular if it is $(\e,d',d-\e)$-super-regular for some $d'\ge d$. 

\begin{lemma}\label{lem:reg-exact-density}
	For all~$0<\e<1$, there is some~$n_0 > 0$, such that the following is true for every~$n \geq n_0$ and every bipartite graph~$G$ with parts~$V_1, V_2$ of size~$n$.
	Suppose that~$(V_1,V_2)$ is~$(\e^2,d^+)$-super-regular for some~$d$ such that~$4\e\le d \le 1$ and~$dn^2\in \NN$. Then there is a spanning subgraph~$G' \subseteq G$ so that~$(V_1,V_2)$ is~$(4\e,d)$-super-regular in~$G'$.
\end{lemma}

\begin{proof}
Let~$d'\geq d$ be the density of~$(V_1,V_2)$.	For~$i \in [2]$, let~$Y_i \coloneq \{v \in V_i: \deg(v;V_{3-i}) \leq \big(d'-\e^2\big)n\}$ and observe that by the~$\e^2$-regularity of~$(V_1,V_2)$ and \cref{lem:reg-degree}, we have~$|Y_i| \leq \e^2n$ for both~$i \in [2]$. Let~$E_Y \subset E(G)$ be the set of edges with at least one vertex in~$Y\coloneq Y_1 \cup Y_2$ and let~$E \coloneq E(G) \setminus E_Y$. Let~$m \coloneq |E_Y| \leq 2\e^2n^2$.
	Let~$p\coloneq \tfrac{d n^2 - m}{|E|} = \tfrac{d \pm 2 \e^2}{d'}$.
	Let~$E'$ be a uniformly random subset of~$E$ of size exactly~$p|E|\in \NN$ and let~$G'$ be the spanning subgraph of~$G$ with edge set~$E' \cup E_Y$. By construction, we have~$d_{G'}(V_1,V_2) = d$; we will show that~$(V_1,V_2)$ is whp~$(4\e,d,d-\e)$-super-regular in~$G'$.

	Let~$A_i \subseteq V_i$ with~$A_i \geq 4\e n$, and let~$A'_i = A_i \setminus Y_i$ and~$B_i = A_i \setminus A'_i$ for both~$i \in [2]$.
	By~$\e^2$-regularity in~$G$, we have~$Z \coloneq \card{E_{G}(A'_1,A'_2)} = (d' \pm \e^2) |A'_1||A'_2|$. Let now~$X \coloneq \card{E_{G'}(A'_1,A'_2)}$. Then~$X$ is hypergeometrically distributed with parameters~$N=|E|, K = Z, t = p|E|$ and thus~$\lambda \coloneq \Exp X = p  Z = (d \pm 2\e) |A'_1||A'_2|$.
	Since~$\lambda \geq 8\e^3n^2$, it follows from \cref{lem:conc-hypergeometric} that
	\[
	\Pro{|X-\lambda| > \e  \lambda} \leq  2e^{-2\e^2 (K/N)  \lambda} \leq 2e^{-\e^8 n^2}.
	\]
	In particular, we have~$\Pro{d_{G'}(A_1,A_2) = d \pm 4\e} \geq 1- 2e^{-\e^8 n^2}$.
	By taking a union bound over all choices of~$A_1, A_2$, we deduce that~$(V_1,V_2)$ is~$4\e$-regular with probability at least~$1- 2e^{2n-\e^8n^2}$.
	Similarly, we deduce that~$\deg_{G'}(v_i;V_{3-i}) \geq (d-\e)n$ for each~$i \in [2]$ and~$v_i \in V_i$ with probability at least~$1 - 4ne^{-\e^8n}$. Note that this is automatically true for all~$v \in Y$ as these vertices retain their neighbours from~$G$. 
	Hence, taking another union bound, it follows that~$(V_1,V_2)$ is  whp~$(4\e,d,d-\e)$-super-regular in~$G'$. Therefore, for all large enough~$n$, there is a suitable choice for~$E'$.
\end{proof}


\subsection{Entropy}\label{sec:entropybasics}

In this section we explain basic definitions and properties related to the entropy function, which will play a central r\^ole in our proof. 
We will be following  the notes of Galvin~\cite{Galvin2014} and all proofs we do not include here can be found or follow immediately from the results there. 
Throughout this subsection we fix a finite probability space~$(\Omega,\mathbb P)$. Recall also that~$\log$ denotes the natural logarithm function.

Let~$X: \Omega \to S$ be a random variable, and note that we will sometimes use the notation~$X(\omega)$, which is an element of~$S$, for the value of~$X$ given the outcome~$\omega\in\Omega$. Given~$x \in S$, we denote~$p(x) \coloneq\Pro{X=x}$. We define the \emph{entropy} of~$X$ by
\[h(X)\coloneq\sum_{x\in S}-p(x)\log p(x).\]
Entropy can be interpreted as a measure of the ``uncertainty'' of a random variable, or of how much information is ``gained'' by revealing~$X$.
The following lemma shows that the entropy is maximised when~$X$ is uniform, corresponding to maximal ``uncertainty''.
Define the \emph{range} of~$X$ as the set of values that~$X$ takes with positive probability, that is~$\rg(X) = \{x \in S: p(x) > 0\}$.
\begin{lemma}[maximal entropy]\label{lem:maximality of the uniform}
	For every random variable~$X: \Omega \to S$, we have $h(X) \leq \log(|\rg(X)|) \leq \log(|S|)$ with equality if and only if~$p(x) = \tfrac{1}{|S|}$ for all~$x \in S$.
\end{lemma}
\cref{lem:maximality of the uniform} provides the key to using entropy in combinatorial arguments. Indeed, the basic method relies on taking a uniformly random object~$F$ from some family~$\cF$ whose cardinality we are interested in estimating. By analysing the entropy of the random variable~$F$, using the tools listed below, we can obtain bounds on the entropy which translate to bounds on the size of~$\cF$ via \cref{lem:maximality of the uniform}. We now further develop the theory.

Given random variables~$X_i : \Omega \to S_i$ for~$i \in [n]$, we denote the entropy of the random vector~$(X_1, \ldots, X_n)$ by~$h(X_1, \ldots, X_n) \coloneq h((X_1, \ldots, X_n))$.
The entropy function has the following subadditivity property.
\begin{lemma}[subadditivity]\label{lem:subadditivity}
	Given random variables~$X_i: \Omega \to S_i$,~$i \in [n]$, we have
	\[h(X_1, \ldots, X_n) \leq \sum_{i=1}^n h(X_i),\]
	with equality if and only if the~$X_i$ are mutually independent.
\end{lemma}
Intuitively, this means that revealing a random vector cannot give us more information than revealing each component separately.
We say a random variable~$X: \Omega \to S_X$ \emph{determines} another random variable~$Y: \Omega \to S_{Y}$ if the outcome of~$Y$ is completely determined by~$X$. For example if~$X$ is the outcome of rolling a regular six-sided die and~$Y$ is~$1$ if this outcome is even, and~$0$ otherwise, then~$X$ determines~$Y$. 
Formally,~$X$ \emph{determines}~$Y$ if there is a function~$f: S_X \to S_Y$ such that~$Y(\omega) = f(X(\omega))$ for all~$\omega \in \Omega$.
If~$X$ determines~$Y$, then no additional information is needed to reveal~$Y$ once~$X$ is revealed. This is formalised in the following lemma.
\begin{lemma}[redundancy]\label{lem:redundancy}
	If~$X:\Omega \to S_X$ and~$Y: \Omega \to S_Y$ are random variables and~$X$ determines~$Y$, then
	$h(X)=h(X,Y)$.
\end{lemma}

If~$E \subset \Omega$ is an event with positive probability, we define the \emph{conditional entropy} given the event as
\[h(X|E) \coloneq \sum_{x\in S}-p(x|E)\log p(x|E),\]
where~$p(x|E)=\Pro{X=x|E}$.
Note that~$h(X|E)$ is the entropy of the random variable obtained from~$X$ by conditioning on~$E$, so that if~$Z$ has distribution~$\Pro{Z=x} = \Pro{X=x|E}$ then~$h(Z)=h(X|E)$.
Given two random variables~$X : \Omega \to S_X$ and~$Y: \Omega \to S_Y$, the conditional entropy of~$X$ given~$Y$ is defined as
\begin{align}
h(X|Y)
\coloneq\mathbb{E}_Y[h(X|Y=y)]
&=\sum_{y\in S_Y}p(y) h(X|Y=y)\label{eq:conditionalentropydef}\\
&=\sum_{\omega \in \Omega} \Pro{\omega} h(X|Y=Y(\omega)),\label{eq:conditionalentropydef-2}
\end{align}
where~$p(y)=\Pro{Y=y}$.
As conditioning on an event or another random variable only gives us more information, we have the following inequalities. 
\begin{lemma}[dropping conditioning]\label{lem:dropping conditioning}
	Given random variables~$X: \Omega \to S_X$ and~$Y: \Omega \to S_Y$, and an event~$E \subset \Omega$ we have
	\[h(X|Y) \leq h(X) \quad\text{and}\quad h(X) \geq \Pro{E} h(X|E).\]
	Furthermore, if~$Y' :\Omega \to S_{Y'}$ is another random variable and~$Y$ determines~$Y'$, then
	\[h(X|Y) \leq h(X|Y').\]
\end{lemma}
The following chain rule strengthens \cref{lem:subadditivity}.
\begin{lemma}[chain rule]\label{lem:chain rule}
	Given random variables~$X: \Omega \to S_X$ and~$Y: \Omega \to S_Y$, we have
	\[h(X,Y) = h(X) + h(Y|X)\]
	and more generally, for random variables~$X_i: \Omega \to S_i$,~$i \in [n]$, we have
	\[h(X_1, \ldots, X_n) = \sum_{i=1}^n h(X_i|X_1, \ldots, X_{i-1}).\]
\end{lemma}
%

%
%
\cref{lem:maximality of the uniform,lem:subadditivity,lem:chain rule} have the following conditional versions.
Given a random variable~$X:\Omega \to S_X$ and an event~$E \subset \Omega$, we define the conditional range of~$X$ given~$E$ by~$\rg(X|E) = \{x \in S_X: p(x|E) > 0\}$.
\begin{lemma}[maximal conditional entropy]\label{lem:cond maximality of the uniform}
	For every random variable~$X: \Omega \to S$ and event~$E\subset \Omega$, we have
	\[h(X|E) \leq \log\left(\card{\rg(X|E)}\right).\]
\end{lemma}
\begin{lemma}[conditional subadditivity]\label{lem:cond subadditivity}
	Given random variables~$X_i: \Omega \to S_i$,~$i \in [n]$, and~$Y:\Omega \to S_Y$, we have
	\[h(X_1, \ldots, X_n|Y) \leq \sum_{i=1}^n h(X_i|Y),\]
	with equality if and only if the~$X_i$ are mutually independent conditioned on~$Y$.
\end{lemma}
\begin{lemma}[conditional chain rule]\label{lem:cond chain rule}
	Given random variables~$X_i: \Omega \to S_i$,~$i \in [n]$, and~$Y: \Omega \to S_Y$, we have
	\[h(X_1, \ldots, X_n |Y) = \sum_{i=1}^n h(X_i|X_1, \ldots, X_{i-1},Y).\]
\end{lemma}
The following lemma will play an essential r\^ole in our proof. It sharpens a similar lemma that appeared in~\cite{Johansson2008}. It states that if a random variable has almost maximal entropy, then it must be close to uniform. This can be seen as a stability result for \cref{lem:maximality of the uniform}.
%
\begin{lemma}[almost maximal entropy]\label{lem:large-entropy}
	For all~$\beta > 0$, there is some~$\beta' > 0$ such that the following is true for every finite set~$S$ 
	and every random variable~$X: \Omega \to S$. 
	If~$h(X) \geq \log (\card{S}) - \beta'$, then letting~$a:=\tfrac{1}{|S|}$ and~$J: =\{x\in S:(1-\beta)a\leq \Pro{X=x}\leq (1+\beta)a\}$, we have that
	\begin{equation}\label{eq:large-entropy}
	|J| \geq (1-\beta)|S| \quad \text{ and } \quad \Pro{X\in J} \geq (1-\beta).
	\end{equation}
\end{lemma}

\begin{proof}
Let~$\beta>0$ be given and assume that $\beta<\tfrac{1}{10}$. Fix~$\beta'= \tfrac{\beta^4}{2000}$.
Let~$X: \Omega \to S$ be a random variable with~$h(X) \geq \log(|S|) - \beta'$ and let $a$ and $J$ be as defined in the statement of the lemma. 
Further, we define~$J^+ = \{y \in S:  \Pro{X=y} > \big(1+\tfrac{\beta}{4}\big)a \}$ and~$J^- = \{y \in S:  \Pro{X=y} < \big(1-\tfrac{\beta}{4}\big)a \}$. 
Note that~$|J| \geq  |S| - (|J^+| + |J^-|)$.

\begin {claim}
We have~$|J^+| \leq \tfrac{\beta}{4} |S|$.
\end {claim}
\begin{claimproof}
	Choose~$\eta \leq \tfrac{\beta}{4}$ so that~$\eta |S| = \lfloor \tfrac{\beta}{4} |S| \rfloor$. Assume for contradiction that~$|J^+| > \eta |S|$ and let~$\tilde J^+ \subset J^+$ be a set of size exactly~$\eta |S|$.
	Define~$X^+$ by
	\[
	 \Pro{X^+=y} =
	\begin{cases}
	(1 + \eta)a &\text{if } y \in \tilde J^+\\
	(1 - \xi)a &\text{if } y \not\in \tilde J^+,
	\end{cases}
	\]
	where~$\xi \coloneq \tfrac{\eta^2}{1-\eta}$ is chosen so that~$\sum_{y \in S}  \Pro{X^+=y} =1$.
	Now it follows from Karamata's inequality and the fact that $-x\log(x)$ is concave on $[0,1]$, that $h(X^+)\geq h(X)$. 
	We further let~$Y=1$ if~$X^+ \in \tilde J^+$ and~$0$ otherwise. We then have  that
	\begin{align*}
	h(X) \leq h(X^+) = h(X^+,Y)
	&=  h(X^+|Y=1) \Pro{Y=1} + h(X^+|Y=0) \Pro{Y=0} + h(Y),
	\end{align*}
	where we used \cref{lem:redundancy}, the chain rule (\cref{lem:chain rule}) and the definition of conditional entropy.
	Note that~$\Pro{Y=1} = \eta(1+\eta)$ and
	\[h(Y) = -\eta (1+\eta) \log\left(\eta (1+\eta)\right) - (1-\eta (1+\eta)) \log\left(1-(\eta (1+\eta))\right).\]
	Therefore, using also \cref{lem:cond maximality of the uniform}, we get
	\begin{align*}
	h(X)
	&\leq \logb{\eta |S|} \eta (1+\eta) + \logb{(1-\eta) |S|} (1 - \eta (1+\eta)) + h(Y)\\
	&= \logb{|S|} + \log(\eta) \eta(1+\eta) + \log(1-\eta) (1-\eta(1+\eta)) + h(Y)\\
	&= \logb{|S|} + \eta(1+\eta) \left(\log(\eta) - \log(\eta(1+\eta)) \right)\\
	&\phantom{{}=\logb{|S|}} + (1-\eta(1+\eta)) \left( \log(1-\eta) - \log(1-\eta(1+\eta)) \right)\\
	&= \logb{|S|} - \eta(1+\eta) \log(1+\eta) + \left(1-\eta-\eta^2\right) \logb{\tfrac{1-\eta}{1-\eta - \eta^2}}\,.
        \end{align*}
	Using the approximation~$x-\tfrac{x^2}{2} \leq \log(1+x) \leq x$, which holds for all~$x \in (0,1)$, in the forms~$\log(1+\eta) \geq \eta \left(1- \tfrac{\eta}{2}\right)$ and
	$\logb{\tfrac{1-\eta}{1-\eta - \eta^2}}
	= \logb{1+\tfrac{\eta^2}{1-\eta - \eta^2}}
	\leq \tfrac{\eta^2}{1-\eta - \eta^2}$, we conclude
	\begin{align*}
	h(X)
        &{\leq} \logb{|S|} - \eta^2(1+\eta)\left(1-\tfrac{\eta}{2}\right) + (1-\eta-\eta^2) \tfrac{\eta^2}{1-\eta - \eta^2}\\
	&= \logb{|S|} - \eta^2 - \tfrac{\eta^3}{2} + \tfrac{\eta^4}{2} + \eta^2
	\leq \logb{|S|} - \tfrac{\eta^3}{4}
	< \logb{|S|} - \beta',
	\end{align*}
	a contradiction.
\end{claimproof}

Similarly, we can show that~$|J^-| \leq \tfrac{\beta}{4}  |S|$ and conclude that
$|J| \geq |S| - (|J^+| + |J^-|) \geq (1-\beta)|S|$.
Furthermore, by the definition of~$J^-$ we have
\begin{align*}
  \sum_{y \in J}  \Pro{X=y}
  \ge\sum_{y\in S\setminus(J^+\cup J^-)} \left(1-\tfrac{\beta}{4}\right)a
&\geq \left(1-\tfrac{\beta}{2}\right)|S| \left(1-\tfrac{\beta}{4}\right)a
\geq (1-\beta).
\end{align*}
This completes the proof.
\end{proof}

\section{The main technical result and its proof overview} \label{sec:proofoverview}

The main technical result we reduce \cref{thm:main} to is the following partite version with the minimum degree condition replaced by regularity.
\begin{thm}[main technical theorem]\label{thm:main-super-reg}
For every~$0<d \le 1$ there exists constants~$\e>0$ and~$C > 0$ such that the following holds for every~$n \in \N$  and~$p\in (0,1)$ such that~$p \geq C(\log n)^{1/3} n^{-2/3}$. If $\Gamma$ is an~$(\eps,d^+)$-super-regular tripartite graph with parts of size~$n$ 
then~$\Gamma_p$ whp contains a triangle factor.
\end{thm}
The reduction of \cref{thm:main} to this partite version uses the regularity method together with a stability result for the fractional Hajnal--Szemer\'{e}di Theorem developed in~\cref{sec:HSzFrac} and an analysis of the extremal cases.
We give the full details in~\cref{sec:reduction}.

The main challenge of this paper is proving \cref{thm:main-super-reg}, and in this section we will reduce \cref{thm:main-super-reg} further to two intermediate propositions. We will then discuss the proof of these propositions, outlining  the remainder of the paper and some of the key ideas involved. We encourage the reader to recall the relevant terminology from the notation section (\cref{sec:notation}) on embedding partial factors in tripartite graphs, in particular the definition of~$\Psi^t$.

The first proposition counts partial triangle factors.
\begin{prop}[counting partial-factors]\label{prop:almost-factor}
For all~$0<\eta,d\le 1$ there exists~$\e>0$ and~$C>0$ such that  the following holds for all sufficiently large~$n \in \N$ and for any~$p \geq C(\log n)^{1/3} n^{-2/3}$.   If~$\,\Gamma$ is an~$(\e,d)$-regular tripartite graph with parts of size~$n$,  then whp we have that
	\begin{equation}\label{eq:almost-factor}
	\card{\Psi^{t}(\Gamma_p)} \geq (1-\eta)^{t} (p d)^{3t} \left(n!_t\right)^3,
	\end{equation}
	for all~$t\in \N$ with~$t \leq (1-\eta)n$.
\end{prop}
Here the condition~\cref{eq:almost-factor} should be read as~$\Gamma_p$ having roughly the `correct' number of embeddings of~$D_t$,  the graph with~$t$ labelled disjoint triangles. Indeed, the term~$(p d)^{3t} \left(n!_t\right)^3$ is the expected number of embeddings of~$D_t$ in a random sparsification of the complete tripartite graph with probability~$pd$, which provides a sensible benchmark for our model~$\Gamma_p$. The~$(1-\eta)^t$ factor is then an error term which we can control.
In order to go beyond \cref{prop:almost-factor} to counting subgraphs~$D_t$ with larger~$t$,  we need different techniques. Our second proposition allows us to extend partial triangle factors by embedding further triangles one by one.
\begin{prop}[extending by one triangle]\label{prop:full-step}
	 For all~$0<d\le 1$ there exists~$\alpha,\eta,\e>0$ and~$C>0$ such that  for all sufficiently large~$n \in \N$ and for any~$p \geq C(\log n)^{1/3} n^{-2/3}$,  if~$\,\Gamma$ is an~$(\e,d)$-super-regular tripartite graph with parts of size~$n$, then whp the following holds in~$\Gamma_p$  for every~$t\in \N$ with~$\left(1-\eta\right)n \leq t < n$.
	If \begin{equation} \label{eq:full-step1}
	    \card{\Psi^t(\Gamma_p)} \geq \left( 1-\eta\right)^n (pd)^{3t}(n!_t)^3, 	\end{equation} then
	\begin{equation} \label{eq:full-step2} \card{\Psi^{t+1}(\Gamma_p)} \geq \alpha (pd)^3(n-t)^3\card{\Psi^t(\Gamma_p)}.\end{equation}
\end{prop}
Again the condition~\cref{eq:full-step1} in \cref{prop:full-step} should be read as~$\Gamma_p$ having roughly the `correct' number of embeddings of~$D_t$ and condition~\cref{eq:full-step2} then implies that~$\Gamma_p$ has  roughly the `correct' number of embeddings of~$D_{t+1}$. In contrast to \cref{prop:almost-factor} we now lose control of the error term (given by~$\alpha$) but as we will only apply \cref{prop:full-step} for large~$t$, we can make sure the error term does not accumulate too much. Indeed, recall that our goal is merely to obtain one triangle factor in the end.

We now show how \cref{thm:main-super-reg} follows from these two intermediate propositions before outlining the proofs of these propositions.
\begin{proof}[Proof of \cref{thm:main-super-reg}]
Given~$d$ choose~$0<\e,\tfrac{1}{C}\ll \eta \ll \alpha \ll d$ and note that by choosing~$C>0$ sufficiently large,  we can assume that~$n$ is sufficiently large in what follows, as otherwise the statement is trivially true. Let us fix~$\Gamma$ to be an~$(\e,d^+)$-super-regular tripartite graph with parts of size~$n$. We can assume that~$dn^2\in\NN$. Indeed, if this is not the case,  then replace~$d$ with the minimum~$d'>d$ such that~$d'n^2\in \NN$ and note that, after redefining~$d$ (if necessary), we maintain that~$\Gamma$  is~$(\e,d^+)$-super-regular.  Now let~$\Gamma'$ be the~$(4\sqrt{\e},d)$-super-regular tripartite graph obtained by applying \cref{lem:reg-exact-density} between each of the parts of~$\Gamma$. As~$\Gamma'$ is a spanning subgraph of~$\Gamma$ it suffices to find our triangle factor in~$\Gamma'$.

Note that by our choice of constants, we have that whp both the conclusion of \cref{prop:almost-factor} (with~$\eta$ replaced by~$\eta^2$) and the conclusion of  \cref{prop:full-step} hold in~$\Gamma'$ simultaneously.
We will now assume they hold and show that this implies
\begin{equation}\label{eq:main-super-reg}
  \card{\Psi^{t}(\Gamma'_p)} \geq \left(1-\eta^2\right)^n\alpha^{t-\left(1-\eta^2\right)n} (p d)^{3t} \left(n!_t\right)^3,
\end{equation}
for all~$\left(1-\eta^2\right)n \leq t \leq n$.
Indeed, for~$t = \left(1-\eta^2\right)n$,~\cref{eq:main-super-reg} readily follows from (the assumed conclusion of) \cref{prop:almost-factor}. Assume now~\cref{eq:main-super-reg} holds for some~$\left(1-\eta^2\right)n \leq t < n$.
Since~$\eta \ll \alpha$, we have that
\begin{align*}
 \left(1-\eta^2\right)^n\alpha^{t-\left(1-\eta^2\right)n}
  &\geq \left(1-\eta^2\right)^n \alpha^{\eta^2 n}
  = \left(1-\eta^2\right)^n e^{-\log(1/\alpha)\eta^2 n}\\
  &\geq \left(1-\eta^2\right)^n \left(1 - \log\big(\tfrac{1}{\alpha}\big) \eta^2\right)^n
  \geq \left(1-\eta\right)^n.
\end{align*}
It follows from (the assumed conclusion of) \cref{prop:full-step} that~\cref{eq:main-super-reg} holds for~$t+1$.
In particular, we have
\[\card{\Psi^{n}(\Gamma_p)} \geq \card{\Psi^{n}(\Gamma'_p)} \geq  \left(1-\eta^2\right)^n\alpha^{\eta^2 n} (p d)^{3n} \left(n!\right)^3 \geq 1,\]
completing the proof.
\end{proof}

Thus it remains to prove \cref{prop:almost-factor,prop:full-step}.
Proving \cref{prop:almost-factor} is relatively straightforward: It follows
from embedding the triangles of~$D_t$ one by one greedily and counting in how
many ways we can embed each such triangle by using that all large enough vertex
sets whp induce roughly the `correct' number of triangles in~$\Gamma_p$,
which we establish in \cref{lem:triangle-count} using regularity and Janson's
inequality (\cref{lem:janson}).  The details for deriving
\cref{prop:almost-factor} are provided in \cref{sec:almostfactors}.

\smallskip

The proof of \cref{prop:full-step} is much more involved and the main challenge of this paper. In order to count embeddings of partial triangle factors in~$\Psi^{t+1}(\Gamma_p)$, one na\"ive idea would be to proceed as follows: We fix any triple~$\uu=(u_1,u_2,u_3)\in \pzc{V}$ of vertices and count in how many partial triangle factors from~$\Psi^t(\Gamma_p)$ these are isolated.
If this number were roughly the same for each triple of vertices
then we would be able to bound the size of~$\Psi^{t+1}(\Gamma_p)$ using bounds on how many triples actually form triangles in~$\Gamma_p$ to extend a partial triangle factor from~$\Psi^t(\Gamma_p)$ by one triangle.
However, we do not know how to prove that all triples of vertices behave similarly in this sense.
Hence, we need to resort to a more refined strategy, still considering embeddings which leave certain vertices isolated, but doing so in stages, growing our set of isolated vertices one vertex at a time.
This step-by-step process is made precise in the following Local Distribution Lemma, which is a key step of our argument.
We will show that this lemma implies \cref{prop:full-step} in \cref{sec:extendingfactors}.
\begin{lemma}[Local Distribution Lemma]\label{lem:LDL}
	For all~$0<\alpha,d\le 1$ and~$K>0$ there exists~$\eta,\e>0$ and~$C>0$ such that  for all sufficiently large~$n \in \N$ and for any~$p \geq C(\log n)^{1/3} n^{-2/3}$,  if~$\,\Gamma$ is an~$(\e,d)$-super-regular tripartite graph with parts of size~$n$, if $t\in \N$ is such that~$\left(1-\eta\right)n \leq t < n$, if $\ell\in[3]$ and $\uu=(u_1,\ldots,u_{\ell-1})\in \pzc{V}$ then the following holds in~$\Gamma_p$ with probability at least~$1-n^{-K}$. If
	\begin{equation} \label{eq:LDLhypothesis}\card{\Psi_{\hat{\uu}}^t(\Gamma_p)} \geq (1-\eta)^n (pd)^{3t}((n-1)!_{t})^{\ell-1} (n!_t)^{4-\ell},\end{equation}
	then for all but at most~$\alpha n$ vertices~$u_\ell\in V^\ell~$
        we have, with~$\uv=(\uu,u_{\ell})\in \pzc{V}$, that        
	\begin{equation}\label{eq:LDL}
		\card{ \Psi_{\hat{\uv}}^{t}(\Gamma_p)} \geq \left(\frac{d}{10}\right)^2  \left(\frac{n-t}{n}\right)\card{\Psi_{\hat{\uu}}^t(\Gamma_p)}\,.
  \end{equation}
\end{lemma}
Again,~\cref{eq:LDLhypothesis} should be read as~$\Gamma_p$ having roughly the `correct' number (up to the error term~$(1-\eta)^n$) of embeddings of~$D_t$ that avoid using vertices in~$\uu$, where correct means what we expect in a random sparsification of the complete tripartite graph with probability~$pd$. The conclusion of \cref{lem:LDL} then tells us that  that for most choices of extending~$\uu$ to~$\uv$, we have roughly the correct number of embeddings of~$D_t$ that avoid using the vertices in~$\uv$.

For proving \cref{prop:full-step} in \cref{sec:extendingfactors}, we shall use  \cref{lem:LDL} with~$\ell=2$ and~$\ell=3$ to prove a lemma, \cref{lem:embed}, which states that if for a vertex~$w\in V^1$ we have
\begin{equation} \label{eq:many avoiding w}
\card{\Psi_{\hat{w}}^t(\Gamma_p)} \geq \left(1-\eta\right)^n (pd)^{3t}(n-1)!_t (n!_t)^2,    
\end{equation}
then
\begin{equation} \label{eq:many fixing w}
\card{\Psi_w^{t+1}(\Gamma_p)} \geq \alpha (pd)^3  (n-t)^2\card{\Psi_{\hat{w}}^t(\Gamma_p)}\,,\end{equation}
where we recall that~$\Psi_{w}^t(G)$ is the set of embeddings~$\psi \in \Psi^t(G)$ for which~$\psi((1,1)) = w$, that is, the first triangle is embedded so that its first vertex is~$w$. Indeed, using \cref{lem:LDL}, we can see that if there are many embeddings of $D_t$ avoiding $w$ \eqref{eq:many avoiding w}, then for almost all choices  of further vertices $w_2\in V^2$ and $w_3\in V^3$, there will be many embeddings of $D_t$ avoiding all $3$ vertices $w,w_2,w_3$. Intuitively, \eqref{eq:many fixing w} then follows due to the fact that we can expect many of these triangles $w,w_2,w_3$ to feature in $\Gamma_p$ and each triangle that does, gives an embedding of $D_{t+1}$ which maps $w$ to a triangle. We have to be very careful with the dependence of these different random variables here and the essence of the proof of \cref{lem:embed} (which is done in \cref{sec:extendingfactors}) is to work with random variables that are independent of each other. Now  together with \cref{lem:LDL} for~$\ell=1$ and our assumption~\cref{eq:full-step1}, using the conclusion of \cref{lem:embed} (namely \eqref{eq:many fixing w}),    \cref{prop:full-step} follows readily as almost all choices of $w\in V^1$ satisfy \eqref{eq:many avoiding w}.

We will now sketch some of the ideas involved in proving \cref{lem:LDL}. To ease the discussion, let us fix~$\ell=1$ and hence~$\uu=\emptyset$; the other cases are similar. In this case our assumption~\cref{eq:LDLhypothesis} simply states that~$\Gamma_p$ has roughly the correct number of embeddings of~$D_t$ and a simple averaging argument will find some~$u=u_\ell$ for which~\cref{eq:LDL} holds with~$\uv=u$.
Fix some such vertex~$u$.
The challenge now is to show that~\cref{eq:LDL} holds for \emph{almost all} choices of~$u_\ell$.

In order to do this, we fix some typical vertex~$v\in V^1\setminus \{u\}$. We will aim to lower bound the size of~$\Psi^t_{\hat v}(\Gamma_p)$ by comparing it to the size of~$\Psi^t_{\hat u}(\Gamma_p)$. Let us suppose, momentarily,  that~$\tr_u(\Gamma_p)=\tr_v(\Gamma_p)$. In such a case, we can easily compare the sizes of~$\Psi^t_{\hat v}(\Gamma_p)$ and~$\Psi^t_{\hat u}(\Gamma_p)$. Indeed, for every embedding~$\psi\in \Psi^t_{\hat u}(\Gamma_p)$ there are  two cases.
Firstly, if~$v$ is not in a triangle in~$\psi(D_t)$ then~$\psi\in \Psi^t_{\hat v}(\Gamma_p)$ already.
Secondly, if~$v$ is in a triangle~$\{v,w_2,w_3\}$ of~$\psi(D_t)$, then~$\tr_u(\Gamma_p)=\tr_v(\Gamma_p)$ implies that~$\{u,w_2,w_3\}$ is also a triangle, hence we can
switch the triangle~$\{v,w_2,w_3\}$ with~$\{u,w_2,w_3\}$ in~$\psi$ to get an embedding~$\psi'\in \Psi^t_{\hat v}(\Gamma_p)$. This gives an injection from~$\Psi^t_{\hat u}(\Gamma_p)$ to~$\Psi^t_{\hat v}(\Gamma_p)$, proving that~$\Psi^t_{\hat v}(\Gamma_p)$ is also of roughly the `correct' size. 

Of course, the situation that~$\tr_u(\Gamma_p)=\tr_v(\Gamma_p)$ is wildly unrealistic. Let us loosen this and suppose instead  that
\begin{equation} \label{eq:simplified}\card{\tr_u(\Gamma_p)\cap\tr_v(\Gamma_p)}=\Omega( p^3n^2)\,.\end{equation}
As we expect every vertex to be in~$\Theta(p^3n^2)$ triangles,~\cref{eq:simplified} can be interpreted as saying that a constant fraction of the set of edges that form a triangle with~$v$, also form a triangle with~$u$. We can only expect this to happen when~$p$ is constant and this is also a gross oversimplification of our setting but serves to demonstrate a key idea of the proof. So for now, we take~\cref{eq:simplified} to be the case and note that as above, we can perform a switching, replacing triangles containing~$v$ with triangles containing~$u$ to map embeddings in~$\Psi^t_{\hat u}(\Gamma_p)$ to embeddings in~$\Psi^t_{\hat v}(\Gamma_p)$, whenever the embedding~$\psi\in \Psi^t_{\hat u}(\Gamma_p)$ has~$v$ in a triangle~$\{v,w_2,w_3\}$ such that~$\{w_2,w_3\}\in \tr_u(\Gamma_p)$. We have, by~\cref{eq:simplified}, that a constant proportion of the triangles containing~$v$ can be switched in this way but we do \emph{not} know that this translates to having a constant proportion of the \emph{embeddings} in~$\Psi^t_{\hat u}(\Gamma_p)$ being switchable. It could well be that almost all (or even all) of the embeddings in~$\Psi^t_{\hat u}(\Gamma_p)$ map~$v$ to a triangle~$\{v,w_2,w_3\}$ such that~$\{u,w_2,w_3\}\notin K_3(\Gamma_p)$. What we need then, is to be able to discount such a situation and show that each triangle containing~$v$ contributes to roughly the same number of embeddings~$\psi\in \Psi^t_{\hat u}(\Gamma_p)$. Put differently,  when we consider a \emph{uniformly random embedding}~$\psi^*\in \Psi^t_{\hat u}(\Gamma_p)$, we want that the random variable~$T_v$, which encodes the triangle containing~$v$ in~$\psi^*(D_t)$, induces a  roughly uniform distribution  on the set~$\tr_v(\Gamma_p)$. Note that it is possible that~$\psi^*$ leaves~$v$ isolated but this is unlikely (as~$t$ is large) and so we ignore this possibility for this discussion. 

We can now see how entropy enters the picture as it provides a tool for studying distributions, and how far they are from being uniform.
Let us now consider~$v$ as not fixed any more.
Our argument will take a uniformly random~$\psi^*\in \Psi_{\hat u}^t(\Gamma_p)$ and consider the random variables~$T_v$ which describe the triangle containing each vertex~$v\in V^1$. Due to the fact that~$\Psi_{\hat u}^t(\Gamma_p)$ is roughly the `correct' size, we have that~$\psi^*$ has large entropy. Moreover,~$\psi^*$ is completely described (up to labelling) by the set~$\{T_v:v\in V^1\}$ and so the entropy of~$\psi^*$ can be decomposed as a sum of  individual entropy values~$h(T_v)$ of the~$T_v$, using the chain rule (\cref{lem:chain rule}) for example. We will be able to use random properties of~$\Gamma_p$ (for example that no vertex is in too many triangles) to conclude that no single~$T_v$ has too large entropy. This will thus imply that for almost all vertices~$v\in V^1$, the entropy of~$T_v$ is large. Therefore, by applying \cref{lem:large-entropy}, we will be able to conclude that for a typical vertex~$v\in V^1$, the random variable~$T_v$ induces a roughly uniform distribution  on~$\tr_v(\Gamma_p)$, as desired. This idea is formalised in what we call the Entropy Lemma (\cref{lem:entropy lemma}). 

\smallskip

Our discussion above is premised on~\cref{eq:simplified}. In reality,  a typical vertex~$v$ will have~$\tr_v(\Gamma_p)$ completely disjoint from~$\tr_u(\Gamma_p)$ and so the  switching argument outlined above cannot possibly work. However, we can still compare the sizes of~$\Psi_{\hat u}^t(\Gamma_p)$ and~$\Psi_{\hat v}^t(\Gamma_p)$ by noting that a constant proportion of~$\tr_u(\Gamma_p)$ and~$\tr_v(\Gamma_p)$ are \emph{drawn from the same distribution}. By this we mean the following. For a typical~$v$, by using regularity properties, there will be~$\Omega(n^2)$ edges~$F\subset E(\Gamma)$ in the joint neighbourhood (with respect to~$\Gamma$)  of~$u$ and~$v$. Consider revealing all edges in~$\Gamma_p$ apart from those incident to~$u$ or~$v$.
After this,~$F_p:=F\cap E(\Gamma_p)$ is revealed and
whp has size~$\card{F_p}=\Omega(pn^2)$; each edge~$e\in F_p$ has the potential to land in both~$\tr_v(\Gamma_p)$ and~$\tr_u(\Gamma_p)$, depending on which random edges incident to~$u$ and~$v$ appear.

Moreover, without
having revealed the random edges incident to~$u$ or~$v$ yet, we can associate a \emph{weight} to the edges~$e$ in~$F_p$, which encodes the number of embeddings of~$D_{t-1}$ in~$\Gamma_p$, which avoid~$u$,~$v$ and the vertices of~$e$. Now, revealing the edges incident to~$v$, we have that for every~$e\in \tr_v(\Gamma_p)\cap F_p$, the probability that a uniformly random embedding~$\psi^*\in \Psi^t_{\hat u}(\Gamma_p)$ uses the triangle~$\{v\}\cup e$, is directly proportional to the weight of~$e$ in~$F_p$. The Entropy Lemma (\cref{lem:entropy lemma}) discussed above tells us that the random variable~$T_v\in \tr_v(\Gamma_p)$, encoding the triangle containing~$v$ in a uniformly random~$\psi^*\in \Psi^t_{\hat u}(\Gamma_p)$, has a roughly uniform distribution in~$\tr_v(\Gamma_p)$. From this, we can deduce that the weights of edges in~$F_p$ are `well-behaved' in that many of the edges in~$F_p$ have a sufficiently large weight. This in turn gives that~$\Psi^t_{\hat v}(\Gamma_p)$ will be large, as when we reveal the edges incident to~$u$, we can expect that~$\tr_u(\Gamma_p)$ contains many (i.e.~$\Omega(p^3n^2)$) edges  of large weight from~$F_p$. Each such edge~$e$ contributes many embeddings in~$\Psi^t_{\hat v}(\Gamma_p)$ which map~$u$ to a triangle with~$e$.

In order for all of this to work, we need our Entropy Lemma (\cref{lem:entropy lemma}) to be very strong,  due to  the fact that the edges in the~$F_p$ defined above contribute only a small fraction of edges in~$\tr_v(\Gamma_p)$.  Pushing the strength of the Entropy Lemma is one of the   main novelties of the current work, in comparison to previous arguments for triangle factors in random graphs~\cite{Allen2020+,Johansson2008}, and requires a delicate analysis. 


\section{Counting triangles in \texorpdfstring{$\Gamma_p$}{the random graph}}\label{sec:smallsubgraphs}
The purpose of this section is to prove that certain properties of~$\Gamma_p$ hold with high probability when~$\Gamma$ is a (super-)regular tripartite graph and~$p$ is sufficiently large. These properties regard triangle counts in~$\Gamma_p$ and their proofs use the properties of regular tuples given in \cref{sec:regularity} and the probabilistic tools outlined in \cref{sec:probtools}. Our first lemma gives an estimate on the number of triangles induced on vertex subsets.

\begin{lemma}\label{lem:triangle-count}
  For all~$0<\e'<d\le 1$ and~$L>0$ there exists~$\e>0$ and~$C>0$ such that  the following holds for all sufficiently large~$n \in \N$ and for any~$p \geq C(\log n)^{1/3} n^{-2/3}$.   If~$\,\Gamma$ is an~$(\e,d)$-regular tripartite graph with parts~$V^1,V^2,V^3$ of size~$n$,  then with  probability at least~$1-n^{-L}$ we have that
  \begin{equation}\label{eq:lem-triangle-count}
    \card{K_3(\Gamma_p[X_1 \cup X_2 \cup X_3])} = (p d)^3 |X_1||X_2||X_3| \pm \e' p ^3n^3,
  \end{equation}
  for all~$X_1 \subseteq V^1$,~$X_2 \subseteq V^2$ and~$X_3 \subseteq V^3$.
\end{lemma}
%
\begin{proof}
  Choose~$0<\e, \tfrac{1}{C}\ll\e',d,\tfrac{1}{L}$ and fix~$\Gamma$ and~$p\geq C(\log n)^{1/3} n^{-2/3}$. 
  We first show (a stronger version of) the lower bound holds using Janson's inequality.
  \begin{claim}
    With probability at least~$1-e^{-n}$, we have
  \begin{equation}\label{eq:triangle-count-lb}
    \card{K_3(\Gamma_p[X_1 \cup X_2 \cup X_3])} \geq (p d)^3 |X_1||X_2||X_3| - \frac{\e'p^3n^3}{8},
  \end{equation}
  for all~$X_1 \subseteq V^1$,~$X_2 \subseteq V^2$ and~$X_3 \subseteq V^3$.
  \end{claim}
  \begin{claimproof}
  Fix~$X_1 \subseteq V^1$,~$X_2 \subseteq V^2$ and~$X_3 \subseteq V^3$ and let~$Y \coloneq K_3(\Gamma[X_1 \cup X_2 \cup X_3])$. We may assume that
  \begin{equation}\label{eq:triangle-count-lb1}
    |X_1||X_2||X_3| \geq \frac{\e' n^3}{8d^3}\geq \sqrt{\e} n^3,
  \end{equation}
  with the first inequality holding as otherwise~\cref{eq:triangle-count-lb} is trivially true and the second inequality holding by our choice of~$\e$.
  In particular, we have~$|X_i| \geq \e n$ for all~$i \in [3]$ and thus \cref{lem:super-reg-triangle-count} implies~$ \card Y \geq d^3 |X_1||X_2||X_3| - 10 \e n^3.$
  Consider now the random variable
  \[X \coloneq \card{K_3(\Gamma_p[X_1 \cup X_2 \cup X_3])} = \sum_{T\in Y} I_T,\]
  where 
  for each triangle~$T \in Y$, 
~$I_{T}$ is the indicator random variable for the event that~$T$ is present in~$\Gamma_p$.  
  Let
  \begin{equation}\label{eq:lambda-lb1}
  \lambda \coloneq \Exp{X} = p^3  |Y| \geq (p d)^3|X_1||X_2||X_3| - 10\e p^3n^3,
  \end{equation}
  which in combination with~\cref{eq:triangle-count-lb1} implies~$\lambda \geq \e p^3n^3$. 
  Furthermore, we have
  \begin{equation} \label{eq:Delta2} 
  \bar \Delta \coloneq \sum_{T,T' \in Y: \ T \cap T' \not = \emptyset} \Exp{I_{T}I_{T'}}
  \leq p^5 \cdot  |Y| \cdot 3n + p^3 \cdot |Y|  
  = \lambda (3np^2 + 1),
  \end{equation}
  where the inequality  follows from the fact that there are at most~$|Y| \cdot  3n$ pairs of triangles intersecting in exactly one edge, no pairs intersecting in exactly two edges and~$|Y|$ pairs intersecting in three edges.
  Hence Janson's inequality (\cref{lem:janson}) implies
  \begin{align*}
    \Pro{X \leq (1-\e)\lambda}  \leq \exp \left(- \frac{\e^2 \lambda^2}{2\bar\Delta} \right)
    &\leq \exp \left(- \frac{\e^3 p^3n^3  \lambda}{2\bar\Delta} \right)\\
    &\leq \exp \left(-\frac{\e^3 p^3n^3 }{12np^2} \right) + \exp \left(-\frac{\e^3p^3n^3}{4} \right)\\
    &\leq \exp \left(- 4 n \right)
  \end{align*}
  for all large enough~$n$. Here, we used that~$\lambda\geq \e p^3n^3$ (see~\cref{eq:lambda-lb1}) in the second inequality, and~\cref{eq:Delta2} in the third (more precisely, we used that~\cref{eq:Delta2} implies that~$\bar\Delta \leq 6\lambda np^2$ or~$\bar\Delta \leq 2\lambda$).

  By~\cref{eq:lambda-lb1}, we have~$(1-\e) \lambda \geq (p d)^3 |X_1||X_2||X_3| - 11\e p^3n^3\geq (p d)^3 |X_1||X_2||X_3| - \big(\tfrac{\e'}{8}\big) p^3n^3$. Hence, taking a union bound over all choices of~$X_1 \subseteq V^1, X_2\subseteq V^2, X_3\subseteq V^3$, we deduce that,~\cref{eq:triangle-count-lb} holds with probability at least~$1-2^{3n}\cdot e^{-4n}\geq 1-e^{-n}$ for all~$X_1 \subseteq V^1, X_2\subseteq V^2, X_3\subseteq V^3$.
  \end{claimproof}

  We now show that the upper bound holds  in the case when~$X_i = V^i$ for all~$i \in [3]$.
  \begin{claim} \label{clm:triangle-count-ub-Vi}
    With probability at least~$1-n^{-2L}$ we have
  \begin{equation*}
    \card{K_3(\Gamma_p)} \leq (p d)^3 n^3 + \frac{\e'p^3n^3}{8}.
  \end{equation*}
  \end{claim}
  \begin{claimproof}
    Let~$Y = K_3(\Gamma)$ and let~$X = \card{K_3(\Gamma_p)} = \sum_{T \in Y} I_T$ with~$I_T$ being the indicator random variable for the event that a triangle~$T$ appears in~$\Gamma_p$, as above. 
    By \cref{lem:super-reg-triangle-count}, we have~$|Y| = d^3n^3 \pm 10\e n^3.$ 
    It follows that
    \begin{equation}\label{eq:lambda-ub}
      \lambda \coloneq \Exp{X} = (p d)^3 n^3 \pm 10\e p^3n^3.
    \end{equation}
    Using notations from the Kim--Vu inequality (\cref{lem:Kim--Vu}), we have~$E_1 \leq np^2$,~$E_2 = p$ and~$E_3 = 1$. Hence~$E' = \max \{1, np^2\} \leq \lambda^{1/2}$ and~$E = \lambda$.
    Let~$\mu = \lambda^{1/16}$ and let~$c = c(3)$ be the constant from \cref{lem:Kim--Vu}. Then, for large enough~$n$,
    \[c  (EE')^{1/2} \mu^3 \leq c \lambda^{3/4} \cdot \lambda^{3/16} \leq \e \lambda.\]
    Hence, we have
    \begin{align*}
      \Pro{X \geq (1+\e) \lambda} \leq 10cn^4  e^{-\mu} \leq e^{-n^{1/16}}\leq n^{-2L}
    \end{align*}
    for all large enough~$n$. Here, the middle  inequality follows from~\cref{eq:lambda-ub} which implies~$\lambda \geq n \log n$, due to our choice of~$\e$ and~$C$. This finishes the proof of the claim as~$(1+\e)\lambda\leq (pd)^3n^3+\big(\tfrac{\e'}{8}\big)p^3n^3$ by~\cref{eq:lambda-ub} and our choice of~$\e$.
  \end{claimproof}

  We now conclude the proof of the lemma. With probability at least~$1-n^{-L}$ both claims above hold simultaneously. Suppose now both claims hold and fix~$X_1 \subseteq V^1, X_2 \subseteq V^2, X_3 \subseteq V^3$. Let~$\cU = \left(\{X_1, V^1 \setminus X_1\} \times \{X_2, V^2 \setminus X_2\} \times \{X_3, V^3 \setminus X_3\}\right) \setminus \{(X_1,X_2,X_3)\}$ and observe that
  \begin{align*}
    \card{K_3(\Gamma_p[X_1 \cup X_2 \cup X_3])} &= \card{K_3(\Gamma_p)} -  \sum_{(U_1, U_2, U_3) \in \cU} \card{K_3(\Gamma_p[U_1 \cup U_2 \cup U_3])} \\
    &\leq (p d)^3 |X_1||X_2||X_3| + \e' p^3n^3.
  \end{align*}
  Here we used \cref{clm:triangle-count-ub-Vi} to bound~$\card{K_3(\Gamma_p)}$ and~\cref{eq:triangle-count-lb} to bound each~$\card{K_3(\Gamma_p[U_1 \cup U_2 \cup U_3])}$.
  This completes the proof.
\end{proof}

As a corollary, we can conclude that we have the expected count of triangles at almost all vertices.

\begin{cor} \label{cor:fixed vertex triangles}
For all~$0<\e'<d\le 1$ and~$L>0$ there exists~$\e>0$ and~$C>0$ such that  the following holds for all sufficiently large~$n \in \N$ and for any~$p \geq C(\log n)^{1/3} n^{-2/3}$.   If~$\,\Gamma$ is an~$(\e,d)$-regular tripartite graph with parts of size~$n$,  then with  probability at least~$1-n^{-L}$ we have that
\[\card{\tr_v(\Gamma_p)}=(1 \pm \e')(p d)^3n^2,\]
for all but at most~$\e' n$ vertices~$v \in V(\Gamma)$.
\end{cor}
\begin{proof}
Choose~$0<\e,\tfrac{1}{C}\ll\tilde \e \ll \e',d,\tfrac{1}{L}$ and let~$G \subseteq \Gamma$ be any graph with
\begin{equation}\label{eq:cor:triangle-count}
    \card{K_3(G[X_1 \cup X_2 \cup X_3])} = (p d)^3 |X_1||X_2||X_3| \pm \tilde \e p^3n^3,
\end{equation}
for all~$X_1 \subseteq V^1$,~$X_2 \subseteq V^2$ and~$X_3 \subseteq V^3$.
Since (by \cref{lem:triangle-count} and our choice of constants) this is satisfied by~$\Gamma_p$ with probability~$1-n^{-L}$, it suffices to show that~$G$ satisfies the conclusion of \cref{cor:fixed vertex triangles}.
For~$i \in [3]$, let~$X_i$ be the set of vertices~$v \in V^i$ with~$\card{\tr_v(G)} \leq (1-\e') (p d)^3 n^2$, and  let~$Y_i$ be the set of vertices~$v \in V^i$ with~$\card{\tr_v(G)} \geq (1+\e') (p d)^3 n^2$.
We claim that~$|X_1| \leq \tfrac{\e' n}{10}$. Indeed, assuming the contrary, we have
\[ \card{K_3(G[X_1 \cup V^2 \cup V^3])} \leq (p d)^3|X_1||V^2||V^3| - \frac{\e'^2 (pd)^3n^3}{10}<(p d)^3|X_1||V^2||V^3| - \tilde\e p^3n^3,\]
by our choice of~$\tilde \e$. This contradicts~\cref{eq:cor:triangle-count}. Similarly, we can bound the sizes of~$X_2$ and~$X_3$, and~$Y_1$,~$Y_2$ and~$Y_3$, completing the proof.
\end{proof}
Sometimes, we will need an upper bound on~$\card{\tr_v(\Gamma_p)}$ which works for all~$v \in V(\Gamma)$. For this we simply upper bound this quantity by the number of triangles in~$G(3n,p)$ containing a specific vertex using a result of Spencer~\cite{Spencer1990} (see also~\cite{vsileikis2019counting}).
%
\begin{lemma}\label{lem:upper triangle count for all vxs}
For  all~$L>0$ there exists~$C>0$ such that  the following holds for all sufficiently large~$n \in \N$ and for any~$p \geq C(\log n)^{1/3} n^{-2/3}$.   If~$\,\Gamma$ is a  tripartite graph with parts of size~$n$,  then with  probability at least~$1-n^{-L}$ we have that
  \[\card{\tr_v(\Gamma_p)}\leq 10p^3n^2,\]
  for all vertices~$v \in V(\Gamma)$.
\end{lemma}
In the remainder of this section we prove some more technical properties of~$\Gamma_p$ which will be  useful in the proofs of \cref{prop:full-step,lem:LDL}. The ultimate goal will be to lower bound the number of triangles at a fixed vertex but we will need  this lower bound to hold in a robust way, allowing us to apply the count with respect to various prescribed sets of edges and vertices which we either want to avoid or want to be included in  the triangles. 

Our next lemma follows simply from well-known concentration bounds but we wish to highlight the slightly subtle (in-)dependencies of the random variables involved. Recall that, given a vertex~$u$ of our graph~$\Gamma$, by saying that a random variable is \emph{determined by}~$(\Gamma_{\hat u})_p$, we mean that the random variable is completely determined by revealing~$(\Gamma_{\hat u})_p$.
In other words, the random variable is independent of the status of edges adjacent to~$u$ in~$\Gamma_p$. We will now use this concept with the random variable being a vertex set or an edge set.

\begin{lemma} \label{lem:two stage reveal}
   For any~$0<\alpha\le 1$ and~$L>0$, there exists a~$C>0$ such that  the following holds for all sufficiently large~$n \in \N$ and for any~$p \geq C(\log n)^{1/3} n^{-2/3}$. Suppose~$\Gamma$ is a tripartite graph with parts of size~$n$ and~$u\in V(\Gamma)$. Then  we have the following. 
\begin{enumerate}
    \item \label{lem:many neighbours in random set}
  Suppose~$X\subseteq N_\Gamma(u)$ is a random subset of vertices determined by~$(\Gamma_{\hat u})_p$. 
  Then with probability at least~$1-n^{-L}$ we have that the following statement holds in~$\Gamma_p$. 
  \begin{center}
      If~$\card{X}\geq \alpha n$ then~$\card{X\cap N_{\Gamma_p}(u)}\geq  \tfrac{\alpha pn}{2}$.   \end{center}
  \item \label{lem:triangles with F} 
  Suppose~$F\subseteq \tr_u(\Gamma)\cap E(\Gamma_p)$ is a random subset of edges determined by~$(\Gamma_{\hat u})_p$. Then with probability at least~$1-n^{-L}$ we have that the following statement holds in~$\Gamma_p$. 
  \begin{center}
  If~$|F|\geq \alpha pn^2$ then~$\card{F\cap \tr_u(\Gamma_p)}\geq \tfrac{\alpha p^3n^2}{2}$.   \end{center}

\end{enumerate}

\end{lemma}

\begin{proof}
Choose~$\tfrac{1}{C}\ll \tfrac{1}{L},\alpha$. Let~$G_1\subset \Gamma_p$ be the graph on~$V(\Gamma)$ consisting of all edges adjacent to~$u$ and~$G_2=(\Gamma_{\hat u})_p=\Gamma_p\setminus G_1$. For all~$w\in N_\Gamma(u)$, let~$I_w$ be the indicator random variable for the event that the edge~$uw$ appears. By assumption, our random sets~$X$ and~$F$ depend only on~$G_2$ and clearly the random variables~$I_w$ depend only on~$G_1$.

Part \cref{lem:many neighbours in random set} now follows  from Chernoff's inequality (Theorem~\ref{thm:chernoff}). Indeed  we have that 
\[\Pro{\card{X\cap N_{\Gamma_p}(u)}< \frac{\alpha pn}{2} \mbox{ and } \card{X}\geq \alpha n} \leq \Pro{\card{X\cap N_{\Gamma_p}(u)}< \frac{\alpha pn}{2}\, \middle| \, \card{X}\geq \alpha n}\]
and it suffices to show that~$\Pro{\card{X\cap N_{\Gamma_p}(u)}< \tfrac{\alpha pn}{2}}\leq n^{-L}$ holds for any instance of~$G_2$ and~$X$ with~$|X|\geq \alpha n$. Fixing such an instance and  letting~$Y=\card{X\cap N_{\Gamma_p}(u)}=\sum_{w\in X}I_w$, we have that~$Y$ is a sum of independent random variables with expectation~$\lambda=\Exp{Y}=p\card{X}$  and so \[\Pro{Y<\frac{\alpha pn}{2}}\leq \Pro{Y<\frac{\lambda}{2}}\leq e^{-\lambda/8}\leq  e^{-\alpha pn/8}\leq  n^{-L}, \]
for sufficiently large~$n$, as required. 

For part \cref{lem:triangles with F}, we start by noting that~$\Delta(G_2) \leq 4pn$ with probability at least~$1-n^{-2L}$ by another simple application of Chernoff's bound (\cref{thm:chernoff}) and a union bound over all vertices. We have that
\begin{align*}
    &\Pro{\card{F\cap \tr_u(\Gamma_p)}< \frac{\alpha p^3n^2}{2} \mbox{ and } |F|\geq \alpha pn^2} \leq \\
     &  \Pro{\card{F\cap \tr_u(\Gamma_p)}< \frac{\alpha p^3n^2}{2},  |F|\geq \alpha pn^2 \mbox{ and } \Delta(G_2) \leq 4pn } +\Pro{\Delta(G_2)>4pn} \leq \\
     &   \Pro{\card{F\cap \tr_u(\Gamma_p)}< \frac{\alpha p^3n^2}{2} \, \middle| \, |F|\geq \alpha pn^2 \mbox{ and } \Delta(G_2) \leq 4pn} + n^{-2L}.
\end{align*}
Thus it suffices to prove that~$\Pro{\card{F\cap \tr_u(\Gamma_p)}< \tfrac{\alpha p^3n^2}{2} }\leq n^{-2L}$  
for any instance of~$G_2$ 
such  that
$\Delta(G_2) \leq 4pn$ and~$\card{F}\geq \alpha pn^2$. So let us fix such an instance of~$G_2$ and~$F\subseteq \tr_u(\Gamma)$. 
Let~$\cF = \{\{uw_1, uw_2\}: w_1w_2 \in F\}$ and for~$A=\{uw_1, uw_2\}\in \cF$, let~$I_A=I_{w_1}I_{w_2}$ be the indicator random variable for the event that both edges of~$A$ appear in~$G_1$. We will now use Janson's inequality to show that many pairs of edges in~$\cF$ are present in~$G_1$. 
Let
\[Z = \card{F \cap \tr_u(\Gamma_p)} = \sum_{A \in \cF} I_A\]
be the random variable counting the number of triangles containing~$u$ and an edge in~$F$. 
Then 
\begin{equation} \label{eq:triF-lambda} 
\lambda \coloneq \Exp{Z} = p^2 \card{\cF} 
    \geq \alpha p^3n^2 
        \geq C^2 \log n.
\end{equation}
Furthermore, we have that
\begin{equation} \label{eq:triF-Delta3}
\bar\Delta  \coloneq \sum_{(A,A') \in \cF^2: \ A \cap A' \not = \emptyset} \Exp{I_A I_{A'}}
            \leq 8p^4\card{\cF}n + p^2\card{\cF} 
            = \lambda (1 + 8p^2n).
\end{equation}
Here, the inequality follows from the fact that there are at most~$\card{\cF} \cdot 2\cdot  \Delta(G_2) = \card{\cF} \cdot 8pn$ pairs~$(A,A') \in \cF^2$ intersecting in exactly one edge, and~$\card{\cF}$ pairs intersecting in two edges.
Hence Janson's inequality (\cref{lem:janson}) implies
\begin{align*}
  \Pro{Z \leq  \frac{\lambda}{2}}
    \leq \exp \left(- \frac{\lambda^2}{8\bar\Delta} \right)
    &\leq \exp \left(- \frac{\lambda}{8(1+8p^2n)} \right) \\
    &\leq \exp \left(- \frac{\lambda}{16} \right) + \exp \left(-\frac{\lambda}{128p^2n} \right) \\
    &\leq n^{-C} + e^{-n^{1/3}}\leq n^{-2L},
\end{align*}
for all large enough~$n$.
Here, we used~\cref{eq:triF-Delta3} in the second inequality, the fact that~$1 + 8pn^2 \leq 2$ or~$1 + 8pn^2 \leq 16pn^2$ in the third,~\cref{eq:triF-lambda} in the fourth  and our choice of~$C$ in the final inequality. This completes the proof. 
\end{proof}

Finally, we  show that for most pairs of vertices~$u$ and~$v$ in the same part, there are many edges appearing in~$\Gamma_p$ that lie in their common neighbourhood (with respect to~$\Gamma$). We need this to hold even when we forbid certain vertices from being used. This leads to the following statement, for which  we direct the reader to \cref{sec:notation} for the relevant definitions of e.g.~$\pzc{V}$ and~$\tr_u(G)$. 
\begin{lemma} \label{lem:many common neighbours}
 For all~$0<d\le 1$ there exists~$\e>0$ and~$C>0$ such that  the following holds for all sufficiently large~$n \in \N$ and for any~$p \geq C(\log n)^{1/3} n^{-2/3}$.   If~$\,\Gamma$ is an~$(\e,d)$-super-regular tripartite graph with parts~$V^1,V^2,V^3$ of size~$n$,~$\ell\in[3]$,~$\uu=(u_1,\ldots,u_{\ell-1})\in \pzc{V}$ and~$u\in V^\ell$ then with  probability at least~$1-e^{-n}$ we have that
\[\card{\tr_u(\Gamma_{\hat \uu})\cap \tr_v(\Gamma_{\hat \uu})\cap E(\Gamma_p)} \geq \frac{d^5pn^2}{4},\]
for all but at most~$2\e n$ vertices~$v\in V^\ell$.
\end{lemma}
\begin{proof}
  Choose~$0<\e,\tfrac{1}{C} \ll d$  and fix~$\Gamma$,~$\ell\in[3]$,~$\uu=(u_1,\ldots,u_{\ell-1})\in \pzc{V}$ and~$u\in V^\ell$ as in the statement of the lemma.  
  We first use regularity to show that there are many edges in the deterministic graph.
  \begin{clm*}
    We have
    \begin{equation*}
    \card{\tr_u(\Gamma_{\hat \uu})\cap \tr_v(\Gamma_{\hat \uu})} \geq \frac{d^5n^2}{2},
    \end{equation*}
    for all but at most~$2\e n$ vertices~$v \in V^\ell$.
  \end{clm*}
  \begin{claimproof}
      We will prove the claim in the case that~$\ell = 3$, the other cases are  identical. For~$i \in [2]$, let~$X_i = N_\Gamma\big(u; V_{\hat \uu}^i\big)$ and for~$v \in V^3 \setminus \{u\}$, let~$Y_i(v) = N_\Gamma\big(u,v; V_{\hat \uu}^i\big) \subseteq X_i$. Since~$\Gamma$ is~$(\e,d)$-super-regular, we have~$|X_i| \geq (d-2\e)n$ for both~$i \in [2]$ (we need the factor of~$2$ in front of the~$\e$ here to take account of the fact that we are potentially missing a vertex in~$\uu$).
      For~$i \in [2]$, let~$R_i \subset V^3$ be  the set of vertices~$v \in V^3$ for which~$|Y_i(v)| < (d-2\e)^2n$ and let~$R = R_1 \cup R_2$.
      It follows from the~$\e$-regularity of~$(V^i,V^3)$ and \cref{lem:reg-degree}, that~$|R_i| \leq \e n$ for both~$i \in [2]$ and hence~$|R| \leq 2 \e n$.
      Furthermore, for every~$v \in V^3 \setminus R$, it follows from the~$\e$-regularity of the pair~$(V^2,V^3)$ that
~$\card{E(\Gamma)\cap(Y_1(v) \cup Y_2(v))} \geq (d-2\e)^5 n^2$. This completes the proof by our choice of~$\e$.
  \end{claimproof}

  Observe now that each edge in~$E({\Gamma})\cap N_{\Gamma_{\hat{\uu}}}(u,v)=\tr_u(\Gamma_{\hat \uu})\cap \tr_v(\Gamma_{\hat \uu})$ is present independently in~$\Gamma_p$ and hence it follows from Chernoff's inequality (\cref{thm:chernoff}) that for all vertices~$v$ satisfying the conclusion of the claim, we have that 
  \[\Pro{\card{\tr_u(\Gamma_{\hat \uu})\cap \tr_v(\Gamma_{\hat \uu})\cap E(\Gamma_p)} < \frac{d^5pn^2}{4}}\leq \exp\left(-\frac{d^5pn^2}{16}\right)\leq e^{-2n},\]
  for sufficiently large $n$. This completes the proof after a union bound over choices of~$v\in V^\ell$.
  \end{proof}

\section{Embedding (partial) triangle factors}\label{sec:otherlemmas}

In this section we will prove \cref{prop:almost-factor} and reduce \cref{prop:full-step} to \cref{lem:LDL}. As we have already shown in \cref{sec:proofoverview} that \cref{thm:main-super-reg} follows from \cref{prop:almost-factor,prop:full-step}, after this section the only tool used in the proof of
\cref{thm:main-super-reg} that still needs to be established is \cref{lem:LDL}.

\subsection{Counting almost triangle factors} \label{sec:almostfactors}

Here we prove \cref{prop:almost-factor}.
\begin{proof}[Proof of \cref{prop:almost-factor}]
Choose~$\e,\tfrac{1}{C}\ll\e' \ll \eta,d$ and fix some~$\Gamma$ and~$p$ as in the statement of the proposition.
  By \cref{lem:triangle-count}, we have whp that 
  \begin{equation}\label{eq:triangle-count}
    \card{K_3(\Gamma_p[X_1 \cup X_2 \cup X_3])} = (p d)^3 |X_1||X_2||X_3| \pm \e' p^3n^3,
  \end{equation}
  for all~$X_1 \subseteq V^1$,~$X_2 \subseteq V^2$ and~$X_3 \subseteq V^3$. We will show by induction on~$t$ that if~$\Gamma_p$ satisfies~\cref{eq:triangle-count}, then it satisfies
  \begin{equation}\label{eq:factor-count}
  \card{\Psi^{t}(\Gamma_p)} \geq (1-\eta)^{t} (p d)^{3t} \left(n!_t\right)^3,
  \end{equation}
  for all integers~$t \leq (1-\eta)n$, as claimed.
Firstly, note that~\cref{eq:factor-count} is trivial for~$t=0$, recalling that by definition~$n!_0=1$. 
  Suppose now~\cref{eq:factor-count} holds for some integer~$0 \leq t \leq (1-\eta)n$. Fix some~$\psi \in \Psi^t(\Gamma_p)$ and let~$X_i \subseteq V^i$,~$i \in [3]$, be the sets of vertices which are not in~$\psi(D_t)$. Note that~$\card{X_i} = n-t$ for all~$i \in [3]$.
  Now the number of triangles which extend~$\psi$ to an embedding in~$\Psi^{t+1}(\Gamma_p)$ is precisely~$\card{K_3(\Gamma_p[X_1 \cup X_2 \cup X_3])}$ and by~\eqref{eq:triangle-count}, we have
  \begin{align*}
    \card{K_3(\Gamma_p[X_1 \cup X_2 \cup X_3])}
    &\geq (p d)^3 |X_1||X_2||X_3| - \e' p^3n^3\\
    &\geq (p d)^3 (n-t)^3 - \frac{\e'}{\eta^3d^3}  (p d)^3(n-t)^3\\
    &\geq (1 - \eta) (p d)^3(n-t)^3,
  \end{align*}
  by our choice of constants. 
  It follows from the induction hypothesis that
  \begin{align*}
    \card{\Psi^{t+1}(\Gamma_p)}
    &\geq \card{\Psi^{t}(\Gamma_p)}  (1 - \eta) (p d)^3(n-t)^3\\
    &\geq (1-\eta)^{t+1} (p d)^{3(t+1)} \left(n!_{t+1}\right)^3,
  \end{align*}
  finishing the proof.
\end{proof}
\subsection{Extending almost triangle factors}\label{sec:extendingfactors}
In this subsection, we will prove \cref{prop:full-step} using the Local Distribution Lemma (see \cref{lem:LDL}) as a black box for now. We first reduce \cref{prop:full-step} to the following lemma, which concentrates on adding a triangle at a fixed vertex. Recall that given~$G \subseteq \Gamma$, a vertex~$v \in V^1$ and some~$t \in \N$, we denote by~$\Psi_{v}^t(G) \subseteq \Psi^t(G)$ the set of embeddings~$\psi \in \Psi^t(G)$ for which~$\psi((1,1)) = v$.
\begin{lemma}[adding a triangle at a fixed vertex]\label{lem:embed}
	 For all~$0<d\le 1$ there exists~$\alpha,\eta,\e>0$ and~$C>0$ such that  for all sufficiently large~$n \in \N$ and for any~$p \geq C(\log n)^{1/3} n^{-2/3}$,  if~$\,\Gamma$ is an~$(\e,d)$-super-regular tripartite graph with parts of size~$n$, then whp the following holds in~$\Gamma_p$ for all~$t\in  \N$ with~$\left(1-\eta\right)n \leq t < n$ and   for all~$v \in V^1$.
 If
  \[\card{\Psi_{\hat{v}}^t(\Gamma_p)} \geq \left(1-\eta\right)^n (pd)^{3t}(n-1)!_t (n!_t)^2,\]
  then
   \[\card{\Psi_v^{t+1}(\Gamma_p)} \geq \alpha (pd)^3  (n-t)^2\card{\Psi_{\hat{v}}^t(\Gamma_p)}.\]
\end{lemma}
We first show how \cref{prop:full-step} follows  from this  and Lemma \ref{lem:LDL}.
\begin{proof}[Proof of \cref{prop:full-step}]
Choose~$0<\e,\tfrac{1}{C}\ll \eta \ll \eta'\ll \alpha \ll \alpha'\ll d$. Now by our choice of constants (also choosing~$K\geq 5$) and taking a union bound over all choices of~$t$ with~$\left(1-\eta\right)n \leq t < n$,~$\ell\in [3]$ and~$\uu=(u_1,\ldots,u_{\ell-1})\in \pzc{V}$  we have whp that the conclusion of \cref{lem:LDL} holds in~$\Gamma_p$ for all such choices and also the conclusion of \cref{lem:embed} holds with~$\eta'$ and~$\alpha'$ replacing~$\eta$ and~$\alpha$.  We will now show that given these conclusions hold in~$\Gamma_p$, we have the desired  statement of \cref{prop:full-step}. So fix some~$t\in \N$ with~$\left(1-\eta\right)n \leq t < n$ and suppose that 
 \[\card{\Psi^t(\Gamma_p)} \geq \left( 1-\eta\right)^n (pd)^{3t}(n!_t)^3.\]
  Let~$U_1 \subseteq V^1$ be the set of vertices~$u_1 \in V^1$ for which
  \begin{equation} \label{eq:proof full step 1}
      \card{\Psi_{\hat{u}_1}^t(\Gamma_p)} \geq \left(\frac{d}{10}\right)^2  \left(\frac{n-t}{n}\right)\card{\Psi^t(\Gamma_p)}.  \end{equation}
  It follows from (the assumed conclusion of)   \cref{lem:LDL} (with~$\ell = 1$) that we have~$|U_1| \geq \tfrac{n}{2}$.
  Now as
  \[\left(\frac{d}{10}\right)^2  \left(\frac{n-t}{n}\right)\card{\Psi^t(\Gamma_p)}\geq \left(\frac{d}{10}\right)^2  \left( 1-\eta\right)^n (pd)^{3t}(n-1)!_{t}(n!_t)^2, \]
  and~$\left(\frac{d}{10}\right)^2  \left( 1-\eta\right)^n\geq (1-\eta')^n$, 
 we have that  \[\card{\Psi_{u_1}^{t+1}(\Gamma_p)} \geq
   \alpha' \left(\frac{d}{10}\right)^2 (pd)^3  \frac{(n-t)^3}{n}\card{\Psi^t(\Gamma_p)},\]
  for every~$u_1 \in U_1$, from (the assumed conclusion of) \cref{lem:embed} and~\cref{eq:proof full step 1}.
  Therefore, we have that 
  \begin{align*}
    \card{\Psi^{t+1}(G)}
    &\geq \sum_{u_1 \in U_1} \card{\Psi_{u_1}^{t+1}(\Gamma_p)}\\
    &\geq   \frac{\alpha'}{2} \left(\frac{d}{10}\right)^2 (pd)^3  (n-t)^3\card{\Psi^t(\Gamma_p)} \\
    &\geq \alpha (pd)^3  (n-t)^3\card{\Psi^t(\Gamma_p)}, 
  \end{align*}
  by our choice of constants. This finishes the proof.
\end{proof}
It remains to prove \cref{lem:embed}. Before embarking on this, we sketch some of the key ideas involved. For this discussion, we fix some~$t\in [n]$ and~$v \in V^1$ that we think of as  satisfying the conditions of \cref{lem:embed} (including the ``if'' statement). We will say that a pair~$(w_2,w_3)\in V^2\times V^3$ is \emph{good} if 
	\[\card{ \Psi_{\hat{\uw}}^{t}(\Gamma_p)} =\Omega\left( \left(\frac{n-t}{n}\right)^2\card{\Psi_{\hat{v}}^t(\Gamma_p)}\right),\]
 where~$\uw=(v,w_2,w_3)$.  Note that we can appeal to the Local Distribution Lemma (\cref{lem:LDL}) twice (once with~$\ell=2$ and once with~$\ell=3$) to conclude that almost all pairs~$(w_2,w_3)\in V^2\times V^3$ are good. That is, for almost all choices of~$(w_2,w_3)\in V^2\times V^3$, we have that there are roughly the `correct' number of embeddings of~$D_t$ that avoid~$\uw=(v,w_2,w_3)$. Moreover, due to~$\Gamma$ being super-regular, there will be some proportion of these~$(w_2,w_3)$  (say, at least~$\tfrac{1}{2}d^3n^2$) that form triangles with~$v$ in~$\Gamma$. So we have some set~$W\subset V^2\times V^3$ of size at least~$\tfrac{1}{2}d^3n^2$ such that  all~$(w_2,w_3)\in W$ are good and have that~$\{v,w_2,w_3\}\in K_3(\Gamma)$. The conclusion of \cref{lem:embed} will then follow if we can prove that at least, say,~$\tfrac{p^3}{2}|W|$  triangles~$\{v,w_2,w_3\}$ with~$(w_2,w_3)\in W$, appear in~$\Gamma_p$. Of course, every triangle in~$\Gamma$ appears in~$\Gamma_p$ with probability~$p^3$ and so this is something we can expect to be true but we cannot appeal to standard tools to prove this. 
	
	The issue here is that~$W$ itself is a random set as the property of \emph{being good} depends on the random edges that appear in~$\Gamma_p$.  Indeed, in order to determine whether an edge $(w_2,w_3)\in V^2\times V^3$ is good or not, we need to count the number of embeddings of $D_t$ in $\Gamma_p$ that avoid $\uw=(v,w_2,w_3)$ and so certainly need to know the random status of edges in $\Gamma_p$ to carry out this count. However, crucially,~$W$ does not depend on \emph{all} the random edges. Indeed, for \emph{any}~$(w_2,w_3)\in V^2\times V^3$, we can determine whether~$(w_2,w_3)$ is in our set~$W$ without knowing the random status of edges adjacent to~$v$. Indeed, as the property of being good only depends on counting embeddings that \emph{avoid} $v$, the random status of edges adjacent to $v$ has no bearing on whether an edge $(w_2,w_3)\in V^2\times V^3$ is good or not. Therefore, by appealing to a two-stage revealing process (see~\cref{lem:two stage reveal}\cref{lem:triangles with F}), we will be able to prove \cref{lem:embed} if we know that at least, say,~$\tfrac{p}{2}|W|$ of the pairs~$(w_2,w_3)\in W$ host edges in~$\Gamma_p$, as then we will be able to conclude that roughly a~${p^2}$ proportion of these edges in $W\cap E(\Gamma_p)$ extend to triangles with $v$ in $\Gamma_p$. 
	
	Again, requiring that~$\tfrac{p}{2}|W|$ edges in $W$ appear in $\Gamma_p$ is certainly a natural thing to expect as each edge appears with probability $p$,  but again the set~$W$ containing good edges, depends heavily on the random status of edges in~$\Gamma[V^2,V^3]$.  Our aim is to use a two-stage revealing process, manipulating independence, as above. Again here, it is crucial that we are counting embeddings that avoid vertices. That is, if~$e=\{w_2,w_3\}\in E(\Gamma[V^2,V^3])$, then in order to determine the number of embeddings that avoid~$(v,w_2,w_3)$, we do not need to know the random status of $e$ and in fact more is true. The number of embeddings of $D_t$ avoiding~$(v,w_2,w_3)$ is independent of the random status of \emph{all}~$(w_2,u_3)$ with~$u_3\in N_\Gamma(w_2;V^3)$. Therefore our approach is to lower bound the number of edges in $|W\cap \Gamma_p|$  by grouping together edges in $W$ according to their~$V^2$-endpoint. This gives hope to use a two-stage random revealing argument (appealing to  \cref{lem:two stage reveal}\cref{lem:many neighbours in random set}) to conclude that  roughly the expected number of good edges appear in~$\Gamma_p$.
	
	However, there is an oversight in the discussion above. The point is that our definition of whether an edge $(w_2,w_3)\in V^2\times V^3$ is good does not only rely on counting embeddings avoiding $\uw=(v,w_2,w_3)$, we \emph{also} need to know the size of~$\card{\Psi_{\hat{v}}^t(\Gamma_p)}$. Therefore, if~$e=\{w_2,w_3\}\in E(\Gamma[V^2,V^3])$, then in order to determine if~$(w_2,w_3)$ is \emph{good}, we actually need to reveal the random status of~$e$ itself as well as all the random edges between $V^2$ and $V^3$ (to determine~$\card{\Psi_{\hat{v}}^t(\Gamma_p)}$). To remedy this, we adjust our definition of good to be independent of~$\card{\Psi_{\hat{v}}^t(\Gamma_p)}$. We will therefore give a \emph{grading} of the possible range of~$\card{\Psi_{\hat{v}}^t(\Gamma_p)}$ and show that the desired conclusion holds with respect to each \emph{grade} (see \cref{clm:graded} in the proof). In order to be able to perform a union bound over all of the possible grades, we need an upper bound on how large~$\card{\Psi_{\hat{v}}^t(\Gamma_p)}$ can be (whp) and this is provided by \cref{clm:max-embed}. This idea allows us to remove~$\card{\Psi_{\hat{v}}^t(\Gamma_p)}$ from the definition of being good, leading to the definition of being \emph{sound} in the proof. Hence we have that for~$e=\{w_2,w_3\}\in \Gamma[V^2,V^3]$, whether~$(w_2,w_3)$ is sound or not relies only on counting embeddings avoiding $\uw=(v,w_2,w_3)$ and so is independent of whether~$e$ appears in~$\Gamma_p$ and  in fact, as sketched above, the `soundness' of~$(w_2,w_3)$ is independent of the random status of \emph{all}~$(w_2,u_3)$ with~$u_3\in N_\Gamma(w_2;V^3)$. We therefore consider potential triangles one vertex at a time and we refine our definition of sound to handle this, leading to the definition of \emph{sound tuples}  in the proof. We now give the full details of the proof of \cref{lem:embed}. 

\begin{proof}[Proof of \cref{lem:embed}]
Choose $K=L=10$ and $0\ll\e,\tfrac{1}{C}\ll   \eta \ll \eta' \ll \alpha \ll d,\frac1K,\frac1L$ and fix~$p$ and~$\Gamma$ as in the statement of the lemma. We begin by showing the following simple claim which gives a weak upper bound on the number of embeddings that avoid a fixed vertex~$v_1$.
\begin{claim} \label{clm:max-embed}
We have that the following statement holds whp in~$\Gamma_p$. For any~$t\in \N$ such that~$\left(1-\eta\right)n \leq t < n$ and~$v_1\in V^1$,  we have that 
\begin{equation} \label{eq:max-embed-upper}
    \card{\Psi_{\hat{v}_1}^t(\Gamma_p)} \leq n^3 p^{3t}(n-1)!_t (n!_t)^2. \end{equation}
\end{claim}

\begin{claimproof}
Fix some~$t\in \N$ and~$v_1\in V^1$ as in the statement of the claim. Then 
\[ \card{\Psi_{\hat{v}_1}^t(\Gamma)} \leq  \card{\Psi_{\hat{v}_1}^t(K_{n,n,n})} \leq  (n-1)!_t (n!_t)^2,\]
and so, as each embedding of~$D_t$ in~$\Gamma$ appears in~$\Gamma_p$ with probability~$p^{3t}$, we have that~$\lambda:=\Exp{ \card{\Psi_{\hat{v}_1}^t(\Gamma_p)}}\leq p^{3t}(n-1)!_t (n!_t)^2$. Therefore, appealing to Markov's inequality gives that 
\begin{align*}
    \Pro{ \card{\Psi_{\hat{v}_1}^t(\Gamma_p)} > n^3 p^{3t}(n-1)!_t (n!_t)^2} \leq \Pro{ \card{\Psi_{\hat{v}_1}^t(\Gamma_p)}>n^3\lambda} \leq \frac{1}{n^3}. 
\end{align*}
Taking a union bound over the choices of~$v\in V^1$ and~$t\in \N$ with~$(1-\eta) n\leq t\leq n$ completes the proof of the claim.
\end{claimproof}

\cref{clm:max-embed} gives us an upper bound on the size of~$\Psi_{\hat{v}_1}^t(\Gamma_p)$ that holds whp, whilst the statement of the lemma gives a lower bound. Our next claim replaces the lower bound in the statement of the lemma, with  lower bounds independent of~$\card{\Psi_{\hat{v}_1}^t(\Gamma_p)}$. These lower bounds will depend on a parameter~$s\in \Z$ and we make the following definitions which will define the range of~$s$ we are interested in. Firstly let~$s_0$ be the largest (negative) $s\in \Z$ such that~$2^s\leq (1-\eta)^n$.  Further, let~$s_1$ be the minimum  integer~$s\in \N$ such that~$2^sd^{3t}\geq n^3$. So we have that  
\[s_0\geq n\frac{\log\left(1-\eta\right)}{\log 2} -1 \geq -n  \qquad  \mbox{ and } \qquad s_1\leq\frac{3\log n-3t\log(d)}{\log 2} +1\leq \frac{n}{\alpha}.\]
Finally, let~$\mathds{S}:=\{s\in \Z:s_0\leq s\leq s_1\}$. We now state our second claim.

\begin{claim}  \label{clm:graded}
For any~$t\in \N$ with~$\left(1-\eta\right)n \leq t < n$,~$s\in \mathds{S}$ and~$v_1\in V^1$, with probability at least~$1-n^{-4}$, the following statement holds in~$\Gamma_p$. 
If
\begin{equation} \label{eq:s_dependent_lower_bd_1}
    \card{\Psi_{\hat{v}_1}^t(\Gamma_p)} \geq 2^s (pd)^{3t}(n-1)!_t (n!_t)^2, \end{equation}
then 
\begin{equation*}
\card{\Psi_{v_1}^t(\Gamma_p)} \geq 2^{s+1}\alpha (pd)^{3(t+1)}(n-1)!_t (n!_{t+1})^2.\end{equation*}
\end{claim}

Before proving \cref{clm:graded}, we show how the lemma follows from the two claims.  Taking a union bound, we can conclude that whp the conclusion of \cref{clm:graded} holds for all choices of~$t,s$ and~$v_1$ (noting that~$|\mathds{S}|\leq (1+\alpha^{-1})n$), as well as the conclusion of \cref{clm:max-embed}.   Now suppose  that this is the case  and let~$t\in \N$ with~$\left(1-\eta\right)n \leq t < n$ and~$v\in V^1$. If, as in the assumption of \cref{lem:embed}, we have
\begin{equation} \label{eq:min-grade}
  \card{\Psi_{\hat{v}}^t(\Gamma_p)} \geq \left(1-\eta\right)^n (pd)^{3t}(n-1)!_t (n!_t)^2,
\end{equation}
then, letting~$s^*\in \Z$ be the maximum integer~$s\in \Z$ such that \[\card{\Psi_{\hat{v}}^t(\Gamma_p)} \geq 2^s (pd)^{3t}(n-1)!_t (n!_t)^2,\]
we conclude
from~\cref{eq:min-grade} that $s^*\ge s_0$ and
from (the assumed conclusion of) \cref{clm:max-embed} that $s^*\le s_1$ and hence $s^*\in \mathds{S}$. Therefore, from (the assumed conclusion of) \cref{clm:graded}, we obtain that 
\begin{align*}\card{\Psi_{v}^t(\Gamma_p)} \geq 2^{s^*+1}\alpha (pd)^{3(t+1)}(n-1)!_t (n!_{t+1})^2 
\geq \alpha (pd)^{3} (n-t)^2\card{\Psi_{\hat{v}}^t(\Gamma_p)}, 
\end{align*}
as required for the conclusion of \cref{lem:embed}, where we used that
\[\card{\Psi_{\hat{v}}^t(\Gamma_p)}\leq 2^{s^*+1} (pd)^{3t}(n-1)!_t (n!_t)^2\,,\]
by the definition of~$s^*$. Thus it remains to prove \cref{clm:graded}.

\begin{claimproof}[Proof of \cref{clm:graded}]
Let us fix~$t\in \N$ with $\left(1-\eta\right)n \leq t < n$, $s\in \mathds{S}$ and~$v_1\in V^1$.  
  Given some~$\ell \in [3]$, we call a sequence of vertices~$\uu = (u_1, \ldots, u_\ell) \in \pzc{V}$ \emph{sound} if
  \[\card{\Psi_{\hat \uu}^{t}(\Gamma_p)} \geq  \left(8\sqrt{\alpha}\right)^{\ell-1}  2^s (pd)^{3t}((n-1)!_t)^\ell (n!_t)^{3-\ell}. \] 
  Note that~\eqref{eq:s_dependent_lower_bd_1} holds if and only if~$(v_1)$ is sound.
Also note that for any~$\uu=(u_1,\ldots,u_\ell)$ and~$i\in[\ell]$, we can determine whether~$\uu$ is sound  or not without knowing the random status of edges adjacent to~$u_i$ in~$\Gamma$, as determining whether~$\uu$ is sound relies on counting embeddings that \emph{avoid}~$u_i$. 

  We now formulate a sequence of steps, that we will prove later, claiming that certain properties hold. 
  Let~$X_2(v_1) \subseteq  N_\Gamma(v_1;V^2)$ be the set of vertices~$u_2 \in N_\Gamma(v_1;V^2)$ such that~$(v_1,u_2)$ is sound and~$\deg_\Gamma(v_1,u_2;V^3) \geq \tfrac{d^2 n}{2}$.
  \begin{step}\label{step:1}
    With probability at least~$1 - n^{-6}$, the following statement holds in~$\Gamma_p$. 
    \begin{center}
             If~$(v_1)$ is sound,   then~$|X_2(v_1)| \geq \frac{dn}{2}$.
         \end{center}
  \end{step}
  Given~$v_2 \in V^2$, let~$X_3(v_1,v_2) \subseteq N_\Gamma(v_1,v_2;V^3)$ be the set of vertices~$u_3 \in N_\Gamma(v_1,v_2;V^3)$ such that~$(v_1,v_2,u_3)$ is sound. Furthermore, let~$Y_3(v_1,v_2) \subseteq X_3(v_1,v_2)$ be the set of those~$u_3$ such that~$v_2u_3 \in E(\Gamma_p)$.
  \begin{step}\label{step:2}
    With probability at least~$1 - n^{-6}$ the following statement holds in~$\Gamma_p$ for every~$v_2 \in V^2$. 
    \begin{center}
    If~$(v_1)$ is sound and~$v_2 \in X_2(v_1)$, then we have~$\card{Y_3(v_1,v_2)} \geq \frac{pd^2n}{8}$.
            
    \end{center}
  \end{step}
  Let now
~$ Z'(v_1) = \left\{ (u_2,u_3) \in V^2 \times V^3: u_2 \in X_2(v_1), u_3 \in Y_3(v_1,u_2)\right\}$
  and \[ Z(v_1)= \{ (u_2,u_3) \in Z'(v_1): \{v_1,u_2,u_3\} \text{ is a triangle in } \Gamma_p\}=\tr_{v_1}(\Gamma_p)\cap Z'(v_1).\] We will use \cref{step:1,step:2} to deduce the following.
  \begin{step}\label{step:3}
    With probability at least~$1-n^{-5}$, the following statement holds in~$\Gamma_p$. 
    \begin{center}
    If~$(v_1)$ is sound, then~$|Z'(v_1)| \geq \frac{pd^3n^2}{16}$.
    \end{center}
  \end{step}
  The claim in the following last step will be a consequence of \cref{lem:two stage reveal}.
  
  \begin{step}\label{step:4}
    With probability at least~$1 - n^{-5}$, the following statement holds in~$\Gamma_p$. 
    
    \begin{center}
    If~$|Z'(v_1)| \geq \frac{pd^3n^2}{16}$,   then we have~$|Z(v_1)| \geq \frac{(pd)^3 n^2}{32}$.
    \end{center}
  \end{step}
  Before we prove the claims in \cref{step:1,step:2,step:3,step:4}, let us use them to deduce \cref{clm:graded}. Note that assuming the statements in \cref{step:3,step:4} hold in~$\Gamma_p$ we have with probability at least $1-2n^{-5}$ that if~$(v_1)$ is sound then~$|Z(v_1)| \geq \tfrac{(pd)^3 n^2}{32}$. Furthermore,
  by the definition of $Z'_1(v_1)\supseteq Z(v_1)$ and of $X_3(v_1,u_2)\supseteq Y_3(v_1,u_2)$ we have that
  for all~$(u_2,u_3) \in Z(v_1)$, the vector~$(v_1,u_2,u_3)$ is sound, that is,
  \[\card{\Psi_{\hat v_1,\hat u_2,\hat u_3}^{t}(\Gamma_p)} \geq 64\alpha 2^s (pd)^{3t}((n-1)!_t)^3\,.\]
  Therefore, with probability at least $1-2n^{-5}$,
  \begin{align*}
    \card{\Psi_{v_1}^{t+1}(\Gamma_p)}
    &\geq \sum_{(u_2,u_3) \in Z(v_1)} \card{\Psi_{\hat v_1,\hat u_2,\hat u_3}^{t}(\Gamma_p)}\\
     &\geq 
    \frac{(pd)^3 n^2}{32} \cdot 64\alpha 2^s (pd)^{3t}((n-1)!_t)^3 \\
    &\geq  2^{s+1} \alpha    (pd)^{3(t+1)}((n-1)!_t) (n!_{t+1})^2 , 
  \end{align*}
  as required for the claim.  
  It remains to prove \cref{step:1,step:2,step:3,step:4}.

  \smallskip
  
  \textit{Proof of \cref{step:1}:} 
  For~$i =2,3$, let~$A_i \coloneq N_\Gamma(v_1;V^i)$. Furthermore, let~$A_2'
  \subseteq V^2$ be the set of vertices~$u_2 \in V^2$ for which~$(v_1,u_2)$ is
  sound and let~$A_2'' \subseteq V^2$ be the set of vertices~$u_2 \in V^2$ for
  which~$\deg(v_1,u_2;V^3) \geq \tfrac{d^2n}{2}$.  Note that~$X_2(v_1) = A_2
  \cap A_2' \cap A_2''$. Since~$(V^1, V^i)$ is~$(\e,d)$-super-regular, we
  have~$|A_i| \geq (d-\e)n$ for~$i =2,3$.  Since~$(V^2, V^3)$ is~$\e$-regular,
  we have~$|A_2''| \geq (1-\e)n$ by \cref{lem:reg-degree}.  Finally,
  observe that~$(v_1)$ being sound implies that 
  \begin{equation*}\begin{split}
      \card{\Psi_{\hat{v}_1}^t(\Gamma_p)}
      &\geq  2^{s} (pd)^{3t}(n-1)!_t (n!_t)^2
      \geq  2^{s_0} (pd)^{3t}(n-1)!_t (n!_t)^2 \\
      &\geq (1-\eta)^n (pd)^{3t}(n-1)!_t (n!_t)^2\,.
  \end{split}\end{equation*}
  Hence, it follows
  from \cref{lem:LDL} with~$\ell=2$ that
  with probability at least~$1-n^{-6}$, if~$(v_1)$ is sound
  then for all but at most~$\alpha n$ vertices $u_2\in V^2$ we have
  \begin{equation*}\begin{split}
  \card{\Psi_{\hat{v}_1,\hat{u}_2}^t(\Gamma_p)}
  & \geq \left(\frac{d}{10}\right)^2 \frac{n-t}{n} \card{\Psi_{\hat{v}_1}^t(\Gamma_p)}
  \geq \left(\frac{d}{10}\right)^2 \frac{1}{n} \cdot 2^{s} (pd)^{3t}(n-1)!_t (n!_t)^2 \\
  & \geq  8\sqrt{\alpha} 2^s(pd)^{3t}((n-1)!_t)^2 (n!_t)\,,
  \end{split}\end{equation*}
  showing that $(v_1,u_2)$ is sound.
  Thus, we get~$|A_2'| \geq (1-\alpha)n$ with probability at least~$1-n^{-6}$ and
  hence
  \[\card{X_2(v_1)} = \card{A_2 \cap A_2' \cap A_2''} \geq \frac{dn}{2},\]
  as claimed. 
    
  \smallskip
  
  \textit{Proof of \cref{step:2}:} 
    Fix some~$v_2 \in V^2$. 
    Let~$X_3 = X_3(v_1,v_2)$ and~$Y_3 = Y_3(v_1,v_2) \subseteq X_3$.
    It follows from an application  of \cref{lem:LDL} with~$\ell=3$ and~$\eta'$ replacing~$\eta$, that the following statement holds in~$\Gamma_p$ with probability at least~$1-n^{-8}$.
     \begin{center}
    If~$(v_1)$ is sound and~$v_2\in X_2(v_1)$, then~$\card{X_3} \geq \frac{d^2n}{4}$.
    \end{center}
    Here we used here that~$v_2\in X_2(v_1)$ implies that~$\deg_\Gamma(v_1,v_2;V^3) \geq \tfrac{d^2n}{2}$ as well as the fact that~$(v_1,v_2)$ being sound implies that  
    \[\card{\Psi_{\hat{v}_1, \hat{v}_2}^{t}(\Gamma_p)} \geq  8\sqrt{\alpha}  2^{s_0} (pd)^{3t}((n-1)!_t)^2 (n!_t)\geq  (1-\eta')^n (pd)^{3t}((n-1)!_t)^2 (n!_t),\]
    in order to appeal to \cref{lem:LDL}.

    Now note that, in order to determine~$X_3$, we do not need to reveal edges adjacent to~$v_2$. That is, the random set of vertices~$X_3$ is determined by~$(\Gamma_{\hat{v}_2})_p$. Therefore, by \cref{lem:two stage reveal}\cref{lem:many neighbours in random set} we have that with probability at least~$1-n^{-8}$ the following statement holds in~$\Gamma_p$. 
    \begin{center}
        If~$|X_3|\geq \frac{d^2n}{4}$,  then~$\card{Y_3} \geq \frac{pd^2n}{8}$.
    \end{center}
   Therefore with probability at least~$1-n^{-7}$, both the above statements hold in~$\Gamma_p$ and so by combining them we have the desired statement of this step for~$v_2\in V^2$. 
    Taking a union bound over all~$v_2 \in V^2$ then completes the proof.

  \smallskip
  
  \textit{Proof of \cref{step:3}:} 
    This is a simple case of combining \cref{step:1,step:2}. Indeed with probability at least~$1-n^{-5}$ both the statements of   \cref{step:1,step:2} hold in~$\Gamma_p$. Taking this to be the case,  if~$(v_1)$ is sound, we then have that 
    \[|Z'(v_1)| = \sum_{u_2 \in X_2(v_1)} \card{Y_3(v_1,u_2)} \geq \frac{dn}{2} \cdot \frac{pd^2n}{8} = \frac{pd^3n^2}{16},\]
    as required.

  \smallskip
  
  \textit{Proof of \cref{step:4}:} 
  This is a direct application of \cref{lem:two stage reveal}\cref{lem:triangles with F}. Indeed, note that~$Z'(v_1)\subseteq  \tr_{v_1}(\Gamma)$ is a random subset of edges determined by~$(\Gamma_{\hat v_1})_p$. The conclusion of \cref{step:4} then follows immediately from
\cref{lem:two stage reveal}\cref{lem:triangles with F}.

  \smallskip

  This concludes the proof of \cref{clm:graded} and hence the proof of the lemma. 
  \end{claimproof}
  \end{proof}

%
\section{Proof of the Local Distribution Lemma}\label{sec:LDL}
The purpose of this section is to prove the Local Distribution Lemma,  \cref{lem:LDL}. We will begin by reducing \cref{lem:LDL} to another lemma, \cref{lem:compare} below, using a simple averaging  argument. Before proving \cref{lem:compare}, we will then take a detour, establishing an Entropy Lemma (\cref{lem:entropy lemma}) which will be crucial for the proof of Lemma~\ref{lem:compare}, which is finally given in \cref{sec:counting via comparison}.  

\subsection{A simplification} \label{sec:simp}

Given  some~$t, \ell$ and~$\uu=(u_1, \ldots, u_{\ell-1})$ as in the statement of Lemma~\ref{lem:LDL}, we aim to prove a lower bound on the size of~$\Psi^t_{\hat \uu,\hat u_{\ell}}$ for almost all of the~$u_{\ell}\in V^{\ell}$. The key step for this is given in the following lemma, which we now motivate.
Given that~$\Psi^t_{\hat{\uu}}$ is large, a simple averaging argument shows that~\cref{eq:LDL} is true `on average' (i.e.\ if we take the average of~$|\Psi_{\hat \uu,\hat u_\ell}^t(\Gamma_p)|$ over all~$u_\ell \in V^\ell$). That is, there is a vertex~$u$ such that the assumption on $|\Psi_{\hat \uu,\hat u}^t(\Gamma_p)|$ in Lemma~\ref{lem:compare} below holds. Lemma~\ref{lem:compare} then states that this implies that~\cref{eq:LDL} holds indeed for almost all choices of~$u_{\ell}$, which is the challenging part in the proof of Lemma~\ref{lem:LDL}. In order to prove Lemma~\ref{lem:compare} in \cref{sec:counting via comparison}, we compare the difference in the sizes of~$\Psi^t_{\hat \uu,\hat u_{\ell}}$ for different choices of~$u_{\ell}\in V^{\ell}$ using the Entropy Lemma (\cref{lem:entropy lemma}).
\begin{lemma}\label{lem:compare}
For all~$0<\alpha,d\le 1$ and~$K>0$ there exists~$\eta,\e>0$ and~$C>0$ such that  for all sufficiently large~$n \in \N$ and for any~$p \geq C(\log n)^{1/3} n^{-2/3}$,  if~$\,\Gamma$ is an~$(\e,d)$-super-regular tripartite graph with parts of size~$n$,~$t\in \N$  such that~$\left(1-\eta\right)n \leq t < n$,~$\ell\in[3]$,~$\uu=(u_1,\ldots,u_{\ell-1})\in \pzc{V}$ and~$u\in V^\ell$ then the following holds in~$\Gamma_p$  with probability at least~$1-n^{-K}$.
If
\[\card{\Psi_{\hat \uu,\hat u}^t(\Gamma_p)} \geq (1-\eta)^n (pd)^{3t}  ((n-1)!_{t})^{\ell} (n!_t)^{3-\ell},\]
then
\[\card{\Psi_{\hat\uu,\hat v}^t(\Gamma_p)} \geq \left(\frac{d}{10}\right)^2\cdot \card{\Psi_{\hat \uu,\hat u}^t(\Gamma_p)}\]
for at least~$(1-\alpha)n$ vertices~$v\in V^{\ell}$.
\end{lemma}
Indeed, with \cref{lem:compare} in hand, \cref{lem:LDL} follows easily.
\begin{proof}[Proof of \cref{lem:LDL}]
Fix~$\e,\tfrac{1}{C}\ll \eta \ll d,\alpha$.  Fix~$\Gamma$,~$t\in \N$ with~$\left(1-\eta\right) n\leq t < n$,~$\ell\in[3]$ and~$\uu=(u_1\ldots,u_{\ell-1})\in \pzc{V}$.
By applying \cref{lem:compare} with~$K+1$ replacing~$K$ and taking a union bound, we have that with probability at least~$1-n^{-K}$, the conclusion of \cref{lem:compare}  holds in~$G=\Gamma_p$ for all~$u\in V^\ell$. So suppose that this is the case and  further 
  suppose that
\[\card{\Psi_{\hat\uu}^t(G)} \geq (1-\eta)^n  (pd)^{3t}((n-1)!_{t})^{\ell-1} (n!_t)^{4-\ell}.\]
Now, for each~$\psi\in\Psi_{\hat{\uu}}^t(G)$, we have~$\psi \in \Psi_{\hat \uu,\hat u_{\ell}}^t(G)$ for exactly~$n-t$ choices of~$u_{\ell}\in V^{\ell}$. Therefore, we have that
\begin{align*}
  \sum_{u\in V^{\ell}} \card{\Psi_{\hat\uu,\hat u}^t(G)} & = (n-t) \card{\Psi_{\hat\uu}^t(G)}. 
\end{align*}
By averaging, there must be some~$u^*\in V^{\ell}$ such that
\begin{align*}
  \card{\Psi_{\hat\uu,\hat u^*}^t(G)} &\geq  \left(\frac{n-t}{n}\right) \card{\Psi_{\hat\uu}^t(G)}
  \\&\geq \left(\frac{n-t}{n}\right)(1-\eta)^n  (pd)^{3t}((n-1)!_{t})^{\ell-1} (n!_t)^{4-\ell} \\
  &= (1-\eta)^n (pd)^{3t} ((n-1)!_{t})^{\ell} (n!_t)^{3-\ell}.
\end{align*}
The result now follows from applying the assumed conclusion of \cref{lem:compare} with~$u^*$ playing the r\^ole of~$u$.
\end{proof}
\subsection{The Entropy Lemma}\label{sec:entropylem}
In this section, we will prove a key lemma, \cref{lem:entropy lemma}, which we call the Entropy Lemma.  We start with some definitions.
Given some tripartite~$\Gamma$ with parts of size~$n$, some~$\ell \in [3]$,~$t \in [n]$ and some~$\psi \in \Psi^t(\Gamma)$, we define~$I^\ell(\psi) \subset V^\ell$ to be the vertices in~$V^\ell$ which are isolated in the embedded subgraph~$\psi(D_t)$. If~$\ell$ is clear from context, we will drop the superscript.
If we are further given some~$v \in V^{\ell}$,  we define
\[\psi_v=%
\begin{cases}%
  \hfil \emptyset &\text{if } v \in I(\psi),\\
  \left(N_{\psi(D_t)}\left(v;V^j\right):j\in J\right) &\text{if } v \not \in I(\psi),
\end{cases}\]
where~$J=[3]\setminus\{\ell\}$.
So~$\psi_v$ either returns an empty set, indicating that the vertex~$v$ is isolated in~$\psi(D_t)$, or it returns the pair of vertices which are contained in the triangle containing~$v$ in~$\psi(D_t)$.
We also define the function
\[Y_v(\psi)=\1[\{\psi_v\neq \emptyset\}]=
\begin{cases}
  1 &\text{if } \psi_v \neq \emptyset,\\
  0 &\text{if } \psi_v=\emptyset,
\end{cases}\]
which returns~$1$ if~$v \not\in I(\psi)$ and~$0$ otherwise.
 Note that for any~$\ell\in [3]$ the set~$\{\psi_v:v\in V^{\ell}\}$ completely determines the (unordered) subgraph~$\psi(D_t)$.

For a fixed~$u\in V^\ell$ and~$v\in V^{\ell}\setminus \{u\}$, we will be interested in the distribution of~$\psi^*_v$ if~$\psi^*$ is chosen randomly among a set of embeddings we wish to extend.
In order to analyse this, we use entropy. See \cref{sec:entropybasics} for the definition and basic properties.
We remark that there will be two independent stages of randomness in the argument.
First, there is the random subgraph~$\Gamma_p \subseteq \Gamma$, and second, there will be a randomly chosen~$\psi^* \in \Psi^t(\Gamma_p)$.
In particular, the values of the entropy function~$h(\psi^*),h(\psi^*_v)$ are random variables themselves. However, once we fix a particular instance~$G = \Gamma_p$, these values are deterministic.
We proceed with the following definition which will be convenient to ease notation in what follows.
\begin{definition}\label{eq:hhat}
For~$n\in \N$,~$p=p(n)\in(0,1)$ and~$0<d\le 1$, we define \[H=H(n,p,d) \coloneq \log \left((pd)^3 \cdot n^2\right).\]
\end{definition} 
To see the relevance of this function, note that in a random sparsification of the complete tripartite graph~$K_{n,n,n}$  with probability~$pd$, we would expect a given vertex to lie in~$(pd)^3n^2$ triangles. Therefore if we fix a vertex~$v$ and take a uniformly random triangle containing~$v$, we expect the entropy of the random variable which chooses this triangle, to be roughly~$H(n,p,d)$. The function~$H$ can thus be seen as benchmark for the maximum entropy (recalling \cref{lem:maximality of the uniform}) of a randomly chosen triangle containing a fixed vertex. Our aim will be to show that, for most choices of fixed vertex~$v$,~$H$ is a good approximation for the entropy of the random variable~$\psi^*_v$ discussed above.  


We begin with observing that the function~$H$ provides an appropriate upper bound on the entropy we will be interested in. 

\begin{obs}\label{obs:entropy lemma}
   For all~$0<\e'<d\le 1$ and~$L>0$ there exists~$\e>0$ and~$C>0$ such that  for all sufficiently large~$n \in \N$ and for any~$p \geq C(\log n)^{1/3} n^{-2/3}$,  if~$\Gamma$ is an~$(\e,d)$-super-regular tripartite graph with parts of size~$n$,~$t\in [n]$,~$\ell\in[3]$,~$\uu=(u_1,\ldots,u_{\ell-1})\in \pzc{V}$ and~$u\in V^\ell$ then the following holds in~$\Gamma_p$  with probability at least~$1-n^{-L}$.
   
   For~$\psi^*$  chosen uniformly from~$\Psi_{\hat \uu,\hat u}^t(\Gamma_p)$, we have that 
~$h(\psi^*_v|Y_v(\psi^*)=1) \leq H(n,p,d) + \e'$ for all but at most~$\e'n$ vertices~$v \in V^{\ell}$.
\end{obs}
\begin{proof}
Choose~$0<\e,\tfrac{1}{C}\ll \e',d,\tfrac{1}{L}$.   By \cref{cor:fixed vertex triangles}, we have that with probability at least~$1-n^{-L}$, 
\[\card{\tr_v(\Gamma_p)}=(1 \pm \e')(pd)^3n^2,\]
for all but at most~$\e'n$ vertices~$v \in V^{\ell}$.
 In particular, for each such~$v$, we have~$\log \card{\tr_v((\Gamma_p)_{\hat \uu,\hat u})} \leq H(n,p,d)  + \e'$. 
  Therefore, by \cref{lem:maximality of the uniform}, we have~$h(\psi_v^*|Y_v(\psi^*) = 1) \leq H(n,p,d) + \e'$ for all~$v$ as above and for~$\psi^*\in\Psi_{\hat \uu,\hat u}^t(\Gamma_p)$ chosen uniformly at random.
\end{proof}
The main purpose of this section is to  provide a partial converse to the above observation, showing that for almost all vertices~$v \in V^\ell$,~$H$ is a good approximation for the entropy~$h(\psi^*_v| Y_v(\psi^*) = 1)$. The full statement is as follows.

\begin{lemma}[Entropy Lemma]\label{lem:entropy lemma}
For all~$0<\beta,d\le 1$ and~$L>0$ there exists~$\eta,\e>0$ and~$C>0$ such that  for all sufficiently large~$n \in \N$ and for any~$p \geq C(\log n)^{1/3} n^{-2/3}$,  if~$\,\Gamma$ is an~$(\e,d)$-super-regular tripartite graph with parts of size~$n$,~$t\in \N$  such that~$\left(1-\eta\right)n \leq t < n$,~$\ell\in[3]$,~$\uu=(u_1,\ldots,u_{\ell-1})\in \pzc{V}$ and~$u\in V^\ell$ then the following holds in~$\Gamma_p$  with probability at least~$1-n^{-L}$. If
\[\card{\Psi_{\hat \uu,\hat u}^t(\Gamma_p)} \geq (1-\eta)^n (pd)^{3t}  ((n-1)!_{t})^{\ell} (n!_t)^{3-\ell},\]
and~$\psi^*$ is chosen uniformly from~$\Psi_{\hat \uu,\hat u}^t(\Gamma_p)$, then we have that~$h(\psi^*_v| Y_v(\psi^*) = 1) \geq H(n,p,d) - \beta$ for all but at most~$\beta n$ vertices~$v \in V^\ell$.

\end{lemma}
In the remainder of this section, we will prove \cref{lem:entropy lemma}. Recall that we have~$V(\Gamma) = V(\Gamma_p) = V^1 \cup V^2 \cup V^3$ with each~$V^i$ of size~$n$. 
As above, for~$t\in [n]$, an embedding~$\psi\in \Psi^t(\Gamma)$ and some~$\ell \in[3]$, we denote by~$I(\psi) = I^\ell(\psi)$ the vertices in~$V^\ell$ which are not contained in the subgraph~$\psi(D_t)$.
In the proof, we will describe~$\psi$ by revealing the status of~$\psi_v$ one by one for each~$v \in V^\ell$ according to some linear order~$\sigma$ of~$V^\ell$. 
In order to do so, we need to make some further definitions. Firstly we denote by~$w<_\sigma v$ that~$w$ occurs before~$v$ in the ordering~$\sigma$. 
Now given some fixed~$t$,~$\psi$ and~$\ell$ as above and an ordering~$\sigma$ of~$V^\ell$,  we will be interested in revealing~$\psi\in  \Psi^t(\Gamma)$ according to the ordering $\sigma$ as follows. We imagine processing the vertices~$v\in V^\ell$ in order and as we process each vertex $v$ we reveal its status in~$\psi$ by revealing~$\psi_v$. Either~$v$ is not in a triangle in $\psi(D_t)$ or $v$ is in a triangle, in which case, we are given the other vertices of the triangle containing $v$ in $\psi(D_t)$. Now consider the moment before processing some vertex $v\in V^\ell$. At this point, we know all the triangles in $\psi(D_t)$ that contain vertices~$w\in V^\ell$ such that $w<_\sigma v$. We are interested in which vertices are candidates to feature in $\psi_v$ at this point and the following definition captures this.

 For some  fixed~$t$,~$\psi$ and~$\ell$ as above, an ordering~$\sigma$ of~$V^\ell$, some~$\uu \in \cV$, some~$j \in [3] \setminus \{\ell\}$ and some $v\in V^\ell$ we define
\[A_v^j(\psi,\sigma,\uu) \coloneq \left\{a\in V^j_{\hat{\uu}}: a\not\in \bigcup_{w \in V^\ell: \ w<_\sigma v} \psi_w\right\}\]
and~$A_v(\psi,\sigma,\uu) \coloneq \bigcup_{j \in J} A_v^j(\psi,\sigma,\uu)$, where~$J \coloneq [3] \setminus \{\ell\}$.
We think of these vertices as being \emph{`alive'} at the point just before processing $v$ (when we are about to reveal~$\psi_v$).
By `alive', we mean that it is still possible that~$\psi_v$ reveals that~$a\in A_v^j(\psi,\sigma,\uu)$ is in a triangle with~$v$. All other vertices~$a \in V^j \setminus A^j_v(\psi,\sigma,\uu)$ are already embedded in triangles with vertices~$w \in V^\ell$ which come before~$v$ in the ordering~$\sigma$ (or lie in $\uu$ in which case we are forbidden from including them in a triangle in $\psi$).
\subsubsection*{Triangles with alive vertices}
In this subsection, we will prove that most vertices~$v \in V^\ell$ are in the expected number of triangles with the other two vertices still being `alive'.
This will be useful in the proof of the Entropy Lemma, \cref{lem:entropy lemma}.
\begin{lemma}\label{lem:exp-emb}
For all~$0<\tau<d\le 1$ and~$L>0$ there exists~$\e>0$ and~$C>0$ such that  for all sufficiently large~$n \in \N$ and for any~$p \geq C(\log n)^{1/3} n^{-2/3}$,  if~$\,\Gamma$ is an~$(\e,d)$-regular tripartite graph with parts of size~$n$ then the following holds in~$\Gamma_p$  with probability at least~$1-n^{-L}$. If~$t\in [n-1]$,~$\ell\in[3]$,~$\uu=(u_1,\ldots,u_{\ell-1})\in \pzc{V}$,~$u\in V^\ell$,  
$\psi \in \Psi_{{\hat\uu,\hat u}}^t(\Gamma_p)$ and~$\sigma$ is an ordering of~$V^\ell$, then there are at most~$\tau n$ vertices~$v \in V^\ell$ for which
\begin{equation}\label{eq:bad3}
  \card{\tr_v(\Gamma_p)\cap E(\Gamma[A_{v}(\psi,\sigma,\uu)])} > (pd)^3 \prod_{j\in J}\card{A^j_v(\psi,\sigma,\uu)} + \tau (pd)^3 n^2,
\end{equation}
where, as above,~$J= [3] \setminus \{\ell\}$. 
\end{lemma}
\begin{proof}
Choose~$0<\e,\tfrac{1}{C}\ll \e'\ll \tau,d, \tfrac{1}{L}$. Let~$G \subseteq \Gamma$ be any subgraph satisfying
\begin{equation}\label{eq:alive-tri-count}
  \card{K_3(G[X_1 \cup X_2 \cup X_3])} \leq (pd)^3 |X_1||X_2||X_3| + \e' p^3n^3,
\end{equation}
for all~$X_1 \subseteq V^1$,~$X_2 \subseteq V^2$,~$X_3 \subseteq V^3$ and note that~$\Gamma_p$ is such a subgraph with probability at least~$1-n^{-L}$ by \cref{lem:triangle-count}.
We will show that~$G$ already satisfies the conclusion of \cref{lem:exp-emb}.
Let~$\ell\in[3]$,~$t\in [n-1]$,~$\uu=(u_1\ldots,u_{\ell-1})\in \pzc{V}$,~$u_\ell \in V^\ell$,~$\psi \in \Psi_{\hat \uu,\hat u}^t(G)$ and let~$\sigma$ be an ordering of~$V^\ell$.
Enumerate~$V^\ell = \{v_1^\ell, \ldots, v_n^\ell\}$ according to the ordering~$\sigma$, that is,  in such a way that~$v_1^\ell <_\sigma \cdots <_\sigma v_n^\ell$.
Define~$U \subseteq V^\ell$ to be the set of vertices satisfying~\cref{eq:bad3}. We will show that~$|U| < \tau n$.
We split~$V^\ell$ into intervals as follows. Let~$\tau':=\tfrac{\tau}{4}$,~$K:=\ceil{\tfrac{1}{\tau'}}$ and for~$k = 1, \ldots, K$, let
\[W_k = \{v^\ell_i: 1 + (k-1) \cdot \tau' n \leq i < 1 +  k \cdot \tau' n\}\] and~$U_k \coloneq U \cap W_{k}$. 
Fix some~$k \in [K]$ and let~$i_k \coloneq 1 + \lceil(k-1) \cdot \tau' n\rceil$ and~$w_k \coloneq v^\ell_{i_k}$ (that is,~$w_k$ is the first vertex in~$W_k$). Let~$X_\ell = U_k$ and~$X_j = A_{w_k}^{j}(\psi,\sigma,\uu)$  for~$j\in J=[3]\setminus \{\ell\}$. 
It follows that, for any~$z\in U_k$,
\begin{align*}
  \card{\tr_z(G[\cup_{i\in[3]}X_i])} &\geq \card{\tr_z(G[X_\ell\cup A_z(\psi,\sigma,\uu)])} \\
  &\geq (pd)^3 \prod_{j\in J}\card{A^j_z(\psi,\sigma,\uu)} + \tau (pd)^3 n^2 \\
  &\geq (pd)^3 \prod_{j\in J}\left(\card{X_j} - \tau' n \right) + \tau (pd)^3 n^2 \\
  &\geq (pd)^3 \prod_{j\in J}\card{X_j} + \frac{\tau}{2} (pd)^3 n^2.
\end{align*}
Here, the first inequality follows from the fact that~$z >_\sigma w_k$ and thus~$A_{z}(\psi,\sigma,\uu) \subseteq A_{w_k}(\psi,\sigma,\uu)$ for every~$z \in U_k$.
The second inequality follows from the fact that~$z \in U$ and the third from the fact that~$\card{A_{z}^j(\psi,\sigma,\uu)} \geq \card{A_{w_k}^j(\psi,\sigma,\uu)} - \tau' n$ for all~$z \in U_k$ since~$z$ and~$w_k$ are close in the ordering~$\sigma$.
By summing over all~$z \in U_k$, it follows that
\[\card{K_3(G[X_1 \cup X_2 \cup X_3])} \geq (pd)^3 \card{X_1} \card{X_2}\card{X_3} + \frac{\tau}{2} (pd)^3 |X_\ell|n^2.\]
Combining this with~\cref{eq:alive-tri-count} gives~$|U_k| = |X_\ell| \leq \frac{2\e'}{\tau d^3} n < \frac{\tau^2}{8}n$, by our choice of constants.
It follows that~$|U|=\sum_{k=1}^{K}\card{U_k} < \tau n$, as claimed.
\end{proof}
\subsubsection*{Proof of the Entropy Lemma}
Here, we will prove \cref{lem:entropy lemma}. The proof is quite  long and so we will break it up into smaller claims along the way.   Our proof works by contradiction. As  $\card{\Psi_{\hat \uu,\hat u}^t(\Gamma_p)}$ is large, we know that~$h(\psi^*)$ is large as $\psi^*$ is chosen uniformly at random from $\Psi_{\hat \uu,\hat u}^t(\Gamma_p)$. Moreover, using the chain rule (Lemma~\ref{lem:chain rule}), we can decompose $h(\psi^*)$ as the sum of local entropy values depending on the $\psi_v^*$. Now we assume that there are a significant number of bad vertices $v$ for which the local entropy value $h(\psi_v^*|Y_v(\psi^*)=1)$ is too small. We will then apply the chain rule (Lemma~\ref{lem:chain rule}) using an ordering on the vertices which places these bad vertices at the beginning of the ordering. This has the effect that the shortcoming of their contribution to the overall entropy $h(\psi^*)$ is felt the most. We then upper bound the contribution of the entropy values at other (good) vertices, and hence conclude that the overall entropy $h(\psi^*)$ is too small, giving a contradiction. In order to achieve this upper bound, we rely on random properties of $\Gamma_p$ and  we have to split the entropy values further, delving into the average that outputs the entropy values and looking at individual embeddings. 
\begin{proof}[Proof of \cref{lem:entropy lemma}]
Choose~$0<\e,\tfrac{1}{C}\ll \tau \ll \eta\ll \delta \ll \gamma \ll \beta, d,\tfrac{1}{L}$. Fix~$\Gamma$,~$t\in \N$,~$\ell\in [3]$,~$\uu=(u_1\ldots,u_{\ell-1})\in \pzc{V}$, and~$u\in V^\ell$ as in the statement of \cref{lem:entropy lemma}. 
Assume~$G \subseteq \Gamma$ is a subgraph of~$\Gamma$ with~$V(G) = V(\Gamma)$ which satisfies the following properties 
 for all~$\psi \in \Psi_{\hat \uu,\hat u}^t(G)$ and every ordering~$\sigma$ of~$V^\ell$.
\begin{enumerate}[label=\textbf{(P.\arabic*})]
  \item\label{en:P-tri-count-all-vxs}
  For all vertices~$v \in V(G)$, we have
  \[\card{\tr_v(G)} \leq 10p^3n^2.\]
  \item\label{en:P-tri-count-most-vxs}
    There are at most~$\tau n$ vertices~$v \in V^\ell$ for which
    \begin{equation*}
      \card{\tr_v(G)\cap E(G[A_{v}(\psi,\sigma,\uu)])} > (pd)^3 \prod_{j \in [3] \setminus \{\ell\}}\card{A^j_v(\psi,\sigma,\uu)} + \tau (pd)^3 n^2.
    \end{equation*}
\end{enumerate}
By Lemmas~\ref{lem:exp-emb},~\ref{lem:upper triangle count for all vxs} and a union bound,~$\Gamma_p$ satisfies these properties with probability at least~$1-n^{-L}$ and therefore it suffices to show that any~$G$ satisfying the above properties, satisfies the conclusion of \cref{lem:entropy lemma}.

To ease notation, let~$\Psi \coloneq \Psi_{\hat \uu,\hat u}^t(G)$. Furthermore, let~$\psi^*$ be chosen uniformly from~$\Psi$.
We may assume that
\begin{equation*}
  \card{\Psi} \geq (1-\eta)^n (pd)^{3t} ((n-1)!_{t})^\ell (n!_{t})^{3-\ell},
\end{equation*}
as otherwise there is nothing to prove. In particular, by \cref{lem:maximality of the uniform}, we have
\begin{align}
\nonumber h(\psi^*) &\geq n\log(1-\eta)+ 3t\log(pd) + 3 \log (n!_t) - 3\log(n)\\
          &\geq 3t\log(pd) + 3 \log (n!_t) - \delta n,\label{eq:lower bound on h}
\end{align}
where we used~$\eta \ll \delta$ and that~$n$ is large enough in the last step.

Assume for a contradiction that there are at least~$\beta n$ vertices~$v \in V^\ell$ such that~$h(\psi^*_v| Y_v(\psi^*) = 1)< H(n,p,d) - \beta$ and let~$U \subset V^\ell$ be a set of these exceptional vertices of size~$|U|=\gamma n$. We will derive an upper bound on~$h(\psi^*)$ which contradicts~\cref{eq:lower bound on h}.
Recall that~$I(\psi) = I^\ell(\psi) \subset V^\ell$ is the set of vertices which are isolated in~$\psi(D_t)$. We begin as follows
\begin{align}
    h(\psi^*)
    &= h\left(\psi^*,\{\psi^*_v\}_{v\in V^\ell},I(\psi^*)\right) \label{eq:entropy-ub-1}\\
    &= h\left( \{\psi^*_v\}_{v\in V^\ell}, I(\psi^*)\right)+ h\left(\psi^*|\{\psi^*_v\}_{v\in V^\ell}, I(\psi^*)\right) \label{eq:entropy-ub-2}\\
    &\leq h\left( \{\psi^*_v\}_{v\in V^\ell}, I(\psi^*)\right)+ \log (t!) \label{eq:entropy-ub-3}\\
    &= h\left( \{\psi^*_v\}_{v\in V^\ell}| I(\psi^*)\right) +  h\left( I(\psi^*)\right) +\log (t!) \label{eq:entropy-ub-4}\\
    &\leq  h\left( \{\psi^*_v\}_{v\in V^\ell}| I(\psi^*)\right) +\log(t!) +\log\left( \binom{n}{t}\right) \label{eq:entropy-ub-5}\\
    &= h\left( \{\psi^*_v\}_{v\in V^\ell}| I(\psi^*)\right) +\log(n!_t).\label{eq:entropy-ub-6}
\end{align}
Here, we used \cref{lem:redundancy} in~\cref{eq:entropy-ub-1} and the chain rule (\cref{lem:chain rule}) in~\cref{eq:entropy-ub-2} and~\cref{eq:entropy-ub-4}.
In~\cref{eq:entropy-ub-3}, we used \cref{lem:cond maximality of the uniform} coupled with the fact that the set~$\{\psi_v\}_{v\in V^\ell}$ completely determines the (unordered) subgraph~$\psi(D_t)$.  Indeed, note that there are~$t!$ embeddings~$\psi \in \Psi$ which map to the same subgraph~$\psi(D_t)$, namely one for each choice of ordering of the triangles.
Finally, in~\cref{eq:entropy-ub-5} we used \cref{lem:maximality of the uniform}.

Now, in order to estimate this sum further, we fix some ordering~$\sigma$ of~$V^\ell$ in which the vertices in~$U$ come first, that is~$w <_\sigma w'$ for all~$w \in U$ and~$w' \in V^\ell \setminus U$. We then reveal vertices in that order and apply the conditional chain rule (\cref{lem:cond chain rule}). That is,
\begin{align}
\nonumber    h\left( \{\psi^*_v\}_{v\in V^\ell}| I(\psi^*)\right) &= \sum_{v\in V^\ell} h\left( \psi^*_v| \{\psi^*_w:w<_\sigma v \},I(\psi^*)\right) \\
   &\leq  \sum_{v\in U} h\left( \psi^*_v| I(\psi^*)\right)+ \sum_{v\in V^\ell\setminus U} h\left( \psi^*_v| \{\psi^*_w:w<_\sigma v \},I(\psi^*)\right),\label{eq:entropy-calc-split-U}
\end{align}
where we applied \cref{lem:dropping conditioning} in the second step. We treat the vertices in~$U$ separately to those in~$V^\ell\setminus U$. To ease notation, we make the following definition. For~$\psi \in \Psi$ and~$v\in V^\ell$, we let~$t_v(\psi)$ denote the number of vertices~$w\in V^\ell$ such that~$w<_\sigma v$ and~$w \notin I(\psi)$. Let us first address the vertices in~$U$.
\begin{claim} \label{clm:U vxs}
For all~$v \in U$, we have that
\begin{equation*} 
    h(\psi^*_v|I(\psi^*)) \leq \frac{1}{\card{\Psi}} \sum_{\psi\in \Psi} Y_v(\psi) \left(\log\left((pd)^3(n-t_v(\psi))^2\right) - \frac{\beta}{2}\right).
\end{equation*}
\end{claim}
\begin{claimproof}
 Now, for each~$v \in U$, we have
\begin{align*}
 h(\psi^*_v|I(\psi^*))
    &\leq h(\psi_v^*|Y_v(\psi^*))\\
    &= \Pro{Y_v(\psi^*)=1}h(\psi^*_v| Y_v(\psi^*) = 1) +\Pro{Y_v(\psi^*)=0}h(\psi^*_v|Y_v(\psi^*)=0)\\
    &\leq \Pro{Y_v(\psi^*)=1}\left(H(n,p,d)- \beta\right)\\
    &= \frac{1}{\card{\Psi}} \sum_{\psi\in \Psi}Y_v(\psi)\left(H(n,p,d)- \beta\right).
\end{align*}
Here we used \cref{lem:dropping conditioning} and the fact that~$I(\psi^*)$ determines~$Y_v(\psi^*)$,  the definition of conditional entropy~\cref{eq:conditionalentropydef}, and the definition of~$U$.
Furthermore, we have~$t_v(\psi) \leq \gamma n$ for all~$v \in U$ and~$\psi\in \Psi$ since~$U$ comes at the beginning of the ordering~$\sigma$. Therefore,
\begin{align*}
    \log\left((pd)^3(n-t_v(\psi))^2\right) &\geq \log\left((pd)^3(1-\gamma)^2n^2\right) \\
    &= H(n,p,d) + 2\log(1-\gamma)\\
    &\geq H(n,p,d)  - 4\gamma\\
    &\geq H(n,p,d)  - \frac{\beta}{2}.
\end{align*}
Combining this with our upper bound on~$h(\psi^*_v|I(\psi^*))$ above completes the proof of the claim.
\end{claimproof}

We will now deal with the vertices outside~$U$. Given~$v \in V^\ell$ and~$\psi \in \Psi$, we write
\[h'(v,\psi) \coloneq h\left(\psi_v^*|I(\psi^*)=I(\psi),\{\psi^*_w=\psi_w\}_{w<_\sigma v}\right).\]
\begin{claim} \label{clm:other vertices}
  The following is true for all~$\psi \in \Psi$.
\begin{enumerate}[(i)]
    \item For all~$v\in V^\ell$, we have
    \[h'(v,\psi) \leq \logb{(pd)^3 (n - t_v(\psi))^2} + \logb{\frac{10}{d^3}} + \logb{\frac{n^2}{(n-t_v(\psi))^2}}.\]
    \item There exists a set~$B(\psi) \subset V^\ell$ with~$|B(\psi)| \leq \delta n$, such that for all~$v\in V^\ell \setminus B(\psi)$, we have
    \[h'(v,\psi) \leq \log\left((pd)^3(n-t_v(\psi))^2\right) +  \delta.\]
\end{enumerate}
\end{claim}
\begin{claimproof}
The first inequality follows from \cref{en:P-tri-count-all-vxs} and \cref{lem:cond maximality of the uniform}. Indeed, for all~$v\in V^\ell$, we have
\begin{align*}
h'(v,\psi)
    &\leq \log\left(\card{\tr_v(G)}\right)\\
    &\leq \log(10p^3n^2)\\
    &= \logb{(pd)^3 (n - t_v(\psi))^2} + \logb{\frac{10}{d^3}} + \logb{\frac{n^2}{(n-t_v(\psi))^2}}.
\end{align*}
For the second inequality, we will use \cref{en:P-tri-count-most-vxs} in combination with \cref{lem:cond maximality of the uniform}. We have that for all but at most~$\tau n$ vertices,
\begin{align}
 \nonumber  h'(v,\psi)&\leq \log\left(\card{\tr_v(G)\cap E(G[A_{v}(\psi,\sigma,\uu)])}\right) \\ \nonumber
   &\leq \log\left((pd)^3 \prod_{j\in J}\card{A^j_v(\psi,\sigma,\uu)} + \tau (pd)^3 n^2\right) \\
   &\leq \log\left((pd)^3 (n-t_v(\psi))^2 + \tau (pd)^3 n^2\right).\label{eq:h'-first-step}
\end{align}
Observe that~$t_v(\psi) \leq \big(1-\tfrac{\delta}{2}\big)n$ for all but at most~$\tfrac{\delta n}{2}$ vertices~$v \in V^\ell$. In particular, we have
\[(n-t_v(\psi))^2 \geq \frac{\delta^2 n^2}{4} \geq \frac{\delta^2}{4\tau} \cdot \tau n^2 \geq \frac{1}{\delta} \cdot \tau n^2,\]
for all but at most~$\tfrac{\delta n}{2}$ vertices~$v \in V^\ell$ (we used that~$\tau \ll \delta$ here).
Plugging this back into~\cref{eq:h'-first-step}, we get
\begin{align*}
 h'(v,\psi)
   &\leq \log\left((1+\delta) \cdot (pd)^3 (n-t_v(\psi))^2\right) \leq \delta + \log\left((pd)^3 (n-t_v(\psi))^2\right)
\end{align*}
for all but at most~$\big(\tau +\tfrac{\delta}{2}\big)n \leq \delta n$ vertices~$v \in V^\ell$.
\end{claimproof}

We will now use \cref{clm:U vxs,clm:other vertices} to finish the proof.
Indeed, it follows from \cref{clm:U vxs} that
\begin{equation}
\sum_{v \in U} h(\psi^*_v|I(\psi^*)) \leq \frac{1}{\card{\Psi}} \sum_{\psi\in \Psi} \sum_{v \in U}Y_v(\psi) \left(\logb{(pd)^3(n-t_v(\psi))^2} - \frac{\beta}{2}\right).\label{eq:entropy-calc-U}
\end{equation}
Furthermore, using \cref{clm:other vertices},  the definition of conditional entropy~\cref{eq:conditionalentropydef-2} (and \cref{lem:cond maximality of the uniform} to conclude that~$h'(v,\psi)=0$ if~$Y_v(\psi)=0$), we have
\begin{align}
\nonumber \sum_{v \in V^\ell \setminus U} &h\left( \psi^*_v| \{\psi^*_w:w<_\sigma v \},I(\psi^*)\right)
    = \sum_{v \in V^\ell \setminus U} \frac{1}{\card{\Psi}}\sum_{\psi\in \Psi} Y_v(\psi)h'(v,\psi)\\
      &\leq \frac{1}{\card{\Psi}}\sum_{\psi\in \Psi}
      \left(\delta n + N_1(\psi) + \sum_{v \in V^\ell \setminus U} Y_v(\psi) \logb{(pd)^3(n-t_v(\psi))^2}\right),\label{eq:entropy-calc-other}
\end{align}
where
\[N_1(\psi) = \sum_{v \in B(\psi)} Y_v(\psi) \left(\logb{\frac{10}{d^3}} + 2\logb{\frac{n}{n-t_v(\psi)}}\right).\]
Let now
\[
M(\psi) \coloneq \sum_{v \in V^\ell}Y_v(\psi) \logb{(pd)^3(n-t_v(\psi))^2}, \qquad \text{ and } \qquad
N_2(\psi) \coloneq \sum_{v \in U}Y_v(\psi)\cdot \frac{\beta}{2}.\]
Then, combining~\cref{eq:entropy-calc-U,eq:entropy-calc-other,eq:entropy-calc-split-U}, we get
\begin{equation}\label{eq:entropy-main-calc}
    h\left( \{\psi^*_v\}_{v\in V^\ell}| I(\psi^*)\right)
    \leq \frac{1}{\card{\Psi}}\sum_{\psi \in \Psi} \left( M(\psi) + N_1(\psi) + \delta n - N_2(\psi) \right).
\end{equation}
We will bound each of these terms one by one.
\begin{claim}\label{cl:entropy-error-terms}
For all~$\psi \in \Psi$, we have that \[M(\psi) = 3t\log(pd)+ 2\log(n!_t), \qquad N_1(\psi) \leq \sqrt{\delta }n \qquad \mbox{ and } \qquad N_2(\psi) \geq \gamma^2  n.\]
\end{claim}
Before we prove \cref{cl:entropy-error-terms}, let us finish the main proof. Combining \cref{cl:entropy-error-terms} with~\cref{eq:entropy-main-calc}, we get (using~$\delta \ll \gamma$) that 
\begin{align*}
h\left( \{\psi^*_v\}_{v\in V^\ell}| I(\psi^*)\right)
  &\leq 3t \log(pd) + 2 \log(n!_t) + (\delta + \sqrt{\delta} - \gamma^2)n\\
  &\leq 3t \log(pd) + 2 \log(n!_t) - 2\delta n.
\end{align*}
Plugging this back into~\cref{eq:entropy-ub-6}, we get
that~$h(\psi^*) \leq 3t \log(pd) + 3 \log(n!_t) - 2\delta n,$ 
contradicting~\cref{eq:lower bound on h}.
Hence it remains to prove \cref{cl:entropy-error-terms}.
\begin{claimproof}
Let~$\psi \in \Psi$ and observe that~$\{t_v(\psi): v\in V^\ell\setminus I(\psi)\} = [t-1]_0$. Thus
\begin{align*}
M(\psi)
    &=\sum_{v\in V^\ell\setminus I(\psi)}  \logb{(pd)^3(n-t_v(\psi))^2}\\
    &=\sum_{k=0}^{t-1}  \logb{(pd)^3(n-k)^2} =3t\log(pd)+ 2\log(n!_t).
\end{align*}
We now turn to bounding~$N_1(\psi)$. We define~$B' =: B(\psi) \setminus I(\psi)$ and observe that~$\card{B'} \leq \card{B(\psi)} \leq \delta n$.
Further, let~$\mathds{K}= \{t_v(\psi): v\in B'\}$. Enumerate~$\mathds{K} = \{k_1, \ldots, k_{|B'|}\}$ so that~$k_1 \ge \ldots \ge k_{|B'|}$ and observe that~$k_i \leq n-i$ for all~$i \in [|B'|]$, by virtue of the fact that~$t_v(\psi)\leq t \leq n-1$ for all~$v\in B'$ and, as~$B'\cap I(\psi)=\emptyset$, we cannot have that~$t_v(\psi)=t_{v'}(\psi)$ for~$v\neq v'\in B'$. Hence,
\begin{align*}
N_1(\psi)
    &=\sum_{v\in B'} Y_v(\psi) \left(\logb{\frac{10}{d^3}} + 2\logb{\frac{n}{(n-t_v(\psi))}}\right)\\
    &\leq \delta n \logb{\frac{10}{d^3}} + \sum_{\ell=1}^{\delta n} 2\logb{\frac{n}{\ell}}\\
    &\leq \delta n \logb{\frac{10}{d^3}} + 2\delta n \log (n) - 2\log((\delta n)!)\\
    &\leq \delta n \logb{\frac{10}{d^3}} + 2\delta n \left(\log (n) -\logb{\frac{\delta n}{e}} \right)  \\
    &\leq \sqrt{\delta}  n,
\end{align*}
where we used~$(\delta n)! \geq \big(\tfrac{\delta n}{e}\big)^{\delta n}$ in the second to last line.
Finally, let~$U' = U \setminus I(\psi)$ and observe that, since~$\eta \ll \gamma$, we have~$\card{U'} \geq \tfrac{\gamma n}{2}$. Therefore,
\begin{align*}
N_2(\psi)
    &=\sum_{v\in U'} \frac{\beta}{2} \geq \gamma^2 n,
\end{align*}
as claimed.
\end{claimproof}
\end{proof}
\subsection{Counting via comparison}\label{sec:counting via comparison}
In this subsection, we will prove \cref{lem:compare} which we used in \cref{sec:simp} to prove the Local Distribution Lemma  (\cref{lem:LDL}).
Elements of the proof of \cref{lem:compare}  were already sketched in \cref{sec:proofoverview} but before embarking on the details, we outline and reiterate some of the key ideas, ignoring the technicalities in order to  elucidate the general proof scheme. For this discussion, we fix some~$(\e,d)$-super-regular tripartite graph~$\Gamma$, fix~$\ell=1$ and some~$t\in[n]$ close to~$n$. We also fix a vertex~$u\in V^\ell$ which we think of as satisfying  the ``if'' statement in \cref{lem:compare} and some \emph{typical}~$v\in V^\ell$ which we aim to show satisfies the conclusion of \cref{lem:compare}. By typical, we mean that~$v\in V^\ell$ satisfies certain conditions that we have shown whp almost all vertices in~$V^\ell$ satisfy. For example, we can assume that~$\psi_v^*$ has large entropy, when~$\psi^*$ is a uniformly random embedding in~$\Psi_{\hat u}^t(\Gamma_p)$, from the Entropy Lemma (\cref{lem:entropy lemma}). 

Now our aim is to lower bound the number of embeddings~$\psi$ of~$D_t$ that leave~$v$ isolated and we concentrate on the subset  of embeddings that place~$u$ in some triangle (as~$t$ is large we can expect that almost all embeddings do place~$u$ in a triangle). 
Refining further, we will only count  embeddings that place~$u$ in a triangle with an edge that lies in some special set~$F\subset E(\Gamma[V^2, V^3])$. To define~$F$, we begin by concentrating on edges in~$\tr_{u}(\Gamma)\cap\tr_v(\Gamma)$. That is, any edge in~$F$ will form a triangle with \emph{both}~$u$ and~$v$. We then take~$F$ to be the edges in~$\tr_{u}(\Gamma)\cap\tr_v(\Gamma)$ which appear in~$\Gamma_p$. Note that we do \emph{not} require that for an edge~$w_2w_3\in F$, any of the edges~$vw_i$ or~$uw_i$ with~$i=2,3$, lie in~$\Gamma_p$, just that they lie in~$\Gamma$.  

To motivate this definition, we consider a multi-stage revealing process. First, we reveal all edges of~$\Gamma_p$ that are \emph{not} adjacent to~$u$ or~$v$. The  definition of~$F$ comes from the fact that   at this point in the process, any edge in~$F$ has the \emph{potential} to lie in~$\tr_u(\Gamma_p)$ and also~$\tr_v(\Gamma_p)$, depending on which random edges are adjacent to the vertices~$u$ and~$v$. Now note that, in particular, if an  edge~$e=w_2w_3\in F$ \emph{does} end up in~$\tr_u(\Gamma_p)$, then it will contribute to  embeddings that avoid~$v$ and place~$u$ in a triangle. We introduce a weight function~$\zeta$ on~$F$ (we will in fact define it more generally on~$E(\Gamma[V^2, V^3])$)  which precisely counts the contribution to our desired lower bound,  from embeddings which use the triangle~$u\cup e=\{u,w_2,w_3\}$. That is, for all~$w_2w_3\in F$, we have that~$\zeta(w_2w_3)$ encodes the number of embeddings of~$D_{t-1}$ (with~$t-1$  triangles) in~$\Gamma$, that avoid~$v$ \emph{and} the vertices~$u,w_2,w_3$. Therefore, as we can assume~$F$ is large (as~$v$ is typical, using~\cref{lem:many common neighbours}), our desired conclusion will follow if we can lower bound the~$\zeta$ values in (some subset of)~$F$. 

The central idea of the proof is that we \emph{can} lower bound~$\zeta$ values in~$F$ by reasoning about embeddings that place~$v$ in a triangle (and avoid~$u$). Indeed, if we consider a uniformly random embedding~$\psi^*\in\Psi^t_{\hat u}(\Gamma_p)$, as~$v$ is typical, we know from \cref{lem:entropy lemma}, that the random variable~$\psi_v^*$, which encodes the triangle  containing~$v$  in~$\psi(D_t)$, has high entropy. Appealing to \cref{lem:large-entropy} then implies that the distribution of~$\psi^*_v$ in~$\tr_v(\Gamma_p)$ is close to uniform and hence for almost all edges~$f\in \tr_v{(\Gamma_p)}$, we have that~$\Pro{\psi^*_v=f}$ is large (in that it is close to the 
average). Moreover, we have that~$\Pro{\psi^*_v=f}$ is directly proportional to~$\zeta(f)$ by the definition of~$\zeta$. Therefore, using~\cref{lem:two stage reveal}\cref{lem:triangles with F} (and observing that the~$\zeta$ values do  not depend on random edges adjacent to~$u$ or~$v$), we can see that we must have a significant proportion of the edges in~$F$ having large~$\zeta$ values. Indeed, if this were not the case, then it would be very unlikely that almost all edges in~$\tr_v(\Gamma_p)$ have large~$\zeta$ values.

We  can therefore conclude that there is some subset~$F_L\subset F$ of half the edges in~$F$ such that~$\zeta(f)$ is large for all $f\in F_L$.   Finally, through another application of~\cref{lem:two stage reveal}\cref{lem:triangles with F}, we can show that many edges in~$F_L$ end up in~$\tr_u(\Gamma_p)$ and therefore contribute to  the lower bound on the number of embeddings that leave~$v$ isolated.  We now give the full details of the proof.

\begin{proof}[Proof of \cref{lem:compare}]
Choose~$0<\e,\tfrac{1}{C} \ll \e' \ll \eta \ll \beta' \ll \beta \ll \tfrac{1}{L} \ll \alpha, d, \tfrac{1}{K}$. Fix~$\Gamma$,~$p=p(n)$,~$\ell\in[3]$,~$(1-\eta)n\leq t < n$,~$\uu=(u_1\ldots,u_{\ell-1})\in \pzc{V}$ and~$u\in V^\ell$ as in the statement of \cref{lem:compare}. We define~$J \coloneq [3]\setminus \{\ell\}$ and label the indices of~$J$ by~$j_1,j_2 \in [3]$ so that~$J = \{j_1,j_2\}$.

Now for a subgraph~$G$ of~$\Gamma$, we will make some  definitions relative to~$G$ and posit certain properties of~$G$. Our proof will then proceed by first proving that any~$G$  satisfying  all the properties,  satisfies the desired conclusion of the lemma. After this we will show that whp we can take that~$\Gamma_p$ satisfies all the defined properties, which will complete the proof. Herein, we fix some subgraph~$G$ of~$\Gamma$ for the discussion. Our first property comes from the statement of the lemma. 
\begin{enumerate}[label=\textbf{(Q.\arabic*})]
  \item\label{en:Q-granted} We have 
  \[\card{\Psi_{\hat \uu,\hat u}^t(G)} \geq (1-\eta)^n (pd)^{3t}  ((n-1)!_{t})^{\ell} (n!_t)^{3-\ell}.\]
\end{enumerate}

 For~$v\in V^\ell$, we now define the set of edges  which lie in~$G$ and in the common neighbourhood (with respect to~$\Gamma$) of both~$u$ and~$v$. In symbols, 
\begin{equation}\label{eq:F(v)def} F(v)\coloneq\tr_u(\Gamma_{\hat \uu})\cap \tr_v(\Gamma_{\hat \uu})\cap E(G)\subseteq V_{\hat\uu}^{j_1} \times V_{\hat\uu}^{j_2}.\end{equation}
Note that here (and throughout this proof), for convenience, we will think of edges in~$e=\{y_1,y_2\}\in E(\Gamma[V^{j_1} \cup V^{j_2}])$ as \emph{ordered} pairs~$(y_1,y_2)\in V^{j_1}\times V^{j_2}$.

Now let~$\psi^*$ be chosen uniformly from~$\Psi_{\hat\uu,\hat u}^t(G)$.
We define the following subsets of~$V^\ell$, recalling the definition of~$H(n,p,d)$ from~\cref{eq:hhat}. 
\begin{align*}Z_1 &\coloneq \{v\in V^\ell: h(\psi^*_v| Y_v(\psi^*) = 1) \geq H(n,p,d)-\beta'\}, \\
 Z_2 &\coloneq \left\{v\in V^\ell: \card{\tr_v(G)}=(1 \pm \e')(pd)^3n^2\right\}, \\
    Z_3 &\coloneq \left\{v\in V^\ell:\card{F(v)} \geq \frac{d^5pn^2}{4}\right\}, \\
    Z &\coloneq Z_1 \cap Z_2 \cap Z_3.
\end{align*}
Our second property of~$G$ posits that~$Z$ is large. 

\begin{enumerate}[label=\textbf{(Q.\arabic*})] \addtocounter{enumi}{1}
  \item\label{en:Q-largeZ} If \cref{en:Q-granted} holds in~$G$ then
  \[|Z| \geq (1-\alpha)n.\]
\end{enumerate}

We now define the weight functions we will be interested in. For~$v\in V^\ell\setminus \{u\}$ and~$(w_1,w_2) \in V_{\hat\uu}^{j_1} \times V_{\hat\uu}^{j_2}$, define~$\zeta_v(w_1,w_2)$ to be~$t$ times the number of labelled embeddings of~$D_{t-1}$ into~$G_{\hat\uu,\hat u,\hat v}$ in which both~$w_1$ and~$w_2$ are isolated vertices. That is,
\begin{equation} \label{eq: zeta def}
    \zeta_v(w_1,w_2) \coloneq t \cdot \card{\Psi^{(t-1)}_{\hat w_1,\hat w_2} \left(G_{\hat \uu,\hat u,\hat v}\right) }.
\end{equation}
For our last property of~$G$,  we need a further definition. 
For~$v\in V^\ell$, consider~$F(v)$ as in~\cref{eq:F(v)def}. We split~$F(v)$ in half according to the values of the weight function~$\zeta_v$. That is we partition~$F(v)$ into~$F_S(v)$ and~$F_L(v)$  so that~$\zeta(y_1,y_2) \leq \zeta(z_1,z_2)$ for all~$(y_1,y_2) \in F_S(v)$ and~$(z_1,z_2) \in F_L(v)$, and~$\card{F_S(v)}=\card{F_L(v)}\pm 1$. Our final property gives that~$G$ has many triangles containing~$u$ (resp.~$v$) and the edges of~$F_L(v)$ (resp.~$F_S(v)$).

\begin{enumerate}[label=\textbf{(Q.\arabic*})] \addtocounter{enumi}{2}
  \item\label{en:Q-manytriangles} If~$v\in Z$,  then
  \[\card{F'} \geq \frac{d^5p^3n^2}{20}, \]
  for both~$F'=F_L(v)\cap \tr_u(G)$ and~$F'=F_S(v)\cap \tr_v(G)$. 
\end{enumerate}

We now proceed by taking that~$G$ satisfies  \cref{en:Q-largeZ} and \cref{en:Q-manytriangles} and showing that  it then satisfies the desired conclusion of the lemma. We will do this by proving that if~$G$ satisfies \cref{en:Q-granted} then  every~$v\in Z$ satisfies
\[\card{\Psi_{\hat\uu,\hat v}^t(G)} \geq \left(\frac{d}{10}\right)^2\cdot \card{\Psi_{\hat \uu,\hat u}^t(G)},\]
which in combination with the fact that~$G$ satisfies \cref{en:Q-largeZ}, gives what is needed. So let us fix some~$v\in Z$.  We define the following sets of embeddings.  
\begin{align*}
    \Psi_{\hat u\hat v}&\coloneq \Psi_{\hat \uu,\hat u}^t(G) \cap \Psi_{\hat\uu,\hat v}^t(G), \\
    \Psi_{v\hat u}&\coloneq \Psi_{\hat \uu,\hat u}^t(G) \setminus  \Psi_{\hat u\hat v} \text{ and } \\
     \Psi_{u\hat v} &\coloneq \Psi_{\hat \uu,\hat v}^t(G) \setminus  \Psi_{\hat u\hat v}.
\end{align*}
In words,~$\Psi_{\hat u\hat v}$ consists of those embeddings which leave both~$u$ and~$v$ isolated whilst embeddings in~$\Psi_{v\hat u}$ leave~$u$ isolated but have~$v$ contained in a triangle, and vice versa for~$\Psi_{u\hat v}$. Clearly, we have%
\begin{align*}
    \card{\Psi_{\hat \uu,\hat u}^t(G)}&=\card{ \Psi_{\hat u \hat v}} +\card{\Psi_{v\hat u}}, \text{ and}\\
    \card{\Psi_{\hat \uu,\hat v}^t(G)}&=\card{ \Psi_{\hat u\hat v}} +\card{\Psi_{u\hat v}}.
\end{align*}
If~$\card{\Psi_{\hat u \hat v}}\geq \big(\tfrac{d}{10}\big)^2 \card{\Psi_{\hat \uu,\hat u}^t(G)}$, we are done and so we may assume that
\begin{equation}\label{eq:lb-Phi-v}
\card{\Psi_{v\hat u}}
  \geq \left(1- \left(\frac{d}{10}\right)^2\right) \card{\Psi_{\hat \uu,\hat u}^t(G)}
  \geq  \frac{1}{2}\card{\Psi_{\hat \uu,\hat u}^t(G)}.
\end{equation}
In what remains, we will compare the sizes of~$\Psi_{v\hat u}$ and~$\Psi_{u\hat v}$.
Let~$\zeta=\zeta_v$ be the weight function as  defined in~\cref{eq: zeta def}. 
Observe that
\begin{align*}
\card{\Psi_{v\hat u}} &= \sum_{(y_1,y_2) \in \tr_{v}\left(G_{\hat\uu}\right)} \zeta (y_1,y_2), \text{ and }\\
\card{\Psi_{ u\hat v}} &= \sum_{(y_1,y_2) \in \tr_{u}\left(G_{\hat \uu}\right)} \zeta (y_1,y_2).
\end{align*}
Recall that we took~$\psi^*$ to be a uniformly random embedding in~$\Psi_{\hat\uu,\hat u}^t(G)$.  Note that~$\psi^*_v|Y_v(\psi^*) = 1$ is a random variable taking values in~$S:=\tr_{v}\left(G_{\hat\uu}\right)$ and the distribution of~$\psi^*_v|Y_v(\psi^*) = 1$ is determined by~$\zeta$. That is, for all $(z_1,z_2)\in S$,
\begin{equation} \label{eq:followzeta}
    \Pro{\psi^*_v=(z_1,z_2)|Y_v(\psi^*)=1} = \frac{\zeta(z_1,z_2)}{\sum_{(y_1,y_2)\in S } \zeta(y_1,y_2)}=\frac{\zeta(z_1,z_2)}{\card{\Psi_{v\hat{u}}}}. \end{equation}
Moreover, as~$v \in Z\subseteq Z_2$, we have that~$\log(\card{S})\leq \log(1+\e')+H(n,p,d)$ and therefore, using also that~$v\in Z\subseteq Z_1$, we can apply \cref{lem:large-entropy} (with~$2\beta'$ replacing~$\beta'$)  to obtain some set~$W^* \subseteq  S=\tr_{v}\left(G_{\hat\uu}\right)$ with the following properties (using \eqref{eq:followzeta} to unpack the conclusions here):
\begin{enumerate}[(i)]
    \item\label{cond: most of weight}
~$\mathlarger{\sum}_{(w_1,w_2)\in W^*}\zeta(w_1,w_2)\geq (1-\beta)\card{\Psi_{v\hat{u}}}$;
    \item\label{cond: almost uniform}
      There exists some value~$\bar{\zeta}$ such that for each~$(w_1,w_2)\in W^*$, we have that \[\zeta(w_1,w_2)=(1\pm \beta)\bar{\zeta};\]
    \item\label{cond:big set}
      We have~$(1-\beta)|S| \leq |W^*| \leq |S|$.
\end{enumerate}
Therefore we can estimate the size of~$\Psi_{v\hat{u}}$ using \cref{cond:big set,cond: most of weight,cond: almost uniform} in that order, as follows:
\begin{align}
 \nonumber \card{\Psi_{v\hat{u}}}
      &\leq\left(\frac{1}{1-\beta}\right) \sum_{(w_1,w_2)\in W^*}\zeta(w_1,w_2)\\
   \nonumber    &\leq \left(\frac{1+\beta}{1-\beta}\right) |W^*|\bar{\zeta} \\ 
       &\leq \left(\frac{1+\beta}{1-\beta}\right) |S|\bar{\zeta}  
      \leq 2 \bar{\zeta} (pd)^3n^2.\label{eq:ub-Phi-v}
\end{align}
In the last inequality, we used that~$|S| =\card{\tr_{v}\left(G_{\hat\uu}\right)}\leq (1 +\e') (pd)^3n^2$ since~$v \in Z\subseteq Z_2$.

We are now going to lower bound~$\card{\Psi_{u\hat v}}$ in a similar way. However,   the entropy argument above only shows that~$\zeta$ is `well-behaved' on~$S=\tr_{v}\left(G_{\hat\uu}\right)$ but nothing about~$\tr_{u}\left(G_{\hat\uu}\right)$. Using \cref{en:Q-manytriangles} though, we can infer though that~$\zeta$ is `well-behaved' on a large part of~$F(v)$, as defined in~\cref{eq:F(v)def}. Recall also our definitions of~$F_L(v)$ and~$F_S(v)$. 
\begin{claim}\label{claim:zeta-FL}
 We have~$\zeta(y_1,y_2) \geq (1-\beta)\bar\zeta$ for all~$(y_1,y_2) \in F_L(v)$.
\end{claim}
\begin{claimproof}

By \cref{en:Q-manytriangles}, we have that
\[\card{\tr_v{\left(G_{\hat\uu}\right)} \cap F_S(v)} \geq  \frac{d^5p^3n^2}{20},\]
noting that~$\tr_v{\left(G_{\hat\uu}\right)} \cap F_S(v)=\tr_v{\left(G\right)} \cap F_S(v)$ due to the fact that~$F_S(v)\subset E(\Gamma_{\hat \uu})$. 
Furthermore, it follows from \cref{cond:big set} and the fact that~$v\in Z\subseteq  Z_2$, that
\[\card{\tr_v{\left(G_{\hat\uu}\right)} \setminus W^*} \leq \beta \card{\tr_v{\left(G_{\hat\uu}\right)}} \leq 2\beta (pd)^3n^2.\]
Hence, as~$\beta \ll d$, we can conclude that~$W^*\cap F_S(v)\neq \emptyset$ and so
\[(1-\beta)\bar\zeta\leq \min_{(y_1,y_2) \in W^*} \zeta(y_1,y_2)
  \leq \max_{(y_1,y_2) \in F_S(v)} \zeta(y_1,y_2)
  \leq \min_{ (y_1,y_2) \in F_L(v)} \zeta(y_1,y_2),\]
using \cref{cond: almost uniform} in the first inequality.
\end{claimproof}

%
We now appeal to \cref{en:Q-manytriangles} to lower bound the size of~$\card{\Psi_{u\hat{v}}}$ as follows:
\begin{align}
\nonumber  \card{\Psi_{u\hat{v}}}
    &= \sum_{(y_1,y_2) \in\tr_{u}\left(G_{\hat \uu}\right)} \zeta(y_1,y_2)\\
\nonumber    &\geq \sum_{(y_1,y_2) \in \tr_{u}\left(G_{\hat \uu}\right) \cap F_L(v)} \zeta(y_1,y_2)\\ 
\nonumber    &\geq (1-\beta) \bar\zeta\card{\tr_{u}\left(G_{\hat \uu}\right) \cap F_L(v)}\\ \label{eq:lb-Phi-u}
    &\geq \frac{\bar\zeta d^5p^3 n^2}{25},
\end{align}
where we used  \cref{claim:zeta-FL}. 
Putting~\cref{eq:lb-Phi-v},~\cref{eq:ub-Phi-v} and~\cref{eq:lb-Phi-u} together, we get that
\[
 \card{\Psi_{\hat \uu,\hat v}^t(G)}
    \geq  \card{\Psi_{u\hat{v}}}
    \geq \frac{\bar \zeta d^5p^3 n^2}{25}
    \geq \frac{d^2}{50}\card{\Psi_{v\hat{u}}}
    \geq \frac{d^2}{100}  \card{\Psi_{\hat \uu,\hat u}^t(G)},\]
as required.

It remains to verify that for~$G=\Gamma_p$ the statements in \cref{en:Q-largeZ} and \cref{en:Q-manytriangles} hold with probability at least~$1-n^{-K}$. We start with \cref{en:Q-largeZ}, which follows simply from \cref{cor:fixed vertex triangles} and  \cref{lem:entropy lemma,lem:many common neighbours}. Indeed, from those results (using that~$\tfrac{1}{L}\ll \tfrac{1}{K}$) and a union bound, with probability at least~$1-n^{-2K}$, we have  that~$|Z_2|\geq (1-\e')n$,~$|Z_3|\geq (1-2\e)n$ and if \cref{en:Q-granted} holds in~$G=\Gamma_p$ then~$|Z_1|\geq (1-\beta')n$. It then follows easily by our choice of constants that the statement of \cref{en:Q-largeZ} holds in~$G=\Gamma_p$ with probability at least~$1-n^{-2K}$. 

For \cref{en:Q-manytriangles}, we will appeal to \cref{lem:two stage reveal}\cref{lem:triangles with F}. Note that  for a fixed~$v\in V^\ell\setminus \{u\}$ the value of~$\zeta_v(w_1,w_2)$ for~$(w_1,w_2) \in V_{\hat\uu}^{j_1} \times V_{\hat\uu}^{j_2}$ does not depend on the random status of any of the edges containing~$u$ or~$v$. 
Indeed, our definition of~$\zeta_v$ counts only embeddings that leave both~$u$ and~$v$ isolated. We also have that the random set of edges~$F(v)$, as defined in~\cref{eq:F(v)def}, is independent of the random status of any edges adjacent to~$u$ or~$v$. Consequently, in the language of \cref{lem:two stage reveal}, we have that the random sets of edges~$F_L(v)$ and~$F_S(v)$ are \emph{determined} by~$(\Gamma_{\hat u})_p$ (resp.~$(\Gamma_{\hat v})_p$). Therefore, for a fixed~$v\in V^\ell$, two applications of \cref{lem:two stage reveal}\cref{lem:triangles with F} (once for~$u$ and~$F_L(v)$ and once for~$v$ and~$F_S(v)$) give that with probability at least~$1-n^{-(2K+1)}$, we have that \cref{en:Q-manytriangles} holds for~$v$. Here we  used that~$v\in Z\subseteq Z_3$ implies that~$\card{F_L(v)},\card{F_S(v)}\geq \tfrac{d^5pn^2}{10}$. Taking a union bound over all~$v\in V^\ell$, we have that \cref{en:Q-manytriangles} holds in~$G=\Gamma_p$ for all~$v\in V^\ell$, with probability at least~$1-n^{-2K}$. A final union bound gives that with probability at least~$1-n^{-K}$, both \cref{en:Q-largeZ} and \cref{en:Q-manytriangles} hold in~$G=\Gamma_p$ which completes the proof. 
\end{proof}


\section{Stability for a fractional version of the Hajnal--Szemer\'edi theorem}
\label{sec:HSzFrac}


In this section we discuss some fractional variants of the Hajnal--Szemer\'edi theorem for clique factors  (\cref{thm:HajSze-factor} with $x=0$). We will use the results here in our proof reducing \cref{thm:main} to \cref{thm:main-super-reg} in \cref{sec:reduction}. The starting point  is to relax the notion of a~$K_k$-factor to that of a \emph{fractional~$K_k$-factor}. That is, for a graph~$G$, a fractional~$K_k$-factor in~$G$ is a  weighting~$\omega:K_k(G)\rightarrow \mathbb{R}_{\ge 0}$ such that~$\sum_{K \in K_k(G,u)}\omega(K) = 1$ for all~$u \in V(G)$.  If all cliques~$K\in K_k(G)$ are assigned weights in~$\{0,1\}$, we  recover the notion of a~$K_k$-factor and so the definition of a fractional~$K_k$-factor is more general. However, from an extremal point of view, the same minimum degree condition is needed to force both objects. Indeed, focusing on the case when~$n\in k\NN$, the Hajnal--Szemer\'edi theorem (\cref{thm:HajSze-factor} with $x=0$) gives that graphs~$G$ with~$n$ vertices and minimum degree at least~$\big(\tfrac{k-1}{k}\big)n$  have~$K_k$-factors and hence fractional~$K_k$-factors whilst the same construction proving tightness for $K_k$-factors can be used to show tightness for fractional factors, as we now show. Take a graph~$G$ to be a complete graph with~$n\in k \NN$ vertices with a clique of size~$\tfrac{n}{k}+1$ removed to leave an independent set of vertices~$I$. Therefore~$G$ has minimum degree~$\delta(G)=\big(\tfrac{k-1}{k}\big)n-1$ and suppose for a contradiction that~$G$ has a fractional~$K_k$-factor given by a weight function~$\omega:K_k(G)\rightarrow \mathbb{R}_{\ge 0}$. Then we have that~$\sum_{K \in K_k(G,u)}\omega(K) = 1$ for all~$u\in V(G)$ and note that for~$w\neq w'\in I$, we have that~$K_k(G,w)\cap K_k(G,w')=\emptyset$  as~$I$ is an independent set. Therefore
\[\sum_{K \in K_k(G)}\omega(K) \geq \sum_{w\in I}\sum_{K \in K_k(G,w)}\omega(K)\ge |I|=\frac{n}{k}+1\]
but we also have that 
\[\sum_{K \in K_k(G)}\omega(K)=\frac{1}{k}\sum_{u\in V(G)}\sum_{K \in K_k(G,u)}\omega(K)=\frac{n}{k},\]
a contradiction.  The results of this section, which may be of independent interest, will give stability for this phenomenon, showing that if we avoid the construction detailed above (and other similar constructions), by imposing that~$\alpha(G)\le \big( \tfrac{1}{k}-\eta\big) n$ for some~$\eta>0$, then a weaker minimum degree condition of~$\delta(G)\ge \big(\tfrac{k-1}{k}-\gamma\big)n$ for some~$\gamma=\gamma(\eta)>0$,  suffices to force a fractional~$K_k$-factor.  

We will  use that the existence of a fractional~$K_k$-factor can be encoded by a linear program whose dual is a covering linear program which assigns weights to vertices such that every clique is sufficiently `covered'. The duality theorem from linear programming will then be used to transfer between the two settings.

\begin{thm}[stability for fractional Hajnal--Szemer\'edi]\label{thm:HSzFrac0}
For  every~$\eta >0$ and~$2\le k\in \NN$, there is some~$\gamma > 0$ such that the following is true for all~$n \in \N$.
Let~$G$ be an~$n$-vertex graph with~$\delta(G) \geq \big(\tfrac{k-1}{k}-\gamma\big)n$ and~$\alpha(G) < \big(\tfrac{1}{k}-\eta\big)n$.
Then~$G$ contains a fractional~$K_k$-factor.
\end{thm}

\begin{proof}  We will prove the theorem for $\gamma = \tfrac{\eta}{8^k(k!)^2}$.
Observe that the existence of a fractional $K_k$-factor is the same as saying that the value of the following packing linear program is $\tfrac{n}{k}$. We ask for non-negative real weights on the elements of $K_k(G)$ with maximum sum, subject to the condition that the total weight on copies of $K_k$ at any given vertex is at most $1$. The dual of this is the covering linear program in which we place nonnegative weights on the vertices of $G$, and are aiming at minimising their sum, subject to the constraint that the total weight on the vertices of each element of $K_k(G)$ is at least $1$. The strong duality theorem for linear programs implies that these two linear programs have the same optimal objective function value. So it is enough to show that the latter linear program has optimal objective function value at least $\tfrac{n}{k}$ (and thus exactly $\tfrac{n}{k}$), which we do inductively. More precisely, we want to prove the following claim by induction on $k$. We define $z_2=3$ and inductively $z_k=8k^2z_{k-1}$ for $k\ge 3$.

\begin{claim}\label{cl:HSz} Given any $k\ge2$ and $\gamma>0$, suppose that $G$ is an $n$-vertex graph with minimum degree at least $\big(\tfrac{k-1}{k}-\gamma\big)n$ and no independent set of size $\big(\tfrac{1}{k}-z_k\gamma\big)n$. Suppose $c:V(G)\to \R_{\geq 0}$ is any weight function such that for each $Q\in K_k(G)$ we have $\sum_{v\in Q}c(v)\ge1$. Then $\sum_{v\in V(G)}c(v)\ge\tfrac{n}{k}$.
\end{claim}

\begin{claimproof}
It is convenient to let the vertices of $G$ be $v_1,\dots,v_n$ in order of decreasing weight, i.e.\ $c(v_i)\ge c(v_j)$ if $i\le j$. If $\sum_{i\in[n]}c(v_i)\ge\tfrac{n}{k}$ there is nothing to prove, so we can assume the sum is less than $\tfrac{n}{k}$, and hence in particular that $c(v_n)<\frac1k$. We next argue that we can assume $c(v_n)=0$. Indeed, if $c(v_n)>0$, then we can define a new weight function by $c'(v_i) \coloneq \tfrac{1}{k}+\mu\big(c(v_i)-\tfrac{1}{k}\big)$ for all $i \in [n]$, where~$\mu$ is chosen so that $c'(v_n)=0$. Here, $\mu>1$ because $c(v_n)<\frac1k$. Observe that the~$v_i$ remain ordered by weight with this new weight function.
We have
\begin{align*}
\sum_{i\in[n]}c'(v_i)
	&= 
	\tfrac{n}{k} + \mu \sum_{i\in[n]} \left(c(v_i) - \tfrac{1}{k}\right)\\
	&= \sum_{i \in [n]} c(v_i) + (\mu - 1) \left( \sum\nolimits_{i\in[n]} c(v_i) - \tfrac{n}{k} \right) < \sum_{i \in [n]} c(v_i).
\end{align*}
However, for every $Q \in K_k(G)$,
\[\sum_{v\in Q}c'(v) = \sum_{v\in Q} \left( \tfrac{1}{k}+\mu\big(c(v)-\tfrac{1}{k}\big) \right) = 1 + \mu \left(\sum\nolimits_{v\in Q}c(v) - 1 \right) \geq 1.\]
Therefore, $c'$ also satisfies the condition of \cref{cl:HSz} and we thus can assume $c(v_n)=0$.

We are now in a position to prove the base case $k=2$. Since $v_n$ has at least $\big(\tfrac12-\gamma\big)n$ neighbours, and $c(v_n)=0$, we see that for each $i$ such that $v_iv_n\in E(G)$, we have $c(v_i)=1$. In particular, $c(v_i)=1$ for each $i\le\big(\tfrac12-\gamma\big)n$. Furthermore, the vertices $\{v_i:i\ge\tfrac{n}{2}+2\gamma n\}$ do not form an independent set, so there is an edge within this set. At least one endpoint of this edge has weight at least $\tfrac12$. As vertices are ordered by weight, this implies that each vertex $v_i$ with $\tfrac{n}{2}-\gamma n<i<\tfrac{n}{2}+2\gamma n$ has weight at least $\tfrac12$. Summing, we obtain weight at least $\tfrac{n}{2}$ as desired.

Next, we prove the induction step; let $k\ge 3$. We build a copy of $K_k$ containing $v_n$ as follows: we take $u_1=v_n$, and then for each $2\le i\le k-2$ in succession, we take $u_i$ to be the common neighbour of $u_1,\dots,u_{i-1}$ with smallest weight. From the minimum degree condition, when we choose $u_i$ there are at least $n\left(1-(i-1)\big(\tfrac{1}{k}+\gamma\big)\right)$ common neighbours to choose from; in particular, the common neighbourhood of all $k-2$ vertices we choose has size at least $\tfrac{2n}{k}-(k-2)\gamma n$. Now consider the last $\big(\tfrac{1}{k}-k(k-1)\gamma\big)n$ of these common neighbours. Since $z_k\ge k(k-1)$, they do not form an independent set, so contain an edge $u_{k-1}u_k$. Since $\sum_{i=1}^kc(u_i)\ge 1$, and $c(u_1)=0$, one of these vertices has weight at least $\tfrac{1}{k-1}$. In particular, $c(v_i)\ge\tfrac{1}{k-1}$ whenever $i\le\big(\tfrac{1}{k}+(k-1)^2\gamma\big)n$.

Now let $c^* \coloneq c(v_{n/k-(k-1)\gamma n})$, and let $G'$ denote the subgraph of $G$ induced by vertices $v_i$ with $i\ge\big(\tfrac{1}{k}+(k-1)^2\gamma\big)n$. If $c^*\geq 1$ then we have
\[\sum_{i\in[n]}c(v_i)\ge \tfrac{n}{k}-(k-1)\gamma n+\tfrac{1}{k-1}\cdot k(k-1)\gamma n>\tfrac{n}{k}\]
and we are done; so we can assume $c^*<1$.
If $Q$ is any copy of $K_{k-1}$ in $G'$, then $Q$ has a common neighbourhood in $G$ of size at least $\tfrac{n}{k}-(k-1)\gamma n$, and so in particular $Q$ extends to a copy of $K_k$ in $G$ by adding a vertex whose weight is at most $c^*$. Thus the function $c'(u) \coloneq \tfrac{1}{1-c^*}c(u)$ on $V(G')$ is a weight function on $V(G')$ taking values in $\R_{\geq 0}$ and such that $\sum_{u\in Q}c'(u)\ge 1$ for each $Q\in K_{k-1}(G')$. Furthermore every vertex in $G'$ has at most $\tfrac{n}{k}+\gamma n$ non-neighbours in $G$, at most all of which are in $G'$, so the minimum degree of $G'$ is at least $\tfrac{(k-2)n}{k}-((k-1)^2+1)\gamma n$. Since $v(G')=\tfrac{(k-1)n}{k}-(k-1)^2\gamma n$, we have $\delta(G')\ge\tfrac{k-2}{k-1}v(G')-\gamma'v(G')$ where $\gamma' \coloneq 2k^2\gamma$. Furthermore $G'$ has no independent set of size
\[\tfrac{1}{k}n-z_k\gamma n=\tfrac{1}{k}n-4z_{k-1}\gamma'  n\le\tfrac{1}{k-1}v(G')-z_{k-1}\gamma'v(G')\,.\]
We are therefore in a position to apply the induction hypothesis (that is, \cref{cl:HSz} for $k-1$) to $G'$, with $\gamma'$ replacing $\gamma$. We conclude that
\[\sum_{u\in V(G')}c'(u)\ge\tfrac{1}{k-1}v(G') \geq \frac{\big(1-\tfrac{1}{k} - (k-1)^2 \gamma\big)n}{k-1} = \big(\tfrac{1}{k} - (k-1)\gamma\big)n\]
and so
\begin{align*}
\sum_{i\in[n]}c(v_i)
	&\ge c^*\big(\tfrac{1}{k}-(k-1)\gamma\big)n+\tfrac{1}{k-1}\cdot k(k-1)\gamma n+(1-c^*) \cdot \big(\tfrac{1}{k} - (k-1)\gamma \big)n\\
	&=\big(\tfrac{1}{k} -(k-1)\gamma\big)n + k\gamma n >\tfrac{n}{k},
\end{align*}

as desired.
\end{claimproof}
This completes the proof by strong LP-duality.
\end{proof}

Note that we obtain from this proof a little more: the unique optimal cover is the uniform cover (since after assuming $c(v_n)<\tfrac{1}{k}$ we eventually conclude the total weight is strictly bigger than $\tfrac{n}{k}$). However we will not need this fact. We will also need only the $k=2$ and $k=3$ cases, but for future use give the general result.

Next, we need some modifications of \cref{thm:HSzFrac0}.  First we want to be able to set (potentially different but close to uniform) weights~$\lambda(u)$ for each~$u\in V$ and obtain a weighting~$\omega:K_k(G) \to \R_{\geq 0}$ such that~$\sum_{K \in K_k(G,u)}\omega(K) = \lambda(u)$ for all~$u \in V(G)$. The case of fractional~$K_k$-factors  corresponds to setting~$\lambda(u)=1$ for all~$u\in V(G)$.

\begin{cor} \label{thm:HSzFrac}
For every integer $k \geq 2$ and every $\eta >0$, there is some $\gamma > 0$ such that the following is true for all $n \in \N$.
Let $G$ be an $n$-vertex graph with $\delta(G) \geq \big(\tfrac{k-1}{k}-\gamma\big)n$ and $\alpha(G) < \big(\tfrac{1}{k}-\eta\big)n$.
Let $\lambda: V(G) \to \N$ be a weight function with $\lambda(u) = (1 \pm \gamma) \tfrac{1}{n}\sum_{v \in V(G)} \lambda(v)$ for all $u \in V(G)$.
Then there is a weight function $\omega:K_k(G) \to \R_{\geq 0}$ such that $\sum_{K \in K_k(G,u)}\omega(K) = \lambda(u)$ for all $u \in V(G)$. 
\end{cor}

\begin{proof}
Fix some~$2\le k\in \NN$ and~$\eta>0$. Choose~$0\ll\gamma\ll \gamma'\ll \eta$.  Now let~$G$ and~$\lambda$ be as in the statement of the corollary.  
 We define an auxiliary graph~$H$ by blowing-up every~$v \in V(G)$ to an independent set of size~$\lambda(v)$ (that is, every edge is replaced by a complete bipartite graph). Then, with~$N \coloneq v(H) = \sum_{v \in V(G)} \lambda(v)$, we have~$\delta(H) \geq \big(\tfrac{k}{k-1} -  \gamma'\big)N$ and~$\alpha(H) \leq \big(\tfrac{1}{k} - \tfrac{\eta}{2}\big)N$. Hence, we can apply \cref{thm:HSzFrac0} to~$H$ and obtain a weight function~$\omega_H: K_k(H) \to \R_{\geq 0}$ such that~$\sum_{K' \in K_k(G,x)}\omega_H(K') = 1$ for all~$x\in V(H)$. We define~$\omega: K_k(G) \to \R_{\geq 0}$ by~$\omega(K) = \sum_{K' \in K_k(H[K])} \omega_H(K')$, where~$H[K]$ is the subgraph of~$H$ induced by the blown-up vertices of~$K$.  This weight function~$\omega$  satisfies the desired conditions. 
\end{proof}

 We now extend yet further to  guarantee an integer-valued weight-function~$\omega:K_k(G)\rightarrow \NN$. In order for this to work, we need that  our function~$\lambda$ assigns each vertex a sufficiently large weight. In applications this will be guaranteed as our weights~$\lambda$ will be proportional to the number of vertices~$n$ of a host graph but \cref{thm:HSzFrac-integer} will actually be applied to the reduced graph~$R$ after applying the regularity lemma to the host graph and hence the number of vertices of~$R$ (the parameter~$n$ in \cref{thm:HSzFrac-integer}) will be bounded by some constant. 

\begin{thm}[stability for fractional Hajnal--Szemer\'edi with integer weights]\label{thm:HSzFrac-integer}
For every integer $k \geq 2$ and every $\eta >0$, there is some $\gamma > 0$ such that the following is true for all $n \in \N$.
Let $G$ be a connected $n$-vertex graph with $\delta(G) \geq \big(\tfrac{k-1}{k}-\gamma\big)n$ and $\alpha(G) < \big(\tfrac{1}{k}-\eta\big)n$.
Let $\lambda: V(G) \to \N$ be a weight function such that $\lambda(u) = \big(1 \pm \frac{\gamma}{2}\big) \tfrac{1}{n}\sum_{v \in V(G)} \lambda(v)$, $\lambda(u) \geq n^{2k}$ for all $u \in V(G)$ and $k$ divides $\sum_{v \in V(G)} \lambda(v)$.
Then there is a weight function $\omega:K_k(G) \to \N_0$ such that $\sum_{K \in K_k(G,u)}\omega(K) = \lambda(u)$ for all $u \in V(G)$.
\end{thm}

Note that for $k \geq 3$ the requirement that~$G$ is connected is readily implied by the minimum degree condition in this theorem.

\begin{proof}[Proof of \cref{thm:HSzFrac-integer}]
Suppose that $k, \eta, G$ and $\lambda$ are given as in the statement and suppose that $\gamma$ is small enough to apply \cref{thm:HSzFrac} and $\gamma \ll \eta/k$.
We will construct $\omega$ in three steps.
Define $\lambda': V(G) \to \N$ by $\lambda'(u) = \lambda(u) - k|K_k(G,u)|n^k \geq 0$. By \cref{thm:HSzFrac}, there is some weight function $\omega':K_k(G) \to \R_{\geq 0}$ such that $\sum_{K \in K_k(G,u)}\omega'(K) = \lambda'(u)$ for all $u \in V(G)$.
We define $\omega'': V(G) \to \N_0$ such that, for each $K \in K_k(G)$,
\begin{enumerate}
\item $\omega''(K) \in \left\{\lfloor \omega'(K) + kn^k \rfloor, \lceil \omega'(K) + kn^k \rceil \right\}$, and
\item $k \sum_{K \in K_k(G)} \omega''(K) = \sum_{v \in V(G)} \lambda(v).$
\end{enumerate}
Note that this is possible since by construction the unrounded sum satisfies $(ii)$ and since $k$ divides $\sum_{v \in V(G)} \lambda(v)$.
Furthermore, for each $u \in V(G)$, we have $\sum_{K \in K_k(G,u)}\omega''(K)=\lambda(u) \pm n^{k-1}$ (since the unrounded sum would be exactly correct and $|K_k(G,u)| \leq n^{k-1}$).

Finally, we obtain $\omega$ from $\omega''$ via the following iterative process. As long as possible, we identify pairs $u,v \in V(G)$ such that $\sum_{K \in K_k(G,u)} \omega''(K) > \lambda(u)$ and $\sum_{K \in K_k(G,v)} \omega''(K) < \lambda(v)$.
If $k \geq 3$, we claim that there is a clique of size $k-1$ in the common neighbourhood of $u$ and $v$. Indeed, since $\delta(G) \geq \big(\tfrac{k-1}{k}-\gamma\big)n$, we can iteratively find a clique with vertices $u_2, \ldots, u_{k-2}$ in the common neighbourhood of $u$ and $v$ and the common neighbourhood of $u,v,u_2, \ldots,u_{k-2}$ has size at least $(\tfrac{1}{k}-(k-1)\gamma))n > (\tfrac{1}{k} - \eta) n$. In particular, there is an edge $u_{k-1}u_k$ in there, completing the clique.
Let $K_u = \{u,u_2, \ldots, u_k\}$ and $K_v = \{v,u_2, \ldots, u_k\}$, and decrease the weight of $K_u$ by $1$ and increase the weight of $K_v$ by $1$.
If $k=2$, we do the following: Since $\alpha(G) < n/2$, $G$ is not bipartite and hence contains an odd cycle. Since $G$ is connected, this implies that there is a walk 
from $u$ to $v$ of even length (even number of edges). We  take a shortest such walk (in terms of edges) and note that every edge is traversed at most twice by this walk. We  decrease the weight of the edge at $u$ and then alternate increasing and decreasing the weight of the edges along the walk. Note that in both cases the total weight at $u$ decreases by $1$ and the total weight at $v$ increases by $1$, and the total weight at any other vertex remains unchanged.

Note that $\sum_{v \in V(G)} \card{ \lambda(v) - \sum_{K \in K_k(v,G)} \omega(K)}$ decreases by $2$ in every step. So this process finishes after at most $n^{k}$ steps.
Clearly, at this time, we have $\sum_{K \in K_k(v,G)} \omega(K) = \lambda(v)$ for all $v \in V(G)$ and $\omega(K) \geq \omega''(K) - 2n^k \geq 0$ for all $K \in K_k(G)$, completing the proof.
\end{proof}

\section{Triangle matchings} \label{sec:triangle matchings}
In this section,  we detail some probabilistic lemmas which allow us to find a \emph{triangle matching}, that is,  a collection of vertex-disjoint triangles, in various settings. These will be useful in proving \cref{thm:main} in \cref{sec:reduction}. Recall that the \emph{size} of a triangle matching is the number of triangles it contains and    we write~$V(\cT)$ for the set of vertices covered by a triangle matching~$\cT$.
The first lemma allows us to find a triangle matching in~$G_p$ if~$G$ contains many triangles. We refer the reader to the \cref{sec:notation} for any notational conventions (for example, the definition of~$G[X_1,X_2,X_3]$).

\begin{lemma}\label{lem:greedy-triangles}
For all~$\mu>0$ there exists~$C>0$ such that the following holds. 
Let~$k,n\in \N$,~$p \geq Cn^{-2/3}$ and let~$G$ be an~$n$-vertex graph.
\begin{enumerate}
\item[(i)]  Assume that for every set~$X \subseteq V(G)$ with~$|X| \geq 3k$,~$G[X]$ contains at least~$\mu n^3$ triangles. Then, whp,~$G_p$ contains a triangle matching of size at least~$\tfrac n3 - k$.
\item[(ii)] Assume that~$n_0\geq k$ and~$V(G) = V_1 \cup V_2 \cup V_3$ is a partition into sets of size at least~$n_0$ so that for every~$X_i \subseteq V_i$ with~$|X_i| \geq k$ for all~$i \in [3]$,~$G[X_1, X_2, X_3]$ contains at least~$\mu n^3$ triangles.
Then, whp,~$G_p$ contains triangle matching of size at least~$n_0 - k$.
\end{enumerate}
\end{lemma}

\begin{proof}
Let~$\mu>0$ and set~$C=50\mu^{-2}$. Let~$p,k,n,G$ be given as in the statement.
We will deduce the lemma from the following claim.
\begin{claim}
The following holds whp for all~$X \subseteq V(G)$. If~$G[X]$ contains at least~$t\ge\mu n^3$ copies of~$K_3$, then the number of triangles in~$G_p[X]$ is at least~$\tfrac12p^3t$.
\end{claim}

\begin{claimproof}
 This is a straightforward application of Janson's inequality (\Cref{lem:janson}) and the union bound. Note that the total number of choices of~$X$ is at most~$2^{n}$. Fix one such choice. The expected number of triangles in~$G_p[X]$ is~$p^3t\ge\mu p^3n^3$, and we have~$\bar{\Delta}\le 2  \max(p^5n^4,p^3n^3)$. Hence Janson's inequality tells us that the probability of having less than~$\tfrac12p^3t$ triangles is at most
 \[\exp\Big(-\tfrac{\mu^2p^6n^6}{16\max(p^5n^4,p^3n^3)}\Big) \le \exp\Big(-\tfrac{\mu^2}{16}\min(pn^2,p^3n^3)\Big)\le\exp\big(-\tfrac{C\mu^2}{16}n\big)\]
 and by our choice of~$C$ and the union bound, the claim follows.
\end{claimproof} 

We only prove~$(i)$ as~$(ii)$ is similar. Suppose that~$\cT$ is a maximal collection of vertex-disjoint triangles with~$\card \cT < \tfrac n3 - k$. Then~$X \coloneq V(G) \setminus V(\cT)$ has size at least~$3k$ but~$G_p[X]$ does not contain a triangle. Thus, the claimed result follows from the above claim.
\end{proof}

The next lemma allows us to find triangles which cover a given small set of vertices, using edges in specified places.

\begin{lemma}\label{lem:vtxcover}
For any~$0 < \mu < \tfrac{1}{100}$, 
there exists~$C>0$ such that the following holds for every~$n \in \N$ and~$p\ge Cn^{-2/3}(\log n)^{1/3}$.
Let~$G$ be an~$n$-vertex graph, and let~$v_1,\dots,v_\ell \in V(G)$ be distinct vertices with~$\ell\le\mu^2 n$. For each~$i \in [\ell]$, let~$E_i \subseteq \tr_{v_i}(G)$ be a set of edges that form a triangle with~$v_i$ such that~$|E_i|\ge\mu n^2$. 
Moreover, suppose~$A_1, \ldots, A_t \subset V(G) \setminus \{v_1, \ldots, v_\ell\}$ are disjoint sets for some~$t \in \N$.
Then, whp, there is a triangle matching~$\cT= \{T_1,\dots,T_\ell\}$ in~$G_p$ such that for each~$i\in[\ell]$ the triangle~$T_i$ consists of~$v_i$ joined to an edge of~$E_i$ and~$\card{A_k \cap V(\cT)} \leq 12\mu |A_k| + 1$ for all~$k \in [t]$.
\end{lemma} 

\begin{proof}
Given~$0 < \mu < \tfrac{1}{100}$, we set~$C=1000\mu^{-1}$. We can assume~$p=Cn^{-2/3}(\log n)^{1/3}$, since the probability of any given collection of triangles of~$G$ appearing in~$G_p$ is monotone increasing in~$p$.

We use a careful step-by-step revealing argument and choose~$T_1, \ldots, T_\ell$ one at a time. We will call an edge~$e\in E(G)$ \emph{alive} if its random status is yet  to be revealed. Given~$k \in [t]$ and~$i \in [\ell]$, say that~$A_k$ is \emph{full} at time~$i$ if~$\card{A_k \cap V(\{T_1, \ldots, T_{i-1}\})} \geq 12\mu |A_k|$. Let~$X_i$ be the union of the sets~$A_k$ that are  
full at time~$i$.
For each step~$i\in[\ell]$ in succession, we will reveal certain edges of~$G_p$ and then choose a triangle~$T_i$ among the edges revealed. Specifically, we first reveal the random status of all edges in~$G$ adjacent to~$v_i$, which do not go to~$v_1,\dots,v_\ell$,~$X_i$ or a vertex of~$T_1,\dots,T_{i-1}$. Let the edges amongst these that appear in~$G_p$ be denoted by~$S_i$. We then reveal all \emph{alive} edges of~$E_i$  which form a triangle with~$v_i$ using two edges of~$S_i$. 
From these edges we pick any that appears, fixing the resulting triangle~$T_i$, and move on to the next~$i$.

Observe that by definition we do not reveal any edge of~$G_p$ twice; and if we successfully choose a triangle at each step we indeed obtain the desired triangle matching. To begin with, we argue that when we come to~$v_i$, most edges of~$E_i$ are potential candidates to be in~$T_i$. 
Note that any edge of~$E_i$ which is adjacent to any~$v_j$ or~$T_j$ will not be a candidate; there are at most~$3\mu^2n$ such vertices, which are adjacent to at most~$3\mu^2n^2$ edges of~$E_i$.
Any edge adjacent to~$X_i$ is also not a candidate; we have~$\card{X_i} \leq \tfrac{3\ell}{12\mu} \leq \tfrac{\mu}{4} n$ and hence there are at most~$\tfrac{\mu}{4} n^2$ edges adjacent to~$X_i$.
We also have that any candidate edge of~$E_i$ must be alive. When we reveal edges at some~$v_j$, with probability at least~$1-n^{-2}$  by Chernoff's inequality (\Cref{thm:chernoff}),  we reveal at most~$2pn=2Cn^{1/3}(\log n)^{1/3}$ edges, and hence we reveal at most~$4C^2n^{2/3}\log^{2/3}n$ edges of~$E_i$ in this step. Since there are at most~$\mu^2n$ steps, in total we will have revealed less than~$n^{7/4}$ edges of~$E_i$ whp.
Note that any edge in~$E_i$ which has not been ruled out for reasons outlined above,  is a candidate at the beginning of step~$i$,  for forming~$T_i$ with~$v_i$. Putting this together then, we have that whp, for each~$i$ there remains  at least~$\tfrac12\mu n^2$ candidate edges of~$E_i$ at the beginning of step~$i$. We denote this set of candidate edges by~$F_i$. 

When we reveal edges at~$v_i$, for each edge of~$F_i$ we keep the edges from~$v_i$ to the endpoints of~$F_i$ with probability~$p^2$, and so the expected number of edges of~$F_i$ whose ends are both adjacent to~$v_i$ in~$G_p$ is~$p^2|F_i|\ge\tfrac12p^2\mu n^2$.
Now we want to apply Janson's inequality (\Cref{lem:janson}):
We have~$\bar{\Delta}\le p^3n^3$, which is tiny compared to the square of the expectation, so by Janson's inequality with probability at least~$1-n^{-2}$, at least~$\tfrac14p^2\mu n^2$  edges of~$F_i$ are revealed to lie in~$N_{G_p}(v_i)$. We now reveal which of these edges survive in~$G_p$; by Chernoff's inequality (\cref{thm:chernoff}) and by our choice of~$C$, with probability at least~$1-n^{-2}$, at least~$\tfrac18p^3\mu n^2$ of these edges survive in~$G_p$, and in particular~$T_i$ exists.

Taking a union bound, the probability of failure at any step is~$o(1)$.
\end{proof}

The next lemma allows us to find a reasonably large triangle matching using a possibly sparse set of edges,  each of which extends to many triangles; we will use this to deal with nearly independent sets which have size larger than~$\tfrac13n$.
Recall that we denote by~$\deg_G(e;X)$ the size of the common neighbourhood of the endpoints of an edge~$e$ inside a set~$X$. Recall also that given a set of edges~$E$, we will sometimes think of~$E$ as the graph~$H_E\coloneq(V(E), E)$ where~$V(E)$ denotes the set of vertices contained in edges in~$E$. We use notation like 
$\delta(E) \coloneq \delta(H_E)$ and~$\deg_E(v) \coloneq \deg_{H_E}(v)$. Furthermore, given a set of vertices~$A\subseteq V(G)$,~$E[A]$ is used to denote the set of edges in~$E$ that are contained in~$A$, that is,~$E[A] := \{e \in E: e \subset A\}$.

\begin{lemma}\label{lem:matchcover}
For any~$0 < \mu < \tfrac{1}{1000}$ there exists~$C>0$ such that the following holds for all~$n,\delta,\delta_1,\delta_2 \in \N$, every~$n$-vertex graph~$G$ and every~$p \geq  Cn^{-2/3}(\log n)^{1/3}$.
\begin{enumerate}[(i)]
\item Let~$X_1,X_2,X_3 \subset V(G)$ be disjoint sets of size at least~$\tfrac{n}{10}$, and let~$E \subseteq  E(G[X_1])$ be a set of edges such that~$\deg_E(v) \geq \delta$ for all~$v \in X_1$ and~$\deg_G(e;X_i) \geq \mu n$ for all~$e\in E$ and~$i=2,3$. Let~$n_2,n_3\in \NN$  with~$n_2 + n_3 \leq \min (\delta, \mu^5 n)$.
Then, whp, there is a triangle matching~$\cT=\{T_1,\dots,T_{n_2+n_3}\}$  in~$G_p$ with~$n_i$ triangles  consisting of an edge~$e\in E$ together with a vertex of~$X_i$ for each~$i=2,3$.
\item Let~$X_1,X_2 \subset V(G)$ be disjoint sets of size at least~$\tfrac{n}{10}$. Let~$E_i \subseteq E(G[X_i])$ be sets of edges such that~$\deg_{E_i}(v) \geq \delta_i$ for all~$v \in X_i$ and~$\deg(e;X_{3-i}) \geq \mu n$ for all~$e\in E_i$ and~$i \in [2]$. Let~$n_i\in \NN$ with~$n_i \leq \min(\delta_i, \mu^5 n)$  for each~$i \in [2]$.
Then, whp, there is a triangle matching~$\cT=\{T_1,\dots,T_{n_1+n_2}\}$  in~$G_p$ with~$n_i$  triangles consisting of an edge~$e\in E_i$ together with a vertex of~$X_{3-i}$ for each~$i \in [2]$.
\end{enumerate}
\end{lemma}

Observe that, unlike other lemmas in this section, both cases of this lemma are very tight and we cannot even guarantee more vertex-disjoint triangles in the underlying graph~$G$. Indeed, this is the case when we have complete unbalanced bipartite graphs. If the edges~$E$ have small maximum degree however, the situation is somewhat easier as the following lemma shows and we will make use of this in the proof of \cref{lem:matchcover}.

\begin{lemma}\label{lem:matchcover-help}
For all~$\mu > 0$ there exists~$C>0$ such that the following holds for all~$n \in \N$, every~$n$-vertex graph~$G$ and every~$p \geq  Cn^{-2/3}(\log n)^{1/3}$. Suppose that~$E$ is a subset of~$E(G)$ with~$\Delta(E) \leq \mu n$ and~$\mu n \le |E|\le \mu^2n^2$. Suppose in addition that for each edge~$e\in E$ there is a given set~$X_e$ of size~$|X_e|\ge\mu n$ consisting of vertices~$v \in V(G) \setminus V(E)$ such that~$e\in \tr_v(G)$. 
Then, whp, there is a triangle matching~$T_1,\dots,T_\ell$  in~$G_p$, where each~$T_i$ consists of an edge~$e\in E$ together with a vertex of~$X_e$, such that~$\ell\ge \tfrac{|E|}{10\mu n}$.
\end{lemma}

\begin{proof}
Let~$0 < \tfrac{1}{C} \ll \mu$. We may assume that~$p=Cn^{-2/3}(\log n)^{1/3}$ and that~$n$ is large enough for the following arguments. We will deduce the lemma from the following claim.

\begin{claim}
Whp the following is true for all~$X \subset V(G)$ with~$|X| \leq \tfrac{|E|}{\mu n}$. If~$|E [V(G)\setminus X]| \geq \tfrac{|E|}{2}$ and~$|X_e \setminus X| \geq \tfrac{\mu n}{2}$ for all~$e \in E$, then there is a triangle in~$G_p[V(G) \setminus X]$ consisting of an edge~$e\in E$ together with a vertex of~$X_e$.
\end{claim}

\begin{claimproof}
This is a straightforward application of Janson's inequality and the union bound. Note that the total number of choices of~$X$ is at most~$n^{|E|/(\mu n)}$. Fix one such choice. Let~$Y$ denote the number of suitable triangles in~$G_p[V(G) \setminus X]$ and note that~$\lambda \coloneq \Exp{Y} \geq \tfrac{p^3 \mu |E|n}{4} \geq C^2 \log(n) \tfrac{|E|}{n}$. Furthermore, we have~$\bar{\Delta}\le 2 \max(p^5|E|n^2,\lambda) \leq 2 \max( \tfrac{4}{\mu}p^2n\lambda, \lambda) \leq 2 \lambda$. Hence, by Janson's inequality (see \cref{lem:janson}), the probability of having less than~$\tfrac{\lambda}{2}$ triangles is at most
\[\exp\Big(-\frac{\lambda^2}{8\bar\Delta}\Big) \leq n^{-C|E|/n}.\]
The claim now follows by taking a union bound and noting~$C \gg \tfrac{1}{\mu}$.
\end{claimproof}

Assume now the high probability event in the claim occurs and let~$T_1, \ldots, T_\ell$ be a maximal triangle matching as in  the statement of the lemma. Suppose for contradiction that~$\ell < \tfrac{|E|}{10\mu n}$ and let~$X$ be the set of vertices covered by~$T_1, \ldots, T_\ell$. We have~$|E [V(G)\setminus X]| \geq |E| - |X|\mu n \geq \tfrac{|E|}{2}$ and~$|X_e \setminus X| \geq \mu n - \tfrac{3|E|}{10\mu n} \geq \tfrac{\mu n}{2}$ for all~$e \in E$, and hence there is a suitable triangle in~$G_p[V(G) \setminus X]$ which extends the triangle matching, a contradiction.
\end{proof}

We are now ready to prove \cref{lem:matchcover}. 
\begin{proof}[Proof of \cref{lem:matchcover}]
Let~$0 < \tfrac{1}{C} \ll \mu$. We begin by proving~$(i)$. We may assume that~$\delta \leq \mu^5n$ and that~$n$ is large enough for the following arguments.

 Let~$G_1, G_2, G_3$ be independent copies of~$G_{p/3}$. Observe that~$G_1 \cup G_2 \cup G_3$ is distributed like~$G_{p'}$ for some~$p' \leq p$ and therefore it suffices to show that~$G_1 \cup G_2 \cup G_3$ contains our desired triangle matching~$\cT$  whp. In what follows we will find~$\cT$ as the disjoint union of three triangle matchings~$\cT_1,\cT_2$ and~$\cT_3$. For~$i\in[3]$, the edges of~$G_i$ will be used to find the triangles in~$\cT_i$ and we will reveal~$G_1, G_2$ and~$G_3$ at different stages of our process, making use of their independence. 

Let~$B\coloneq \{v \in X_1: \deg_E(v;X_1) \geq \mu n\}$, and let~$S \coloneq X_1 \setminus B$.  
If~$|B| \geq n_2+n_3$, let~$\cT_1 = \cT_2 = \emptyset$,~$n_2'=n_3'=0$, and move to the last  stage of the process, in which we find~$\cT_3$.
Otherwise, fix~$n_2':=\min(n_2+n_3 - |B|, n_{2})$ and~$n_3':= n_2+n_3-|B|-n'_2=\max(0,n_3-|B|)$.  In a first 
round of probability 
we find a triangle matching~$\cT_1$ of size~$n_2'$  in~$G_1$, each triangle containing an edge in~$E[S]$ and a vertex in~$X_2$. 
 This triangle matching exists whp
 due to \cref{lem:matchcover-help}. Indeed we have that~$\Delta(E[S]) \leq \mu n$ (by the definition of~$S$) and~$\deg(e;X_2) \geq \mu n$ for all~$e \in E[S]$. It remains to estimate~$|E[S]|$. For this, note that 
 \begin{align}
\nonumber |E[S]|&\geq \tfrac{1}{2}|S|(\delta-|B|)
 \\  \nonumber
 &\geq \tfrac{1}{2}\big(\tfrac{n}{10}-\delta\big)(n_2+n_3-|B|) 
 \\ \label{eq:lowerboundE[S]} &\geq \tfrac{n}{40}(n_2+n_3-|B|) \geq \mu n. \end{align}
 Furthermore, if~$|E[S]|>\mu^2n^2$ then we can shrink~$E[S]$ to some subset having size exactly~$\mu^2n^2$.  Applying \cref{lem:matchcover-help}  then gives a triangle matching of size at least~$t\geq\frac{|E[S]|}{10 \mu n}$. If~$E[S]$ was shrunk to have size~$\mu^2n^2$, then~$t\geq \tfrac{\mu}{10}n\geq n_2'$ and if not, then 
\[
 t \geq \frac{ n(n_2+n_3-|B|)}{400 \mu n} \geq n_2+n_3 - |B|\geq n_2',
\]
using~\cref{eq:lowerboundE[S]}. 
In either case we can pick a sub-triangle matching~$\cT_1$ of the desired size~$n_2'$.

We now fix~$S'=S\setminus V(\cT_1)$. 
Similarly to the previous stage, we will use~$G_2$ to find a triangle matching~$\cT_2$  of size~$n_3'$ such that each triangle contains an edge in~$E[S']$ and a vertex in~$X_3$. We still clearly have that~$\Delta(E[S']) \leq \mu n$  and~$\deg(e;X_3) \geq \mu n$ for all~$e \in E[S']$. Moreover, we have that 
\[|E[S']|\geq |E[S]|-\mu n\cdot 2n_2'\geq \big(\tfrac{n}{40}-2\mu n\big)(n_2+n_3-|B|) \geq \mu n,\]
where we used~\cref{eq:lowerboundE[S]} and the fact that~$2n_2'$ vertices of~$S$ were used in~$\cT_1$, each of which has degree at most~$\mu n$ in~$E[S]$. Therefore, as in the previous phase, \cref{lem:matchcover-help} gives the existence of at least~$n_3'$ vertex-disjoint triangles in~$G_3$, each of which contain an edge of~$E[S']$ and a vertex in~$X_3$. From this, we choose our triangle matching~$\cT_2$ of size~$n_3'$.

In our final phase we find a triangle matching~$\cT_3$ in~$G_3$ to complete~$\cT=\cT_1\cup \cT_2\cup \cT_3$ as desired. Let~$X_i''=X_i\setminus (V(\cT_1\cup \cT_2))$ for~$i\in [3]$ and  
note that~$B\subset X_1''$. Further, for~$i\in [2]$, let~$n_i'' = n_i - n_i'$ and note that each~$n_i''\geq 0$ and~$n_2''+n_3''=\min(|B|,n_2+n_3)$. Pick disjoint subsets~$B_i \subset B$ of size~$n_i''$ for each~$i=2,3$. Since~$n_2'' + n_3'' \leq n_2+n_3\leq \delta \leq \mu^5n$, it follows from \cref{lem:vtxcover}, that whp there is a triangle matching~$\cT_3$ of size~$n_2'' + n_3''$  in~$G_3[X_1''\cup X_2''\cup X_3'']$ consisting of~$n_i''$ triangles which contain an edge in~$E[X''_1]$ and one vertex in~$X''_i$, for each~$i =2,3$. Indeed, in applying \cref{lem:vtxcover}, we can fix~$t=0$ (we do not need to use the full extent of the lemma here) and for~$i=2,3$ and~$v\in B_i$, we choose a collection of at least~$\tfrac{\mu^2n^2}{4}$ edges~$f$ in~$\tr_v(G)$ such that~$|f\cap X_1''|=|f\cap X_i''|=1$.  These edges exist as \[\deg_E(v;X''_1) \geq \deg_E(v;X_1)-|V(\cT_1\cup \cT_2)|\geq \mu n-4\delta\geq \tfrac{\mu n}{2}\]
and for each edge~$e\in E[X''_1]$,~$\deg_G(e,X''_i) \geq \mu n-|V(\cT_1\cup \cT_2)|\geq \tfrac{\mu n}{2}$ for~$i=1,2$. To conclude, we have that whp all three stages of the process above succeed and we have a triangle matching~$\cT=\cT_1\cup \cT_2 \cup \cT_3$  in~$G_p$ as in~$(i)$.

Part~$(ii)$ is similar to part~$(i)$. We begin again by noting that we can assume~$\delta_i\leq \mu^5n$ for~$i=1,2$. We will again find three  triangle matchings~$\cT_1,\cT_2,\cT_3$   whose union will give us our desired triangle matching~$\cT$ and we again use three independent copies~$G_1,G_2,G_3$ of~$G_{p/3}$, finding the triangles in~$\cT_i$ using the edges of~$G_i$ for~$i\in[3]$. For convenience, let us also fix~$\mu'=\tfrac{\mu}{2}$. Now for~$i=1,2$, let~$B_i:= \{v \in X_i: \deg_{E_i}(v;X_i) \geq \mu' n\}$ and if~$|B_i|\geq n_i$, then shrink~$B_i$ to have size~$n_i$ (that is, take~$B_i$ to be a subset of~$\{v \in X_i: \deg_{E_i}(v;X_i) \geq \mu' n\}$ of size~$n_i$). Further, for~$i\in [2]$, let~$S_i:=X_i\setminus B_i$ and define~$n_i'=n_i-|B_i|$. Let us assume for now that~$n_1'\leq n_2'$.

In~$G_1$, we now find~$\cT_1$, a triangle matching of size~$n_1'$ with each  triangle containing an edge in~$E[S_1]$ and a vertex of~$S_2$. If~$n_1'=0$ there is nothing to prove here and in the case that~$n_1'\geq 1$ (and so~$|B_1|<n_1$), such a triangle matching exists whp due to \cref{lem:matchcover-help} (applied with~$\mu'$ replacing~$\mu$). Indeed, the verification of the  conditions of \cref{lem:matchcover-help} is almost identical to our proof of the existence of~$\cT_1$ in part~$(i)$, noting that $\Delta(E_1[S_1])\leq \mu n$ as we have  removed all vertices in $B_1$. One slight difference is that, for an edge~$e\in E_1[S_1]$ we cannot use all of~$N(e;X_2)$ to give the set~$X_e$ needed in \cref{lem:matchcover-help}. Indeed, we need to discount vertices in~$B_2$ but as~$|B_2|\leq n_2\leq \mu^5n$ and~$|N(e;X_2)|\geq \mu n$, we can certainly have at least~$\mu'n$ vertices in~$N(e;S_2)$.

Given that we succeed in finding~$\cT_1$, we now turn to finding~$\cT_2$ in~$G_2$. For this we define~$S_i'=S_i\setminus V(\cT_1)$ for~$i=1,2$ and we aim to find~$n_2'$ vertex-disjoint triangles, each containing an edge in~$E_2[S_2']$ and a vertex of~$S_1'$. If~$n_2'=0$, then the existence of~$\cT_2$ is immediate. For the case when~$n_2'\geq 1$, we again  appeal to \cref{lem:matchcover-help} (with~$\mu'$ replacing~$\mu$). Note that due to the fact that~$n_2'\geq 1$, we have that~$B_2$ contains all high degree vertices and so, in particular,~$\Delta(E_2[S'_2])\leq \mu' n$.  Also using this, we have that   \begin{align*}
|E[S_2']|&\geq |E[S_2]|-|V(\cT_1)\cap S_2|\mu'n \\ &\geq \tfrac{1}{2}(|X_2|-|B_2|)(\delta_2-|B_2|)-n_1'\mu' n \\ &\geq \tfrac{n}{40}n_2'-n_1'\mu' n 
  \\  &\geq n\big(\tfrac{1}{40}-\mu'\big)n_2'
  \geq \mu'n,
\end{align*}
where in the last two inequalities, we used that~$n_1'\leq n_2'$ and that we are in the case that~$n_2'\geq 1$. Finally, it is not hard to see that~$|N(e;S_1')|\geq \mu'n$ for all~$e\in E_2[S_2']$ and so the conditions of \cref{lem:matchcover-help} are indeed satisfied and whp we get our desired triangle matching~$\cT_2$. For the above, we needed that~$n_1'\leq n_2'$. In the case that~$n_2'
>n_1'$, we can run exactly the same proof except that we first find~$\cT_2$ and then find~$\cT_1$ after.

Finally, we find~$\cT_3$ in~$G_3$ by applying \cref{lem:vtxcover}. Indeed, similarly to our proof for part~$(i)$, we fix~$S_i''=S'_i\setminus V(\cT_2)$ for~$i\in [2]$ and we know that for each~$i\in [2]$ and~$v\in B_i$, we have at least~$\tfrac{\mu'^2n^2}{8}$ edges~$f\in \tr_v(G)$ such that~$|f\cap S_i''|=|f\cap S_{3-i}''|=1$.  Therefore, as~$|B_1|+|B_2|=n_1+n_2-n_1'-n_2'\leq 2\mu^5n$, \cref{lem:vtxcover} gives that whp, there exists a triangle matching~$\cT_3$ in~$G_3$,  of size~$|B_1|+|B_2|$, such that for each~$i\in [2]$ and~$v\in B_i$, there is a triangle in~$\cT_3$ containing~$v$, some vertex in~$S_i''$ and a vertex in~$S_{3-i}''$. Altogether, we have that whp, we can find all the triangle matchings~$\cT_i$ and~$\cT=\cT_1\cup \cT_2\cup \cT_3$ provides the desired triangle matching, completing the lemma. 
\end{proof}
\section{Reduction} \label{sec:reduction}
We are now in a position to prove \cref{thm:main}, assuming \cref{thm:main-super-reg}.  Our proof relies on the use of the Regularity Lemma (\cref{lem:reglem}), we refer the reader to \cref{sec:regularity} for the relevant definitions. 
Before giving the details, let us briefly sketch the approach. Given~$G$ with~$n \in 3\N$ vertices and minimum degree at least~$\tfrac23n$, we separate three cases.

Our first case is that
there is no set~$S$ of about~$\tfrac{n}{3}$ vertices such that~$G[S]$ has small maximum degree. In this case, we apply the Regularity Lemma (\cref{lem:reglem}) and observe that the~$(\eps,d)$-reduced graph~$R$ has no large independent set. By the Hajnal--Szemer\'edi Theorem  for~$K_3$-matchings (\cref{thm:HajSze-factor}), we find a large triangle matching~$\cT^*$ in~$R$, and make the corresponding pairs of clusters super-regular by removing a few vertices to obtain a subgraph~$T$ of~$G$. If~$T$ were spanning in~$G$, and the clusters were balanced, we would be done by \cref{thm:main-super-reg}.  To arrive at this scenario we need to remove a few more triangles covering the vertices outside~$T$ (which we do using \cref{lem:vtxcover}) and then further triangles to balance the clusters of~$T$ (using \cref{lem:greedy-triangles}).
For the latter we use the~$k=3$ case of \cref{thm:HSzFrac-integer} to find a fractional triangle factor which tells us where to remove triangles. This is the point where we use the fact that~$G$ has no large sparse set. We obtain the following lemma, whose proof we defer to \cref{sec:nosparse}. Note that this lemma shows that in the case that there is no large sparse set, we can reduce the minimum degree necessary slightly. 

\begin{lemma}[No large sparse set]\label{lem:nosparse}
 For every sufficiently small~$\mu>0$ there exist~$C>0$ and~$0<d\le \mu$ such that the following holds. Let~$n \in 3\N$,~$p\ge Cn^{-2/3}(\log n)^{1/3}$ and suppose~$G$ is an~$n$-vertex graph with~$\delta(G)\ge \big(\tfrac23 - \tfrac d2\big) n$ such that there is no~$S\subseteq V(G)$ of size at least~$\big(\tfrac13-2\mu\big)n$ with~$\Delta\big(G[S]\big)\le 2dn$. Then whp~$G_p$ contains a triangle factor.
\end{lemma}

Our second case is that there is a set~$S$ of about~$\tfrac{n}{3}$ vertices such that~$G[S]$ has maximum degree at most~$2dn$, but there is no second such set in~$G-S$. The idea here is that we will remove a few triangles from~$G$ in order to obtain a subgraph of~$G$ which can be partitioned into sets~$X_1, X_2$ of sizes~$|X_2| = 2 |X_1| \approx \tfrac{2n}{3}$, such that all vertices of~$X_1$ are adjacent to almost all vertices of~$X_2$ and vice versa (here \cref{lem:matchcover} will be very useful).
Note that, with this degree condition,~$X_2$ can be very close to the union of two cliques of size about~$\tfrac{n}{3}$; this leads to a `parity case' in which we have to be very careful, which is something of a complication. If we can arrange for the correct parities however, it will be easy to split~$X_1$ into two sets, each of which induces a super-regular triple with one of the `near-cliques' and apply our \cref{thm:main-super-reg}.
If we are not in the parity case, we will apply the Regularity Lemma to~$X_2$ and find an almost-spanning matching~$\cM^*$ in the reduced graph~$R$. We proceed similarly as in the previous case, making these pairs super-regular, removing `atypical' vertices and then balancing the pairs. Here, we make sure that every triangle we remove has two vertices in~$X_2$ and one in~$X_1$ to keep the right balance between the two parts. Finally we can partition~$X_1$ into smaller sets and form balanced super-regular triples with the edges of~$\cM^*$ in order to apply our \cref{thm:main-super-reg}. We obtain the following lemma, whose proof we defer to \cref{sec:onesparse}.

\begin{lemma}[One large sparse set]\label{lem:onesparse}
 For every sufficiently small~$\mu>0$, there exist~$C>0$ and~$0<\tau,d\le \mu$ such that the following holds for all~$n \in 3\N$ and~$p\ge Cn^{-2/3}(\log n)^{1/3}$.
 Suppose~$G$ is an~$n$-vertex graph with~$\delta(G)\ge\tfrac23n$, and suppose~$S$ is a subset of~$V(G)$ with~$|S|\ge\big(\tfrac13-\tau)n$ and~$\Delta(G[S])\le\tau n$. Suppose further that there is no~$S'\subseteq V(G)\setminus S$ of size at least~$\big(\tfrac13-2\mu\big)n$ with~$\Delta\big(G[S']\big)\le 2dn$. Then whp~$G_p$ contains a triangle factor.
\end{lemma}

Our third and final case is that there are two vertex-disjoint sets~$S_1,S_2$ each of which has  size about~$\tfrac{n}{3}$ in~$G$ and small maximum degree. In this case~$G$ must be very close to a balanced complete tripartite graph. We start by partitioning~$V(G)$ into sets~$X_1$,~$X_2$ and~$X_3$ of size around~$\tfrac{n}{3}$,  so that~$(X_1,X_2,X_3)$ is an~$(\e,d^+,\delta)$-super-regular triple, where~$d$ is close to~$1$, but~$\delta$ can be quite small (we need~$\delta \gg \e$ in order to apply \cref{thm:main-super-reg}). We remove some carefully chosen vertex-disjoint triangles in order to balance the~$X_i$ and to remove some `atypical' vertices. This leaves us with a balanced~$(\e,d^+)$-super-regular triple for some~$d$ close to~$1$, and \cref{thm:main-super-reg} finds the required triangle factor, giving the following lemma, which is proved in \cref{sec:twosparse}.

\begin{lemma}[Two large sparse sets]\label{lem:twosparse}
There exist~$C,\tau>0$ such that the following holds for all~$n \in 3\N$ and~$p\ge Cn^{-2/3}(\log n)^{1/3}$.
Suppose~$G$ is an~$n$-vertex graph with~$\delta(G)\ge\tfrac23n$, and suppose~$S_1$ and~$S_2$ are disjoint subsets of~$V(G)$ with~$|S_i|\ge\big(\tfrac13-\tau)n$ and~$\Delta(G[S_i]) \leq \tau n$ for~$i=1,2$. Then whp~$G_p$ contains a triangle factor.
\end{lemma}

Before we give proofs of these three lemmas, we show how they imply \cref{thm:main}.

\begin{proof}[Proof of \cref{thm:main}]
Choose~$0<\mu_2\ll \tau_3\ll 1$ where~$\tau_3$ is chosen small enough to apply \cref{lem:twosparse}. Let~$\tau_2,d_2\le \mu_2$ be the constants returned by \cref{lem:onesparse} with input~$\mu_2$ and choose~$0<\mu_1\ll \tau_2,d_2$. Finally, let~$d_1\le \mu_1$ be the constant returned by \cref{lem:nosparse} with input~$\mu_1$ and choose~$0<\tfrac{1}{C}\ll d_1$. 
Let~$n \in 3\N$ and let~$p\ge Cn^{-2/3}(\log n)^{1/3}$  and suppose that~$G$ is an~$n$-vertex graph with~$\delta(G) \geq \tfrac{2n}{3}$.

If~$G$ contains no subset of size at least~$\big(\tfrac13-2\mu_1\big)n$ vertices with maximum (induced) degree at most~$2d_1n$, then by \cref{lem:nosparse},~$G_p$ contains a triangle factor whp.
We may therefore suppose~$G$ contains a subset~$S_1$ of vertices of size at least~$\big(\tfrac13-2\mu_1\big)n \geq \big(\tfrac13- \tau_2 \big)n$ with maximum degree~$\Delta(G[S_1])$ at most~$2d_1n \leq \tau_2n$.
If there is no~$S_2 \subseteq V(G)\setminus S_1$ of size at least~$\big(\tfrac13-2\mu_2\big)n$ with maximum degree~$\Delta(G[S_2])$ at most~$2d_2n$, then by \cref{lem:onesparse},~$G_p$ contains a triangle factor whp.
We can therefore suppose that~$G$ contains a subset~$S_2$ disjoint from~$S_1$ of size at least~$\big(\tfrac13-2\mu_2\big)n \geq \big(\tfrac13-\tau_3\big)n$ with maximum (induced) degree at most~$2d_2n \leq \tau_3 n$. So by \cref{lem:twosparse},~$G_p$ contains a triangle factor whp.
\end{proof}

The remainder of the section is devoted to proving the three lemmas.


\subsection{Case: No large sparse set}
\label{sec:nosparse}

In this section we prove \cref{lem:nosparse}.

\begin{proof}[Proof of \cref{lem:nosparse}]
Fix some~$0<\mu\ll 1$ and choose~$0 <  \tfrac{1}{m_0} \ll \e \ll d \ll \mu$.
Let~$M_0 \geq m_0$ be returned by \cref{lem:reglem} with input~$m_0,\eps$ and fix~$\gamma=\tfrac23-\tfrac d2$ and~$0<\tfrac{1}{C}\ll \tfrac{1}{M_0}$.  Assume also that~$n \gg M_0$.
Let~$p$ and~$G$ be as in the statement
 and let~$G_1, G_2, G_3$ be independent copies of~$G_{p/3}$; we will show that~$G_1 \cup G_2 \cup G_3$ satisfies the desired properties whp.

We apply \cref{lem:reglem} to~$G$, and obtain an~$(\eps,d)$-reduced graph~$R$ on~$m$ vertices with~$m_0 \leq m \leq M_0$ and minimum degree at least~$\big(\tfrac23- \tfrac d2 - d -2\eps\big)m\ge\big(\tfrac23-2d\big)m$.  Recall that we identify the vertex set of~$R$ as~$[m]$ with each~$i\in [m]$ corresponding to a cluster~$V_i$ in the~$\eps$-regular partition of~$V(G)$. 

\begin{claim}
We have~$\alpha(R) < \big(\tfrac13-\mu\big)m$.
\end{claim}

\begin{claimproof}
Suppose for  a contradiction that~$R$ contains an independent set~$I$ of size~$\big(\tfrac13-\mu\big)m$. Now call an index~$i\in I$ \emph{bad} if there are more than~$\sqrt{\e}m$ indices~$j\in[m]\setminus \{i\}$ such that~$(V_{i},V_j)$ is \emph{not}~$\e$-regular. Due to the fact that the~$V_i$ form an~$\e$-regular partition, we have  that there are at most~$2\sqrt{\e}m$ bad indices. Let~$I'$ be the set obtained from~$I$ after removing bad indices and so~$|I'|\ge \big(\tfrac13-\tfrac{3\mu}{2}\big)m$.   Now in~$\bigcup_{i\in I'}V_i$ there must exist at least~$\tfrac\mu4 n$ vertices, each of whose degree into~$\bigcup_{i\in I'}V_i$ exceeds~$2dn$, otherwise removing all such vertices would leave a set~$S$ whose existence is forbidden in the lemma statement. By averaging, there is some~$i^*\in I'$ such that~$\tfrac\mu4 |V_{i^*}|$ of these vertices  are in~$V_{i^*}$. Let~$U_{i^*}\subseteq V_{i^*}$ be this subset of high degree vertices.
Now vertices of~$V_{i^*}$ can have at most~$|V_{i^*}|$ neighbours in~$V_{i^*}$, and at most~$\sqrt{\eps} m \cdot\tfrac{n}{m}\leq \sqrt{\e} n$ neighbours in sets~$V_j$ such that~$j\in I$ and~$(V_{i^*},V_j)$ is not~$\eps$-regular (as~$i^*\in I'$). So the vertices of~$U_{i^*}$ all have at least~$\tfrac{3d}{2}n$ neighbours in total in sets~$V_j$ such that~$j\in I$,~$j\neq i^*$ and~$(V_{i^*},V_j)$ is~$\eps$-regular. By averaging, there is one of these sets~$V_j$ such that the density between~$U_{i^*}$ and~$V_j$ exceeds~$\tfrac32d$.
But since~$I$ is independent, the fact that $(V_{i^*},V_j)$ is $\eps$-regular implies that it has density less than~$d$. This is a contradiction.
\end{claimproof}

We apply the Hajnal--Szemer\'edi Theorem for~$K_3$-matchings (\cref{thm:HajSze-factor}) to~$R$, which gives us a triangle matching~$\cT^*$  in~$R$ covering at least~$(1-13d)m$ vertices. We denote by~$T^* \coloneq V(\cT^*)$ the set of indices in triangles of~$\cT^*$.
By \cref{lem:super-reg}, there are~$V'_i \subset V_i$ for each~$i \in T^*$ such that~$|V_i'| = \lceil (1- 3\eps)|V_i| \rceil$ and, for every triangle~$ijk \in \cT^*$, the triple~$(V_i',V_j',V_k')$ is~$(2\eps,(d-\e)^+,d-3\e)$-super-regular.
Let~$T = \bigcup_{i \in T^*} V_i'$ be the set of vertices in~$G$ which are in a cluster~$V_i'$ corresponding to a triangle of~$\cT^*$. Let~$X = V(G) \setminus T$. Observe that~$|X| \leq \eps n+ 13 dn+3\eps n\le 14 dn$.
Let~$W \subset T$ be a set such that
\begin{enumerate}[(i)]
\item~$\card{W \cap V_i'} =  \big(\tfrac12 \pm \tfrac1{20}\big) \tfrac{n}{m}$ for each~$i \in T^*$,
\item~$\deg_G(v;W) \geq \tfrac35 |W|$ for each~$v \in V(G)$, and
\item we have that~$\deg_G(v;V_i'\cap W) = \big(\tfrac12 \pm \tfrac14\big) \deg_G(v;V_i')$ for each~$i \in T^*$ and~$v \in V(G)$ with $\deg_G(v;V_i') \geq \e |V_i'|$.
\end{enumerate}
Such a set~$W$ can be found by choosing each vertex of~$T$ independently with probability~$\tfrac12$ and applying Chernoff's inequality (\cref{thm:chernoff}) and a union bound.

We now start building our triangle factor by covering $X$. For this, we will not use vertices that belong to~$T\setminus W$ in order to maintain super-regularity properties.

\begin{claim}\label{claim:rtt-red-nosparse-cover-X}
Whp in~$G_1$, there is a triangle matching~$\cT_1 \subset K_3(G_1[W \cup X])$ so that~$X \subset V(\cT_1)$ and~$\card{V(\cT_1) \cap V_i'} \leq 50\sqrt{d}|V_i'|$ for all~$i \in T^*$.
\end{claim}

\begin{claimproof}
Let~$\tilde \mu \coloneq 4\sqrt d$ and enumerate~$X = \{v_1, \ldots, v_\ell\}$, noting that~$\ell \leq \tilde \mu^2 n$. For each~$i \in [\ell]$, let~$E_i \coloneq E(G[W])\cap \tr_{v_i}(G)$. 
Note that, since~$\deg(v;W) \geq \tfrac35 |W|$ for all~$v \in V(G)$, we have~$|E_i| \geq \tilde \mu n^2$ for all~$i \in [\ell]$. Finally, let~$A_i = V_i'$ for each~$i \in T^*$.
The claim now follows readily from \cref{lem:vtxcover}.
\end{claimproof}

Let now~$V_i'' = V_i' \setminus V(\cT_1)$ for each~$i \in T^*$. We would like to apply \cref{thm:main-super-reg} to the super-regular triples~$(V_i'', V_j'', V_k'')$ for each~$ijk \in \cT^*$. However, these triples are not necessarily balanced. The next claim corrects this. 

\begin{claim}
Whp in~$G_2$, there is a triangle matching~$\cT_2 \subset K_3(G_2[W \setminus V(\cT_1)])$ so that~$|V_i'' \setminus V(\cT_2)| =  \lfloor \tfrac{9}{10} \tfrac n m \rfloor$ for all~$i \in T^*$.
\end{claim}

\begin{claimproof}
The key idea in this proof is to use fractional factors to dictate how we remove triangles in order to balance the parts. More specifically, we will apply our stability theorem for the fractional Hajnal--Szemer\'edi theorem with integer weights (\cref{thm:HSzFrac-integer}), using that the reduced graph has large minimum degree and no large independent sets. In detail, 
let~$R' = R[T^*]$ and let~$\lambda: T^* \to \N$ be given by~$\lambda(i) = |V_i''| - \lfloor \tfrac{9}{10} \tfrac n m \rfloor$. Note that
$\big(\tfrac{1}{10} - 60\sqrt d\big) \tfrac n m \leq \lambda(i) \leq \lceil \tfrac{1}{10} \tfrac n m \rceil$,
and that~$\sum_{i \in T^*} \lambda(i) = n - 3 \card{\cT_1} - 3 \card{\cT^*} \lfloor \tfrac{9}{10} \tfrac n m \rfloor$ is divisible by~$3$. Also, we have that~$\delta(R')\ge \delta(R)-13dm\ge \big(\tfrac{2}{3}-15d\big)|R'|$ and~$\alpha(R')\le \big(\tfrac{1}{3}-\mu\big)m\le \big(\tfrac{1}{3}-\tfrac{\mu}{2}\big)|R'|$. 
Hence, by \cref{thm:HSzFrac-integer} (and the fact that~$d\ll \mu$), there is a weight function~$\omega: K_3(R') \to \N$ such that for each~$i\in T^*$ we have~$\sum_{K \in K_3(R',i)} \omega(K)=\lambda(i)$.
We claim that we can remove~$\omega(ijk)$ triangles from~$G_2[V''_i\cap W,V''_j\cap W,V''_k\cap W]$ for each triangle~$ijk$ of~$R'$, making sure that all our choices are vertex-disjoint. Indeed, observe that for any choice of~$X_h\subset V''_h\cap W$  such that~$|X_h|\geq d \tfrac nm$ for~$h\in \{i,j,k\}$, we have~$|K_3(G[X_i,X_j,X_k])|\geq \tfrac{d^6}{10m^3}n^3$ due to \cref{lem:super-reg-triangle-count} and the~$(\eps,d^+)$ regularity of~$G[V_i,V_j,V_k]$.  
Furthermore, observe that~$|V''_i\cap W| \geq \tfrac25 \cdot \tfrac{n}{m}$ for each~$i\in T^*$. Hence, \cref{lem:greedy-triangles}~$(ii)$ implies that whp there are~$\tfrac7{20} \cdot \tfrac{n}{m}>3\cdot \lceil \tfrac{1}{10} \tfrac n m \rceil$ vertex-disjoint triangles in~$G_2[V''_i\cap W,V''_j\cap W,V''_k\cap W]$ for each~$ijk \in K_3(R')$, so we can select the desired number of triangles for each~$K \in K_3(R')$ one at a time.
\end{claimproof}

Let now~$V_i''' = V_i'' \setminus V(\cT_2)$ for all~$i \in T^*$ and observe that we have covered all vertices except for those in~$\bigcup_{i \in T^*} V_i'''$. We claim that~$(V_i''',V_j''',V_k''')$ is~$(5\e, (d/2)^+,d/8)$-super-regular for all~$ijk \in \cT^*$. Indeed, this follows from the Slicing Lemma (\cref{lem:reg-slicing}), and from~$\deg(v;V_j''') \geq \deg(v;V_j' \setminus W) \geq \frac14 \deg_G(v;V_j') \geq \frac{d}{8} |V_j'|$ for all~$v\in V_i$ and the analogous inequalities for other pairs.
Finally, we apply \cref{thm:main-super-reg} to each of these triples individually in~$G_3$ to obtain (whp)\ a triangle matching~$\cT_3$  covering exactly~$\bigcup_{i \in T^*} V_i'''$.
\end{proof}

\subsection{Case: Two large sparse sets}
\label{sec:twosparse}

Next, we deal with the case when~$G$ has two large sparse sets; i.e.\ it looks similar to the extremal complete tripartite graph. This is the easiest case; we will not need the regularity lemma.

\begin{proof}[Proof of Lemma~\ref{lem:twosparse}]
Choose~$0 < \tfrac{1}{C}\ll\tau \ll \rho \ll \tfrac{1}{1000}$. Let~$n \in 3\N$ be large enough for the following arguments and let~$p\ge Cn^{-2/3}(\log n)^{1/3}$. Let~$G$ and sparse sets~$S_1$ and~$S_2$ be given as in the statement. Let~$G_1, G_2, G_3$ be independent copies of~$G_{p/3}$. We will find a triangle factor in~$G_1 \cup G_2 \cup G_3$.

\begin{claim}\label{claim:rtt-structure-2sparse}
There is a partition~$V(G) = X_1 \cup X_2 \cup X_3$ such that
\begin{enumerate}[(i)]
\item~$|X_i| = \big(\tfrac13 \pm \rho^6\big)n$ for all~$i \in [3]$,
\item~$\deg(v;X_j) \geq \rho n$ for all~$i \not = j \in [3]$ and~$v \in X_i$,
\item~$d(X_i,X_j) \geq 1 - \rho^6$ for all~$1 \leq i < j \leq 3$,
\item For each~$i \in [3]$, if~$|X_i| \geq \tfrac{n}{3}$, then~$\deg(v;X_j) \geq |X_j| - 4\rho n$ for all~$v \in X_i$ and~$j \in [3] \setminus \{i\}$.
\end{enumerate}
\end{claim}

\begin{claimproof}
For~$i \in [2]$,
let~$Z_i = \{v \in V(G) \setminus (S_1 \cup S_2): \deg(v;S_i) \leq \rho n \}$.
Let~$U_i = S_i \cup Z_i$ for~$i \in [2]$ and~$U_3 = \big\{v \in V(G): \deg(v;S_i) \geq \big(\tfrac13 - 2 \rho\big)n \text{ for each } i \in [2]\big\}$.
Note that, since~$\delta(G) \geq \tfrac23 n$,~$Z_1$ and~$Z_2$ are disjoint and hence~$U_1$ and~$U_2$ are disjoint as well. Furthermore, by definition,~$U_3$ is disjoint from~$U_1$ and~$U_2$.
Let~$Z' \coloneq V(G) \setminus (U_1 \cup U_2 \cup U_3)$ be the set of remaining vertices. Partition~$Z' = Z_1' \cup Z_2' \cup Z_3'$ so that~$Z_i' = \emptyset$ if~$|U_i| \geq \tfrac n 3$ and~$|U_i|+|Z_i'| \leq \tfrac{n}{3}$ otherwise. Finally, let~$X_i = U_i \cup Z_i'$ for all~$i \in [3]$. Note that~$V(G) = X_1 \cup X_2 \cup X_3$ is indeed a partition.

We will first show that the sets~$Z_1, Z_2$ and~$Z'$ are small. Let~$i \in [2]$. Since~$|S_i|\ge\big(\tfrac13-\tau\big)n$, each vertex of~$S_i$ has at least~$\big(\tfrac13-2\tau\big)n$ non-neighbours in~$S_i$, and so at most~$2\tau n$ non-neighbours outside~$S_i$.
Therefore, the total number of non-edges between~$S_i$ and~$V(G)\setminus S_i$ is at most~$\tau n^2$ (using here that we certainly have~$|S_i|\le \tfrac{n}{2}$ for~$i=1,2$). Since every~$v \in Z_i$ has at least~$\tfrac n4$ non-neighbours in~$S_i$, this implies~$\card{Z_i} \leq 4\tau n$.
Moreover, the number of non-edges between~$U_1 \cup U_2$ and~$Z'$ is at most~$2\tau n^2 + (|Z_1|+|Z_2|)n \leq 10\tau  n^2$.
Observe that every~$v \in Z'$ has at least~$\rho n$ non-neighbours in~$U_1 \cup U_2$ (otherwise it would be in~$U_3$), and therefore~$|Z'| \leq \rho^8 n$, by our choice of~$\tau$. 
We now show that this implies condition~$(i)$. Indeed, we have that~$|S_1|,|S_2|=\big(\tfrac13\pm \tau\big)n$ where the lower bounds are directly from our assumption and the upper bounds are due to the fact that every vertex in~$S_i$ has~$\big(\tfrac{2}{3}-\tau\big)n$ neighbours outside of~$S_i$ for~$i=1,2$. For each~$i$, we add at most~$(4\tau+\rho^8) n$ vertices to~$S_i$ to obtain~$X_i$ and so we have that~$|X_i|=\big(\tfrac13\pm \rho^7\big)n$ for~$i=1,2$. Finally, the bounds on~$|X_3|$ can be deduced from the fact that the~$X_i$ partition~$V(G)$. 

Furthermore, for each~$v \in Z'$, we have~$\deg(v;S_i) \geq \rho n$ since~$v \not \in Z_i$ for~$i \in [2]$, and $\deg(v;U_3) \geq \rho n$ for otherwise~$v$ would be in~$U_3$. Clearly, we also have that~$\deg(v;X_j) \geq \rho n$ for all~$i \in [2]$,~$j \in [3] \setminus \{i\}$ and~$v \in X_i$ and so~$(ii)$ holds. Moreover, we have~$\deg(v;X_i) \geq |X_i| - 2\tau n$ for all~$v \in S_1$ and~$i =2,3$ as~$v$ already has at least~$\big(\tfrac13-2\tau \big)n$ non-neighbours in~$S_1$. Since~$|Z_1 \cup Z_1'| \leq \rho^7 n$, this implies~$d(X_1,X_i) \geq 1 - \rho^6$ for~$i =2,3$. Similarly~$d(X_2,X_3) \geq 1 - \rho^6$.

Finally, let~$i,j \in [3]$ be distinct. If~$|X_{i}| \geq \tfrac n3$, then~$X_{i} \cap Z' = \emptyset$ by construction. Now if~$i=1$ or~$i=2$,  then it is clear that ~$\deg(v;X_{j}) \geq |X_{j}| - 4\rho n$ for all~$v \in X_{i}$ as~$v$ as~$\deg(v;X_i)\le 2\rho n$ and so~$v$ already has many non-neighbours in~$X_i$ (considering the size of~$X_i$ given in~$(i)$). If~$i=3$,  then for any~$v\in X_i$, we have that~$\deg(v;X_{j}) \geq \deg(v;S_{j}) \geq \big(\tfrac{1}{3}-2\rho \big)n \ge |X_{j}| - 4\rho n$. This establishes~$(iv)$.
\end{claimproof}

We now perform a stage of removing some vertex-disjoint triangles in order to obtain a balanced tripartite graph.

\begin{claim}
Whp in~$G_1$, there is triangle matching~$\cT_1 \subset K_3(G_1)$ so that $\card{X_1 \setminus V(\cT_1)} =  \card{X_2 \setminus V(\cT_1)} = \card{X_3 \setminus V(\cT_1)} \geq (\frac13 - \rho^6)n$.
\end{claim}

\begin{claimproof}
If all three sets~$X_1, X_2, X_3$ have size exactly~$\tfrac{n}{3}$, we are done. Otherwise, one or two of these sets has size exceeding~$\tfrac{n}{3}$.

\underline{Case 1.} Assume first that only one set exceeds~$\tfrac{n}{3}$ in size and, without loss of generality, this set is~$X_1$. Let~$n_2 \coloneq \tfrac{n}{3} - |X_3|$ and~$n_3 \coloneq \tfrac{n}{3} - |X_2|$, and let~$E = E(G[X_1])$. Observe that~$\delta(E) \geq |X_1|-\tfrac n3 = n_2 + n_3$. Furthermore, we have~$\deg(e;X_i) \geq |X_i| - 10 \rho n \geq \tfrac n4$ for both~$i=2,3$. Therefore, by \cref{lem:matchcover}~$(i)$, there is a triangle matching~$\cT_1$ of size~$n_2 + n_3$  in~$G_1$ such that the  triangles in~$\cT_1$  all have two vertices in~$X_1$,~$n_2$ of them have their third vertex in~$X_2$, and~$n_3$ of them have their third vertex in~$X_3$. We then have~$\card{X_1 \setminus V(\cT_1)} = \card{X_2 \setminus V(\cT_1)} = \card{X_3 \setminus V(\cT_1)} = \tfrac{2n}{3} - |X_1| \geq (\frac13 - \rho^6)n$, as claimed, by our definitions of~$n_2$ and~$n_3$.

\underline{Case 2.} Assume now that there are two sets (say~$X_1$ and~$X_2$) exceeding~$\tfrac n 3$ in size. For~$i \in [2]$, let~$n_i \coloneq |X_i| - \tfrac{n}{3}$ and~$E_i = E(G[X_i])$. Observe that, for~$i \in [2]$,~$\delta(E_i) \geq n_i$ and~$\deg(e;X_{3-i}) \geq |X_{3-i}| - 10 \rho n \geq \tfrac n4$ for all~$e \in E_i$.
Therefore, by \cref{lem:matchcover}~$(ii)$, there is a triangle matching~$\cT_1$ of size~$n_1 + n_2$ in~$G_1$, with~$n_1$ triangles having two vertices in~$X_1$ and one in~$X_2$, and~$n_2$ triangles having two vertices in~$X_2$ and one in~$X_1$. Therefore, we have~$\card{X_1 \setminus V(\cT_1)} = \card{X_2 \setminus V(\cT_1)} = \card{X_3 \setminus V(\cT_1)} = \card{X_3} \geq (\frac13 - \rho^6)n$, as claimed.
\end{claimproof}

Let now~$X_i' = X_i \setminus V(\cT_1)$ and observe that~$|X'_1| = |X'_2| = |X'_3|$. Define
\[
Y_i' \coloneq \left\{v \in X_i': \deg(v;X_j') \leq \big(1-\tfrac{\rho}{2}\big)|X_j'| \text{ for some } j \in [3] \setminus \{i\}\right\}.
\]
Since~$d(X_i',X_j') \geq 1- 4\rho^6$ for all~$1 \leq i < j \leq 3$, we have~$\card{Y_i'} \leq 4\rho^5 n$ for each~$i \in [3]$.
Furthermore, for each~$i\in [3]$ and vertex~$v \in Y_i'$ there are at least~$\frac{1}{8} \rho^2n^2$ triangles of~$G$ containing~$v$ and one vertex in each~$X_j'\setminus Y_j'$ for~$j\in[3]\setminus \{i\}$. Indeed, we have that \[\deg(v;X_j'\setminus Y_j')\ge \deg(v;X_j)-2\card{V(\cT_1)}-\card{Y_j'}\ge \tfrac{3\rho}{4}n,\]
for each~$j\in[3]\setminus\{i\}=:\{j_1,j_2\}$. Due to the defining condition of the~$Y'_j$, we then have that for each~$ x\in N(v;X'_{j_1}\setminus Y'_{j_1})$, we have that~$\deg(v,x;X'_{j_2}\setminus Y'_{j_2})\geq \tfrac{\rho}{4}n$. This implies the claimed lower bound on the number of triangles containing~$v\in Y'_i$.  

By applying \cref{lem:vtxcover} (with~$t = 0$), whp in~$G_2$, we can find a triangle matching~$\cT_2 \subset K_3(G_2)$ with each triangle using one vertex from each part and such that~$Y_1' \cup Y_2' \cup Y_3' \subset V(\cT_2) \subset X_1' \cup X_2' \cup X_3'$ and~$|V(\cT_2)| \leq 3(|Y_1'| + |Y_2'| + |Y_3'|) \leq \rho^4 n$.

Let now~$X''_i \coloneq X'_i\setminus V(\cT_2)$ for each~$i \in [3]$ and observe that~$|X_1''| = |X_2''| = |X_3''| \geq (\frac13 - 2\rho^4) n$. Furthermore,~$(X_1'',X_2'',X_3'')$ is~$(\sqrt{\rho},(1-\rho)^+)$-super-regular by \cref{lem:very-dense-implies-regular}. Hence, by \cref{thm:main-super-reg}, whp there is a triangle matching~$\cT_3$ in~$G_3$ covering the~$X_i''$. Together with~$\cT_1$ and~$\cT_2$ this gives a full triangle factor in~$G_p$.
\end{proof}

\subsection{Case: One large sparse set}
\label{sec:onesparse}

Finally, we deal with the second case sketched in the discussion at the beginning of \cref{sec:reduction}, when there is one large sparse set but not a further disjoint one. We will use several of the ideas from the previous two lemmas, and so will abbreviate the details in places.

\begin{proof}[Proof of Lemma~\ref{lem:onesparse}]

Fix some~$0<\mu\ll 1$ and choose~$0 <   \tfrac{1}{m_0} \ll \e \ll d \ll \mu$.
Let~$M_0 \geq m_0$ be returned by \cref{lem:reglem} with input~$m_0,\eps$ and fix   
$0<\tfrac{1}{C}\ll \tau\ll \rho \ll \tfrac{1}{M_0}$.
Assume that~$n \in 3\N$ is large enough for the following arguments.
Let~$p$,~$G$ and~$S$ be as in the statement of the lemma and let~$G_1, \ldots, G_5$ be independent copies of~$G_{p/5}$. We will show that~$G_1 \cup \ldots \cup G_5$ contains a triangle factor whp.

We begin with a claim that gives us a lot of structure. For~$\eta>0$ we will call a set~$X \subseteq V(G)$~\emph{$\eta$-strongly connected} if
$\noe(X',X \setminus X') \leq \tfrac{|X|^2}{4} - \eta n^2$ for all~$X' \subseteq X$,
where we denote by~$\noe(X,Y) = |X||Y| - e(X,Y)$ the number of non-edges between~$X$ and~$Y$.
(This definition might appear somewhat strange now but will assure that the reduced graph in this proof is connected.)
Furthermore, we say that~$X$ is~\emph{$\eta$-close to complete} if~$e(G[X]) \geq \big(\tfrac12 -\eta \big)|X|^2$ and~$\deg(v;X)\ge \tfrac{1}{10}|X|$ for all~$v\in X$.

\begin{claim}\label{claim:one-sparse-structure}
Whp there is a triangle matching~$\cT_1$ in~$G_1 \cup G_2$ and disjoint sets~$X_1, X_2 \subset V(G)$ so that
\begin{enumerate}[(i)]
\item~$X_1 \cup X_2 = V(G) \setminus V(\cT_1)$ and~$|X_1| = \tfrac{|X_2|}{2} = \big(\tfrac13 \pm \rho\big) n$,
\item~$\deg(v;X_{3-i}) \geq (1-4\rho)|X_{3-i}|$ for all~$i \in [2]$ and~$v \in X_i$,
\item~$X_2$ is~$8d$-strongly connected or there is a partition~$X_2 = X_{2,1} \cup X_{2,2}$ so that, for each~$j\in[2]$, we have that~$|X_{2,j}| \geq \frac n4$ is even and~$X_{2,j}$ is~$200d$-close to complete.
\end{enumerate}
\end{claim}

\begin{claimproof}
Let~$Y_1 = \{v \in V(G) \setminus S: \deg(v;S) \leq \rho n \}$.
Let~$U_1 = S \cup Y_1$ and~$U_2 = V(G) \setminus U_1$.
 With a similar (and simpler)  proof to that  of  \cref{claim:rtt-structure-2sparse}, one can show that 
\begin{enumerate}[(P1), font=\normalfont]
\item~$\deg(v;U_2) \geq  |U_2| - 2\rho n$ for all~$v \in U_1$ and~$\deg(v;U_1) \geq \rho n$ for all~$v \in U_2$,\label{it:rtt-P1}
\item~$|U_1| = (\frac13 \pm \rho^6)n$ and~$|U_2|= (\frac23 \pm \rho^6)n$, and\label{it:rtt-P2}
\item~$d(U_1,U_2) \geq 1 - \rho^6$.\label{it:rtt-P3}
\end{enumerate}
Let~$\sigma = 10d$ and let~$U_2 = U_{2,1} \cup U_{2,2}$ be the partition of~$U_2$ which maximises~$\noe(U_{2,1} , U_{2,2})$. Throughout this proof, we will have to distinguish between two cases: either~$U_2$ is~$\sigma$-strongly-connected (this we will call the \emph{connected case} from now on) or~$\noe(U_{2,1}, U_{2,2}) \geq \tfrac{|U_2|^2}{4} - \sigma n^2$ (which we call the \emph{disconnected case}). 
Although the process is very similar for both, we will handle them separately, starting with the disconnected case.

\par\smallskip
\underline{The disconnected case.}
We claim that
\begin{enumerate}[(Q1), font=\normalfont]
\item~$|U_{2,j}| = \big(\tfrac13 \pm 2\sigma\big)n$ and~$e(U_{2,j}) \geq \tfrac12 |U_{2,j}|^2 - 2\sigma n^2$ for both~$j \in [2]$, and\label{it:rtt-Q1}
\item~$\deg(v;U_{2,j}) \geq \tfrac{n}{10}$ for any~$j\in[2]$  and~$v \in U_{2,j}$.\label{it:rtt-Q2}
\end{enumerate}
Indeed, \cref{it:rtt-Q1} follows from the case assumption and the fact that~$\delta(G) \geq \tfrac{2n}{3}$, and \cref{it:rtt-Q2} since $U_{2,1}, U_{2,2}$ are chosen to maximise non-edges in between (otherwise, moving a vertex violating \cref{it:rtt-Q2} to the other set increases the count).

In a first round of probability ($G_1$), our goal is to balance the sizes.
Assume first that~$|U_1| > \tfrac{n}{3}$. Let~$n_2 = 0$ if~$\card{U_{2,1}}$ is even and~$n_2 = 1$ otherwise, and let~$n_3 = |U_1| - \tfrac n3 - n_2\ge 0$. Let~$E = E(G[U_1])$, and observe that~$\delta(E) \geq n_2 + n_3$. Furthermore, we have~$\deg(e;U_{2,j}) \geq |U_{2,j}| - 10 \rho n \geq \tfrac n4$ for both~$j \in [2]$ by \cref{it:rtt-P1} and \cref{it:rtt-Q1}.
Therefore, by \cref{lem:matchcover}~$(i)$, whp there is a triangle matching~$\cT_1'$ of size~$n_2+n_3=|U_1| - \tfrac n3$  in~$G_1$ with each triangle having two vertices in~$U_1$ and one vertex in~$U_2$ ($n_2$ have their third vertex in~$U_{2,1}$ and~$n_3$ have their third vertex in~$U_{2,2}$). Let~$U_i' = U_i \setminus V(\cT_1')$ and~$U_{2,j}' = U_{2,j} \setminus V(\cT_1')$ for~$i,j \in [2]$. By construction, we have~$|U_2'| = 2 |U_1'| = \tfrac{4n}{3} - 2|U_1| \geq 2 \big(\tfrac13 - \rho^5\big)n$ and~$|U_{2,j}'|$ is even for both~$j \in [2]$.

Assume now that~$|U_2| > \tfrac{2n}{3}$. Observe that for each~$j \in [2]$ and~$X \subseteq U_{2,j}$ of size~$|X| \geq 
\tfrac n9$, we have~$\card{K_3(G[X])} \geq \tfrac{n^3}{1000}$ by \cref{it:rtt-Q1}. Thus, by \cref{lem:greedy-triangles}~$(i)$, there are triangle matchings of size~$\tfrac n{15}$ in each of~$G_1[U_{2,j}]$ whp for both~$j=1,2$. Thus, we can pick a triangle matching~$\cT_1'$ of exactly~$\tfrac{n}{3} - |U_1|$ from these, again taking either one or no triangle in~$U_{2,1}$ depending on its parity. By construction, we then have~$|U_2'| = 2 |U_1'| = 2|U_1| \geq 2 \big(\tfrac13 - \rho^5,\big)n$ and~$|U_{2,j}'|$ is even for both~$j \in [2]$ (where~$U_i'$ and~$U_{2,j}'$ are defined as above by removing the vertices of~$\cT'_1$ from the sets~$U_i$ and~$U_{2,j}$ ).

Finally it remains to deal with the case that~$|U_2|=2|U_1|=\tfrac{2n}{3}$. Note that as~$|U_2|$ is even in this case, we have that~$|U_{2,1}|$ and~$|U_{2,2}|$ have the same parity. If they are both  even, there is no need to take any triangles in~$\cT'_1$ and we can move to the next stage. However, if they are odd in size, we have to do a little more work. We say a triangle~$T$ is \emph{transversal} if~$|V(T)\cap U_1|=|V(T)\cap U_{2,1}|=|V(T)\cap U_{2,2}|=1$. We aim to prove the existence of a single transversal triangle in~$G_1$. In order to do this, we first show that there are at least~$\tau n^2$  transversal triangles in~$G$. Indeed, without loss of generality suppose that~$|U_{2,1}|\leq |U_{2,2}|$ and let~$Y_0\subseteq U_{2,1}$ be the set of vertices~$y$ in~$U_{2,1}$ such that~$\deg(y;U_1)\geq (1-\rho^2)|U_1|$. Due to \cref{it:rtt-P3}, we have that~$|Y_0|\geq \tfrac{n}{10}$.  Now for each vertex~$y\in Y_0$, as~$|U_{2,1}|\leq |U_{2,2}|$ we have that~$y$ has some neighbour~$z$ in~$U_{2,2}$ and due to \cref{it:rtt-P1} and the fact that~$y\in Y_0$, we have that~$\deg(y,z;U_{1})\geq \tfrac\rho2n$ and hence~$y$ is contained in at least~$\tfrac\rho2n$ transversal triangles. Considering all~$y\in Y_0$ thus gives the existence of~$\tau n^2$ transversal triangles in~$G$. A simple application of Janson's inequality (\cref{lem:janson}) gives that whp at least one of these transversal triangles survives in~$G_1$ and so taking~$\cT_1'$ to be this single triangle,~$U_i'=U_i\setminus V(\cT_1')$ for~$i\in[2]$ and~$U'_{2,j}=U_{2,j}\setminus V(\cT_1')$ for~$j\in [2]$, we again have in this case that~$|U_2'| = 2 |U_1'| = 2(|U_1|-1) \geq 2 \big(\tfrac13 - \rho^5\big)n$ and~$|U_{2,j}'|$ is even for both~$j \in [2]$.

In a second round of probability ($G_2$), we will remove `atypical' vertices in~$U_2'$. From this point onwards, we will only remove triangles with one vertex in~$U_1'$ and two vertices in~$U_{2,j}'$ for some~$j \in [2]$, thus maintaining the right balance between~$U_1'$ and~$U_2'$ and the parity of~$U_{2,1}'$ and~$U_{2,2}'$.
For~$j \in [2]$, let~$Y_{2,j} \coloneq \{v \in U_{2,j}': \deg(v;U_1') \leq  |U_1'| - \tfrac{\rho }{2}n\}$ and for each~$v \in Y_{2,j}$ let~$E_v \coloneq \{ u_1u_2: u_1 \in U_1', u_2 \in U_{2,j}'\setminus Y_{2,j}, vu_1u_2 \in K_3(G)\}$.
It follows from \cref{it:rtt-P3} (and counting non-edges between~$U_1$ and~$U_2$) that~$|Y_{2,j}| \leq 2\rho^5 n$ for both~$j \in [2]$.
Furthermore, \cref{it:rtt-P1} and \cref{it:rtt-Q2} imply that~$|E_v| \geq  (\rho - \tfrac{\rho}{2})n\cdot\big(\tfrac{1}{10} - \rho^4\big)n  \geq  \rho^2 n$ for all~$v \in Y_{2,1} \cup Y_{2,2}$.
Thus, by \cref{lem:vtxcover}, whp there is a triangle matching~$\cT_1''$ of size at most~$4 \rho^5 n$  in~$G_2[U_1' \cup U_2']$ of the desired form (each triangle having one vertex in~$U_1'$ and two vertices in~$U_{2,j}'$ for some~$j \in [2]$) such that~$Y_{2,1}\cup Y_{2,2}\subset V(\cT_1'')$.
Let~$\cT_1 = \cT_1' \cup \cT_1''$,~$X_i = U_i' \setminus V(\cT_1'')$ and~$X_{2,j} = U_{2,j}' \setminus V(\cT_1'')$ for each~$i,j \in [2]$.
These resulting sets have all the desired properties~$(i)$-$(iii)$.

\par\smallskip
\underline{The connected case.}
This case is very similar but less technical since we do not have to worry about the sets~$U_{2,1}$ and~$U_{2,2}$. We will therefore skip some details.

In a first round of probability ($G_1$), our goal is to balance the sizes.
The case~$|U_1| > \tfrac{n}{3}$ is completely analogous to the disconnected case and we find a  triangle matching~$\cT_1'$ of size~$|U_1| - \tfrac n3$  in~$G_1$ with  each triangle having  two vertices in~$U_1$ and one vertex in~$U_2$. Let~$U_i' = U_i \setminus V(\cT_1')$ for~$i \in [2]$. By construction, we have~$|U_2'| = 2 |U_1'| = \tfrac{4n}{3} - 2|U_1| \geq 2 \big(\tfrac13 - \rho^5\big)n$.

Assume now that~$|U_2| \geq \tfrac{2n}{3}$. Observe that for every set~$Z \subset U_2$ with~$|Z| \leq dn$ and every~$v \in U_2 \setminus Z$, we have~$\deg(v;U_2 \setminus Z) \geq \big(\tfrac{1}{3} - d\big) n$ and thus there are at least~$dn^2$ edges in~$N(v;U_2\setminus Z)$. Indeed due to the fact that  there is no set~$S' \subseteq  X_2$ with~$|S'|\geq \big(\tfrac{1}{3} - 2\mu\big) n$  and~$\Delta(G[S']) \leq 2dn$, we can find~$dn^2$ edges by repeatedly removing high degree vertices from~$N(v;U_2\setminus Z)$ and taking the edges adjacent to them. 
Thus there are at least~$\tfrac{d}{10} n^3$ triangles in~$G[U_2 \setminus Z]$. It follows from \cref{lem:greedy-triangles}~$(i)$ that whp there are at least~$\tfrac d3 n$ vertex-disjoint triangles in~$G_1[U_2]$.
Let~$\cT_1'$ be a triangle matching  consisting  of exactly~$\tfrac{n}{3} - |U_1|$ of these and let~$U_i'=U_i\setminus V(\cT_1')$ for~$i=1,2$. By construction, we have~$|U_2'| = 2 |U_1'| = 2|U_1| \geq 2 \big(\tfrac13 - \rho^5\big)n$.

The process of removing \emph{bad} vertices~$v$ in~$U_2'$  such that~$\deg(v;U_1')\le\card{U_1'}-\tfrac{\rho}{2}n$ is analogous to (and simpler than) the disconnected case and an application of \cref{lem:vtxcover} gives a triangle matching~$\cT_1''\subset K_3(G_2[U_1'\cup U_2'])$ containing all the bad vertices and such that defining~$\cT_1=\cT_1'\cup\cT_1''$ and~$X_i=U_i'\setminus V(\cT_1'')$  for~$i=1,2$, gives the required conditions for the claim. Here in order to verify condition~$(iii)$, we use that for any~$X\subset X_2$, we have 
\[\noe(X,X_2\setminus X)\le \noe(X,U_2\setminus X)\le \frac{|U_2|^2}{4}-10d n^2\le \frac{|X_2|^2}{4}-8dn^2,\]
using that~$|U_2|-|X_2|\le 3 \card{V(\cT_1)}\le \rho n$.
\end{claimproof}

The disconnected case now follows without much more work, as we show now. Let us first remove more atypical vertices of our near-cliques. For~$j \in [2]$, let~$Z_{2,j} \coloneq \{v \in X_{2,j}: \deg(v;X_{2,j}) \leq  |X_{2,j}| - \sqrt{d} n\}$. Observe that, since~$X_{2,j}$ is~$200d$-close to complete, by counting non-edges in~$X_{2,j}$ we have~$|Z_{2,j}| \leq 10 \sqrt{d} n$ for both~$j\in[2]$. Note that any two vertices in~$X_2$ have at least~$\tfrac{n}{4}$ common neighbours in~$X_1$ by \cref{claim:one-sparse-structure}~$(ii)$ and for~$j\in[2]$, any vertex~$v\in X_{2,j}$ has~$\deg(v;X_{2,j}\setminus Z_{2,j})\ge \tfrac{n}{50}$ by \cref{claim:one-sparse-structure}~$(iii)$ and our upper bound on~$|Z_{2,j}|$. Hence it follows from \cref{lem:vtxcover} that whp (in~$G_3$) there is a triangle matching~$\cT_2$ of size  at most~$20 \sqrt{d} n$  in~$G_3[X_1 \cup X_2]$ with each triangle having one vertex in~$X_1$ and two vertices in~$X_2$ (both of which are in the same~$X_{2,j}$) covering~$Z_{2,1} \cup Z_{2,2}$.
Let~$X_i' = X_i \setminus V(\cT_2)$ and~$X_{2,j}' = X_{2,j} \setminus V(\cT_2)$ for each~$i,j \in [2]$.
Let~$X_1' = X_{1,1}' \cup X_{1,2}'$ be a partition such that~$|X_{1,j}'| = \tfrac12 |X_{2,j}'|$ for each~$j \in [2]$ (note that here the parity of~$|X_{2,j}'|$ is important).
Now, for both~$j\in [2]$,~$X_{1,j}' \cup X_{2,j}'$ induces a~$\left(d^{1/6},\left(1-d^{1/3}\right)^+\right)$-super-regular triple (after splitting~$X_{2,j}'$ arbitrarily in two sets of equal sizes) by \cref{lem:very-dense-implies-regular}. Therefore, by \cref{thm:main-super-reg}, whp there are vertex-disjoint triangles in~$G_4$ covering the remaining vertices.

Thus, we may assume that~$X_2$ is~$8d$-strongly connected. This case is very similar to the proof of \cref{lem:nosparse}. Let~$n_i \coloneq |X_i|$ for both~$i \in [2]$ and recall that~$n_2 = 2n_1$.
We apply \cref{lem:reglem} to~$G[X_2]$ with input~$m_0,\e$  and fixing~$\gamma:=\tfrac{1}{2}-\eps$  to get an~$\e$-regular partition~$X_2 = V_0 \cup V_1 \cup \ldots \cup V_m$ for some~$m_0 \leq m \leq M_0$. Let~$R$ be the corresponding~$(\eps,d)$-reduced graph (seen as a graph on~$[m]$) and observe that we have~$\delta(R) \ge \big(\tfrac12-2d\big)m$ and, as in the proof of \cref{lem:nosparse}, we have~$\alpha(R) < \big(\tfrac12-\mu\big)m$.
It is well-known that every graph~$H$ contains a matching of size~$\min\{ \delta(H), \lfloor \tfrac{v(H)}{2} \rfloor\}$.
Indeed, if~$v(H)$ is even this is the~$k=2$ case of \cref{thm:HajSze-factor}, whilst if~$n$ is odd this can be derived from \cref{thm:HajSze-factor} by adding a vertex to~$H$ that is adjacent to all other vertices.
We conclude that~$R$ contains a matching~$\cM^*$ of size~$\big(\tfrac12-2d\big)m$; let~$R'$ be the subgraph of~$R$ induced by~$M^* \coloneq V(\cM^*)$.
Note that~$\delta(R') \geq \big(\tfrac12-6d\big)m$ and we claim that~$R'$ is connected. Indeed, if not, there is a set~$B \subset V(R')$ such that~$e(B, V(R') \setminus B) = 0$. Observe that~$|B|,|V(R')\setminus B| \geq \delta(R') \geq \big(\tfrac12 - 6d\big)m$. Let now~$X' \coloneq \bigcup_{h \in B} V_h$ and observe that~$|X'| = \big(\tfrac12 \pm 20d\big)|X_2|$.
Furthermore, we have~$e(X',X_2 \setminus X') \leq (d + 4d + 2\e )n^2$ and consequently \[\noe(X',X_2 \setminus X') \geq  |X'||X_2 \setminus X'| - 6d n^2 \ge  \left(\tfrac{|X_2|}{2}+20d|X_2|\right)\cdot \left(\tfrac{|X_2|}{2}-20d|X_2|\right) - 6d n^2 > \tfrac{|X_2|^2}{4} - 8d n^2,\] contradicting the fact that~$X_2$ is~$8d$-strongly connected.

By \cref{lem:super-reg}, there are~$V'_h \subset V_h$ for each~$h \in M^*$ such that~$|V_h'| = \lceil (1- 2\eps)|V_h| \rceil$ and, for every edge~$h\ell \in \cM^*$, the pair~$(V_h',V_\ell')$ is~$(2\eps,(d-\e)^+,d-2\e)$-super-regular.
Let~$Y = X_2 \setminus \bigcup_{h \in M^*} V_h'$ be the set of vertices in~$X_2$ which are not in a cluster~$V_h'$ corresponding to a vertex in an edge of~$\cM^*$. Observe that~$|Y| \leq 2\eps n +\eps n+ 4dn \le 5dn$, where the terms in the upper bound come from  bounding the number of vertices in sets~$V_h\setminus V_h'$ for~$h\in M^*$, the number of vertices in~$V_0$ and number of vertices in a set~$V_h$ for~$h\in [m]\setminus M^*$, respectively.
Let~$W \subset X_2 \setminus Y$ be a set such that
\begin{enumerate}
\item~$\card{W \cap V_h'} =  \big(\tfrac12 \pm \tfrac1{20}\big) \tfrac{n_2}{m}$ for each~$h \in M^*$,
\item~$\deg_G(v;W) \geq \tfrac13 |W|$ for each~$v \in X_2$, and
\item we have that~$\deg_G(v;V_h'\cap W) = \big(\tfrac12 \pm \tfrac14\big) \deg_G(v;V_h')$ for each~$h \in M^*$ and~$v \in X_2$ with $\deg_G(v;V_h') \geq \e |V_h'|$.
\end{enumerate}
Such a set~$W$ can be found by choosing each vertex of~$X_2 \setminus Y$ independently with probability~$\tfrac12$ and applying Chernoff's inequality (\cref{thm:chernoff}) and a union bound.

We will start by covering~$Y$. We will not touch vertices outside of~$W$ in order to maintain super-regularity properties.

\begin{claim}
Whp in~$G_3$, there is a triangle matching~$\cT_2 \subset K_3(G_1)$ of  size~$|Y|$  with each triangle  having two vertices in~$W \cup Y \subset X_2$ and one in~$X_1$, so that~$Y \subset V(\cT_2)$ and~$\card{V(\cT_2) \cap V_h'} \leq 50\sqrt{d}|V_h'|$ for all~$h \in M^*$.
\end{claim}

The proof is essentially identical to the proof of \cref{claim:rtt-red-nosparse-cover-X} (appealing to \cref{lem:vtxcover}) and we  omit the details.
Let now~$X_i'' = X_i \setminus V(\cT_2)$ for each~$i \in [2]$ and let~$V_h'' = V_h' \setminus V(\cT_2)$ for each~$h \in M^*$. We will now balance the sizes of the clusters~$V_h''$.

\begin{claim}
Whp in~$G_4$, there is a triangle matching~$\cT_3 \subset K_3(G_4)$ with each triangle having one vertex in~$X_1''$ and two vertices in~$W$, so that~$|V_h'' \setminus V(\cT_3)| =  \lfloor \tfrac{9}{10} \tfrac{n_2}{m} \rfloor$ for all~$h \in M^*$.
\end{claim}

\begin{claimproof}
Let~$\lambda: M^* \to \N$ be given by~$\lambda(h) = |V_h''| - \lfloor \tfrac{9}{10} \tfrac n m \rfloor$. Note that we have
$\big(\tfrac{1}{10} - 60\sqrt d\big) \tfrac{n_2}{m} \leq \lambda(h) \leq \lceil \tfrac{1}{10} \tfrac{n_2}{m} \rceil$,
and that~$\sum_{h \in M^*} \lambda(h) = n_2 - 2 \card{\cT_2} - 2 \card{\cM^*} \lfloor \tfrac{9}{10} \tfrac n m \rfloor$ is even. Note also that~$\delta(R')\geq \big(\tfrac{1}{2}-6d\big)m\geq \big(\tfrac{1}{2}-6d\big)|R'|$ and~$\alpha(R')\leq \alpha(R)\le \big(\tfrac{1}{2}-\mu\big)m\le \big(\tfrac{1}{2}-\tfrac{\mu}{2}\big)|R'|$.
Hence, by applying \cref{thm:HSzFrac-integer} to the connected graph~$R'$, there is a weight function~$\omega: E(R') \to \N$ such that for each~$h\in M^*$ we have~$\sum_{\ell \in N_{R'}(h)} \omega(h\ell) = \lambda(h)$.
We claim that we can remove~$\omega(h\ell)$ triangles from~$G_4[X_1'',V''_h \cap W, V''_\ell \cap W]$ for each edge~$h\ell$ of~$R'$, making sure that all our choices are vertex-disjoint. Indeed, let~$Y_1, \ldots, Y_m \subset X_1''$ be disjoint sets of size at least~$\tfrac25 \cdot \lceil \tfrac{n_2}{m}\rceil$ and observe that for~$i=1,2$ we have that~$\deg(v;X_{3-i}'') \geq |X_{3-i}''| - 4\rho n$ for each~$v \in X_i''$ by \cref{claim:one-sparse-structure}.
Since~$\rho \ll \tfrac{1}{m} \ll \e$, this implies that, for each~$k \in [m]$ and~$h \in M^*$, the pair~$(Y_k,V_h''\cap W)$ is~$\left(\e,\left(1-\e^2\right)^+\right)$-super-regular, appealing to \cref{lem:very-dense-implies-regular}.
It further follows from the Slicing Lemma (\cref{lem:reg-slicing}) and the choice of~$W$ that~$(V''_h \cap W, V''_\ell \cap W)$ is~$(10\e,(d/10)^+)$-super-regular for each~$h\ell \in E(R')$.
Hence the triple~$(Y_k,V''_h \cap W, V''_\ell \cap W)$ is~$(10\e,(d/10)^+)$-super-regular for each~$h\ell \in E(R')$ and~$k \in [m]$. Furthermore, we have~$|V''_h\cap W| \geq \tfrac25 \cdot \tfrac{n_2}{m}$.
Hence, an application of \cref{lem:super-reg-triangle-count} and  \cref{lem:greedy-triangles}~$(ii)$  implies that whp, there are~$\tfrac{7}{20} \cdot \tfrac{n_2}{m}$  vertex-disjoint triangles in~$G_4[Y_k,V''_h \cap W, V''_\ell \cap W]$ for each~$h\ell \in E(R')$ and~$k \in [m]$. Thus we can select the desired number of triangles for each~$e \in E(R')$ one at a time greedily as follows. When we look to find a triangle corresponding to the edge~$h\ell\in E(R')$ with~$h<\ell$ (one of~$\omega(h\ell)$ many), we take the triangle from~$G_4[Y_h,V''_h \cap W, V''_\ell \cap W]$, ensuring that it is vertex-disjoint from previous choices. From above we have that there is a collection of at least~$\tfrac{7}{20} \cdot \tfrac{n_2}{m}$ vertex-disjoint triangles in~$G_4[Y_h,V''_h \cap W, V''_\ell \cap W]$ to choose from  and at most~$3\max\{\lambda(h),\lambda(\ell)\}\leq \tfrac{3}{10}\cdot \tfrac{n_2}{m}<\tfrac{7}{20} \cdot \tfrac{n_2}{m}$ are unavailable due to their vertices having already been used in triangles in our triangle matching. This shows that the greedy process will succeed in finding a triangle matching~$\cT_3$  in~$G_4$ such that~$\cT_3$ contains~$\omega(h\ell)$ triangles in~$G_4[X_1'',V''_h \cap W, V''_\ell \cap W]$ for each edge~$h\ell$ of~$R'$.
\end{claimproof}

Let now~$X_i''' = X_i'' \setminus V(\cT_3)$ for each~$i \in [2]$ and~$V_h''' = V_h'' \setminus V(\cT_3)$ for all~$h \in M^*$ and observe that we have covered all vertices except for those in~$X_1''' \cup X_2'''$.
Since~$|X_1'''| = \tfrac12 |X_2'''|$, we can partition~$X_1''' = \bigcup_{e \in \cM^*} X_e'''$ into~$|\cM^*|$ sets of size exactly~$\lfloor \tfrac{9}{10} \tfrac{n_2}{m} \rfloor$. Observe that~$\deg(v;X_{2}''') \geq |X_{2}'''| - 4\rho n$ for each~$v \in X_1'''$ and vice versa by \cref{claim:one-sparse-structure}. Since~$\rho \ll \tfrac{1}{m} \ll \e$,  \cref{lem:very-dense-implies-regular} implies that, for each~$e \in \cM^*$ and~$h \in M^*$, the pair~$(X_e''',V_h''')$ is~$\left(\e,\left(1-\e^2\right)^+\right)$-super-regular. Furthermore, the pair~$(V_h''',V_\ell''')$ is~$(8\e, (d/8)^+)$-super-regular for each~$h\ell \in \cM^*$ by the Slicing Lemma (\cref{lem:reg-slicing}) and~$\deg(v;V_\ell''') \geq \deg(v;V_\ell' \setminus W) \geq \frac14 \deg_G(v;V_\ell') \geq \frac{d}{8} |V_\ell'|$ for all~$v \in V_h'''$ and vice versa.
Therefore,~$(X_{h\ell}''',V_h''',V_\ell''')$ is~$(8\e, (d/8)^+)$-super-regular for all~$h\ell \in \cM^*$. 
Finally, we apply \cref{thm:main-super-reg} to each of these triples individually in~$G_5$ to obtain whp a triangle matching~$\cT_4$  covering exactly~$X_1''' \cup X_2'''$. So we have that whp all of  the triangle matchings~$\cT_1,\ldots,\cT_4$ exist  and taking~$\cT=\cT_1\cup \cT_2\cup\cT_3\cup \cT_4$, we have that~$\cT$ is a triangle factor in~$G_p$ as required. 
\end{proof}

\section{Concluding Remarks} \label{sec:conclude}

\paragraph{\textbf{Clique factors.}}

Generalising the definition of a triangle factor, a  $K_k$-factor in a graph $G$ is a collection of vertex-disjoint copies of $K_k$ covering the vertex set of $G$.
We say that an $n$-vertex graph~$G$ is \emph{$k$-full} if $n\in k \N$ and $\delta(G)\geq (1-\tfrac{1}{k})n$.
Analogously to \cref{thm:CorradiHajnal}, Hajnal and Szemer\'edi~\cite{Hajnal1970} proved that for any $k\geq 2$,
any $k$-full graph contains a $K_k$-factor, and this is tight. 
Moreover, Johansson, Kahn and Vu~\cite{Johansson2008} also proved the threshold for the existence of clique factors (and indeed many other factors in graphs and hypergraphs), showing that it is 
\[p^*_k(n):=(\log n)^{2/(k^2-k)}n^{-2/k}.\]
We believe that our methods can also be used to give a robust Hajnal--Szemer\'edi
theorem. That is, there is $C>0$ such that for $p\geq
Cp^*_k(n)$ and any $k$-full graph~$G$
the random sparsification $G_p$  contains a $K_k$-factor.
We do not believe that significant new ideas
would be needed for this, but that it would be technically much more involved,
in particular in the analysis of the extremal cases in the proof of
\cref{thm:main}. Consequently, we concentrated on triangle factors here.

It would also be interesting to establish how many~$K_k$-factors are necessarily
contained in a~$k$-full graph. In particular, it would be interesting to
establish the following.

\begin{prob}\label{prob:no of factors}
  Show that there is some constant~$c=c(k)$ such that in
  any~$n$-vertex $k$-full graph the number of distinct~$K_k$-factors is at
  least~$(cn)^{n(1-1/k)}$.
\end{prob}

This would be tight up to the value of~$c$ and is
established for triangle factors in \cref{cor:triangle factor count} with an
extra $\log$-factor.

Similarly, it is interesting to consider edge-disjoint~$K_k$-factors.  By
considering a random partition of edges, \cref{thm:main} implies that
any~$n$-vertex $3$-full graph contains a family of at least~$\Omega(n^{2/3}(\log
n)^{-1/3})$ edge-disjoint triangle factors. In terms of upper bounds, by
considering triangles at a fixed vertex~$v$ with~$\deg(v)=\tfrac{2n}{3}$, it is
clear that one cannot hope for more than~$\tfrac{n}{3}$ edge-disjoint triangle
factors. In fact one can do slightly better than this by considering a
construction similar to that of Nash-Williams~\cite{nash1970hamiltonian} for the
number of edge-disjoint Hamilton cycles in Dirac graphs. Indeed, let~$n\in 3
\NN$ and~$m:=\tfrac{n}{3}$. Consider the~$n$-vertex complete tripartite graph on
vertex parts~$X\cup Y\cup Z$ such that~$|X|=m+2$ and~$|Y|=|Z|=m-1$. Let~$G$ be
the graph obtained from this tripartite graph by adding the edges of some
cycle~$C$ of length~$m+2$ on the vertices of~$X$. It is easy to check that~$G$
is~$3$-full. Moreover, any triangle factor in~$G$ must contain at least~$2$
edges of~$C$. Hence~$G$ contains at
most~$\floor{\tfrac{m+2}{2}}=\floor{\tfrac{n}{6}}+1$ edge-disjoint triangle
factors. This leaves a big gap and it would be very interesting to bring these
bounds closer together.

\begin{prob}
  Determine the number maximal number of edge-disjoint triangle factors
  guaranteed in any $n$-vertex $3$-full graph.
\end{prob}

\paragraph{\textbf{Universality.}}

For~$2\le k \in\NN$, we say an~$n$-vertex graph~$G$ is~\emph{$k$-universal} if
it contains a copy of every graph~$F$ on at most~$n$ vertices with maximum degree at most~$k$. Understanding universality in graphs seems to be a considerable challenge and many beautiful conjectures remain open.

A moment's thought may suggest  that a~$K_{k+1}$-factor is the `\emph{hardest}' maximum degree~$k$ graph to find in a graph~$G$, as a clique is the densest graph with maximum degree~$k$ and a clique factor  maximises the number of cliques. This intuition appears to hold true and has manifested in various settings. For example, we know from the theorem of Hajnal and Szemer\'edi \cite{Hajnal1970} that any~$n$-vertex graph~$G$ with~$\delta(G)\geq \big(\tfrac{k}{k+1}\big)n$ contains a~$K_{k+1}$-factor and that this is tight. Bollob\'as and Eldridge~\cite{bollobas1978packings}, and independently Catlin~\cite{Catlin}, conjectured that the same minimum degree condition actually guarantees~$k$-universality. This has been proven for~$k=2,3$~\cite{aigner1993embedding,alon19962,csaba2003proof} (and large~$n$ when~$k=3$) but remains open in general. In the case of random graphs, we know from the theorem of Johansson, Kahn and Vu \cite{Johansson2008}  that the threshold for the appearance of a~$K_{k+1}$-factor is~$p^*_{k+1}(n)$. The recent breakthrough result of Frankston, Kahn, Narayanan and Park~\cite{frankston2019thresholds} on thresholds implies that for any~$n$-vertex graph~$F$ with maximum degree~$k$, the threshold for the appearance of~$F$ in~$G(n,p)$ is at most~$p_{k+1}^*(n)$. Note that this is \emph{not} implying that~$G(n,p)$ is~$k$-universal whp when~$p=\omega(p^*_{k+1}(n))$ as we can only guarantee that some fixed~$F$ appears whp. However, the stronger version that~$p^*_{k+1}(n)$ is the threshold for~$k$-universality is believed to be true but only verified for~$k=2$~\cite{FKL16}. We remark that in general the~$2$-universality question is considerably more assailable than the general case due to the fact that every maximum degree~$2$ graph
is of a relatively simple structure, that is,
a union of disjoint cycles and paths, and thus this class of graphs is comparatively small.

We also believe that a robustness version for universality holds true as follows.

\begin{conj} \label{conj:universalrobust}
For any~$k\geq 2$, there exists a~$C>0$ such that for all~$n\in \NN$ and~$p\geq Cp^{*}_{k+1}$, the following holds.  If~$G$ is  a graph with~$\delta(G)\geq \big(\tfrac{k}{k+1}\big)n$ then whp~$G_p$ is~$k$-universal.
\end{conj}

 \cref{conj:universalrobust} is a common strengthening of the conjecture of Bollob\'as--Eldridge--Catlin and the threshold for universality and so a full solution to this conjecture at this point would be remarkable. However, establishing  the case~$k=2$ seems attainable and would be  interesting. 

\paragraph{\textbf{Powers of Hamilton cycles.}}

For~$1\le k\in \NN$,  we say an~$n$-vertex graph~$G$ contains \emph{the $k$-th power of a Hamilton cycle}  if it contains a copy of the graph obtained by taking a cycle~$C_n$ of length~$n$ and adding an edge between any pair of vertices that have distance at most~$k$ in~$C_n$. When~$k=1$, this just corresponds to~$G$ being Hamiltonian. For~$k=2$, we say~$G$ contains the \emph{square of a Hamilton cycle}. Powers of Hamilton cycles are a natural generalisation of Hamilton cycles and are well-studied. Note that for~$k\geq 2$, if~$G$ has~$n\in {(k+1)}\NN$ vertices then the existence of the $k$-th power of a Hamilton cycle in~$G$ implies the existence of a~$K_{k+1}$-factor in~$G$. Therefore any threshold for containing the $k$-th power of a Hamilton cycle must be at least as large  as the threshold for a~$K_{k+1}$-factor.
 
In the extremal setting, perhaps surprisingly, it turns out that the minimum
degree thresholds coincide. Indeed, Koml\'os, S\'ark\"ozy and
Szemer\'edi~\cite{KomSarSze1998Seymour} confirmed 
conjectures of P\'osa and Seymour for large~$n$ by showing that any~$n$-vertex
graph with~$\delta(G)\geq \big(\tfrac{k}{k+1}\big)n$ contains the $k$-th power
of a Hamilton cycle. In the probabilistic setting, the situation is different
and we see a separation between the thresholds
for~$K_{k+1}$-factors, which as discussed earlier is $p_{k+1}^*=n^{-2/(k+1)}(\log n)^{2/(k^2+k)}$, and the
thresholds for $k$-th powers of Hamilton cycles, which has been shown to
be~$n^{-1/k}$. For~$k\geq 3$, this threshold follows from a general result of
Riordan~\cite{Riordan} using an argument based on the second moment method. For
squares of Hamilton cycles, the problem of establishing the threshold took much
longer and was only recently proven by Kahn, Narayanan and
Park~\cite{kahn2021threshold}.

 In the robustness setting, the sparse blow-up
 lemma~\cite{AllenBoettcherHanKohayakawaPerson-blowup-sparse} gives that for
 all~$\eps>0$ and~$n$-vertex graphs~$G$ with~$\delta(G)\geq
 \big(\tfrac{k}{k+1}+\eps)n$, if~$p=\omega\big(\tfrac{\log n}{n}\big)^{1/2k}$,
 then~$G_p$ whp contains the $k$-th power of a Hamilton cycle. For squares of
 Hamilton cycles, this bound on~$p$ was improved to $p\geq n^{-1/2+\eps}$
 by Fischer~\cite{fischer2016robustness}. It is believable that for
 all~$k\geq 2$, an analogue of \cref{thm:main} holds in this setting and that
 the conclusions of the above results remain true without the~$\eps$ in the
 minimum degree condition and with probability values all the way down to the
 threshold $n^{-1/k}$ observed in random graphs.

 \begin{conj}
   For every~$k$ there is $C$ such that for $p\ge C n^{-1/k}$ and
   every~$n$-vertex graph~$G$ with~$\delta(G)\geq\tfrac{k}{k+1}n$ the random
   sparsification~$G_p$ whp contains the $k$-th power of a Hamilton cycle.
 \end{conj}

\bibliographystyle{abbrv}
\bibliography{bib}


\end{document}